%% file: final_version.tex
\documentclass[11pt,letterpaper]{amsart}
\input{top}

\title[]{A marking graph for finite-type Artin groups}

\usepackage[foot]{amsaddr}

\author{Kaitlin Ragosta}
\email{kaitlin.ragosta@ehu.eus}
\address[K.~Ragosta]{UPV/EHU, Barrio Sarriena, Leioa, 48940 Spain}

\date{\today}
\subjclass[2020]{20F36, 20F65 (primary); 20F67, 57K20 (secondary)}

\begin{document}

\begin{abstract}
Clean markings on surfaces were a key component in Masur and Minsky's hierarchy machinery, which proved to be a powerful tool in the study of mapping class groups. We construct a marking graph for irreducible finite-type Artin groups which is quasi-isometric to the group modulo its center, i.e., an element of $A_{\Gamma}/Z(A_{\Gamma})$ is determined up to finite error by its action on one of our markings. To construct this graph, we construct suitable collections of transverse parabolic subgroups which extend the maximal simplices of the complex of irreducible parabolic subgroups to analogues of clean markings, and we define natural analogues of elementary moves.
\end{abstract}

\maketitle
\tableofcontents

\section{Introduction}

Finite-type Artin groups first arose as fundamental groups of complex hyperplane arrangements, and have been extensively studied \cite{Tits-Artin-groups,Deligne,Brieskorn-Saito}. For example, they admit finite $K(\pi,1)$s and geodesic biautomatic structures \cite{Charney-Davis-Kpi1,Brady-Watt-Kpi1,ft-artin-groups-geodesically-biautomatic}. They are also Helly, and irreducible finite-type Artin groups modulo their centers were shown to be acylindrically hyperbolic via the construction of a WPD element \cite{Garside-Helly, third-CAL-paper}.

Nevertheless, some natural questions about finite-type Artin groups remain unanswered. Finite-type Artin groups include braid groups, which are also mapping class groups. It is generally unknown to what extent finite-type Artin groups mimic mapping class groups in structure. For example, mapping class groups of surfaces with genus at least 3 are not $\text{CAT}(0)$, while this question is open for all finite-type Artin groups except braid groups on at most seven strands, which are $\text{CAT}(0)$ \cite{Kapovich-Leeb-nonCAT0, Bridson-CAT0-reducibles, 5-strand-braid-CAT0, 6-strand-braid-CAT0, 7-strand-braid-CAT0}. On the other hand, mapping class groups admit a largest acylindrical action on the curve complex, and no similarly nice action of a finite-type Artin group modulo its center on a hyperbolic space is currently known \cite{Bowditch-acylindricity, Largest-acyl-HHG, ABO-hyp-structures}. 

Recently, Cumplido, Gebhardt, Gonz\'{a}les-Meneses, and Wiest gave a description of the curve complex for a braid group using only group theoretic properties of its parabolic subgroups \cite{C-parab-definition}. This description allowed the authors to define a ``curve complex'' for the other irreducible finite-type Artin groups, which they called the \emph{complex of irreducible parabolic subgroups}. They showed that this complex is infinite-diameter, but it is not yet known whether it is $\delta$-hyperbolic, nor whether the action of the Artin group modulo its center is acylindrical.

Inspired by the marking graph for mapping class groups, we construct a new model for finite-type Artin groups modulo their centers.

\begin{thmx}\label{thmx:main}
    An irreducible finite-type Artin group modulo its center acts geometrically on the associated marking graph $\mathcal{W}$ defined in Definition \ref{def:marking graph}.
\end{thmx}

A maximal simplex in the curve complex of a surface determines a pants decomposition of the surface. By additionally associating a suitable transverse curve to each curve in the pants decomposition, one obtains a clean marking. Masur and Minsky constructed a marking graph for surfaces whose vertex set consists of markings and whose edges correspond to elementary moves between markings; see Section \ref{sec:subsurfaces and markings} for definitions and details \cite{Masur-Minsky-1, Masur-Minsky-marking-graph-qi}. The mapping class group of a surface without boundary acts geometrically on this graph, and the marking graph is a key component of the hierarchy machinery Masur and Minsky constructed for the mapping class group \cite{Masur-Minsky-1, Masur-Minsky-marking-graph-qi}. Similarly, Durham constructed an augmented marking graph which is quasi-isometric to Teichm\"{u}ller space with the Teichm\"{u}ller metric  and used this graph to construct hierarchy machinery for Teichm\"{u}ller space \cite{Durham-augmented-marking-graph}. Teichm\"{u}ller space with the Weil-Petersson metric, on the other hand, is quasi-isometric to the pants graph, which can be viewed as a ``marking graph'' with trivial data associated to each curve \cite{Brock-Teich-WP-metric}.

A space which generalizes the marking graph has many potential applications. Notably, the marking graph was a key ingredient in the construction of tight geodesics in the curve graph \cite{Masur-Minsky-marking-graph-qi}. The question of hyperbolicity of the complex of irreducible parabolic subgroups remains open in large part because there is no clear picture of what geodesics should look like, and a marking graph for finite-type Artin groups may be a useful tool for ``guessing'' geodesics, which is a standard technique to prove hyperbolicity \cite{Bowditch-curve-graph-hyp-1, Masur-Schleimer-guessing-geods, Hamenstadt-curve-complex}. Several other potential applications are discussed in Section \ref{subsec:hhg-applications}.

\subsection{Additional results and an overview of the proof}

\subsubsection{Results on the parabolic subgroup graph}
The full definition of a finite-type Artin group is given at the beginning of Section \ref{sec:finite type artin groups}. Much of this paper requires a detailed understanding of the structure of simplices in the \emph{complex of irreducible parabolic subgroups} associated to a finite-type Artin group $A_{\Gamma}$, denoted $C_{parab}(A_{\Gamma})$ or simply $C_{parab}$ when the group is clear. 

\begin{definition}
    Let $A_{\Gamma}$ be an Artin group, and let $X$ be a subset of the Artin generators. The subgroup $\langle X \rangle$, often denoted $A_{X}$, is a \emph{standard parabolic subgroup} of $A_{\Gamma}$. If the induced subgraph of $\Gamma$ with vertex set $X$ is connected, then $A_{X}$ is called \emph{irreducible}. For any $g \in A_{\Gamma}$, a subgroup of the form $g A_{X} g^{-1}$ is called a \emph{parabolic subgroup}.
\end{definition}

Roughly speaking, $C_{parab}$ is a simplicial complex whose vertices are parabolic subgroups and where simplices correspond to subgroup inclusion or commuting subgroups \cite{C-parab-definition}. See Definition \ref{def:c-parab-def} for a precise definition.

The following proposition completely characterizes maximal $C_{parab}$-simplices whose vertex sets consist of \emph{standard} parabolic subgroups. We will see later, specifically in Proposition \ref{propx:sim-std}, that this is sufficient to understand all maximal $C_{parab}$-simplices.

\begin{propx}\label{propx:maxl-characterization}
    Let $\{A_{X_i}\}$ span a simplex $\Sigma$ in $C_{parab}$. The simplex $\Sigma$ is maximal precisely when the following hold.
            \begin{itemize}
                \item $\bigcup_{i} X_i = V(\Gamma) - \{t\}$ for some Artin generator $t$.
                \item For every $A_{X_i}$ in the simplex, the union of all $X_j \subsetneq X_i$ is equal to $X_i - \{t_i\}$ for some Artin generator $t_i$ contained in $X_i$.
            \end{itemize}
\end{propx}

Using this description, we can compute the stabilizer of a maximal \emph{standard} $C_{parab}$ simplex; see Definition \ref{almost ascending product} for the definition of an ascending product. Similarly, Proposition \ref{propx:sim-std} shows that this is sufficient to understand the stabilizers of maximal simplices in general.
\begin{thmx}\label{thmx:simplex-stab}
    If $g$ stabilizes a maximal $C_{parab}$ simplex $\{A_{X_i}\}$, then $g$ can be written as an ascending product of powers of $\Delta_{X_i}$ and $\Delta_{\Gamma}$.
\end{thmx}

As we will explain in more detail in Subsection \ref{sec:finite type artin groups}, in a braid group, the element $\Delta_{X_i}$ is roughly analogous to a half-twist about the boundary of a curve on the disk, and $\Delta_{\Gamma}$ is roughly analogous to a half-twist about the boundary of the disk. Thus elements of the form in the proposition are natural analogues of Dehn twists about some collection of base curves and possibly a permutation of the base curves. The normalizers of both parabolic subgroups of finite-type Artin groups and of arbitrary elements in finite-type Artin groups are well-understood, but little is known in general about the intersections of normalizers, so Theorem \ref{thmx:simplex-stab} is of independent interest.

\subsubsection{Simultaneous standardizability}

Two parabolic subgroups $P$ and $Q$ of an Artin group $A_{\Gamma}$ are said to be \emph{simultaneously standardizable} if there are two subsets $X, Y \subseteq V(\Gamma)$ and a single element $g \in A_{\Gamma}$ such that $P = g A_{X} g^{-1}$ and $Q = g A_{Y} g^{-1}$. In particular, conjugation by $g^{-1}$ sends both $P$ and $Q$ to standard parabolic subgroups. In the braid group, studying simultaneously standardizable subgroups is equivalent to studying curves which can be simultaneously mapped to round curves; details regarding this connection are given in Subsection \ref{sec:braid groups}.


Standardizers of parabolic subgroups are also of independent interest; they are closely related to standardizers of curve systems on the disk, which were studied in \cite{Lee-and-Lee-curve-systems}, and an algorithm to produce the minimal standardizer of a single parabolic subgroup was constructed and studied in \cite{Cumplido-minimal-standardizers}.

A key technical ingredient in this paper is the following proposition. It generalizes a result in \cite{C-parab-definition}, which proves the same result for a pair $P$ and $Q$ which are adjacent in $C_{parab}$.

\begin{propx}\label{propx:sim-std}
    Let $\{P_i\}$ span a simplex in $C_{parab}$. There is a positive element $g$ such that $g^{-1} P_i g$ is standard for every $P_i$.
\end{propx}

The above proposition says that every $C_{parab}$-simplex is conjugate to one whose vertex set consists entirely of standard parabolic subgroups. This implies that there are finitely many conjugacy classes of $C_{parab}$-simplices, and it allows us to reduce many structural questions about $C_{parab}$-simplices to questions about the finite collection of standard parabolic subgroups, as outlined in the previous subsection.

\subsubsection{Connections to hierarchical hyperbolicity and associated applications}\label{subsec:hhg-applications}

The existence of a geometric action on a graph defined similarly to the marking graph in Theorem \ref{thmx:main} is a key property of \emph{combinatorially hierarchically hyperbolic groups}; see \cite{og-chhs-paper, hhs-to-chhs}. Combinatorially hierarchically hyperbolic groups are also hierarchically hyperbolic groups, or HHGs. Braid groups, right-angled Artin groups, and extra-large type Artin groups are HHGs, leading many to ask whether all Artin groups are HHGs \cite{hyp-structures-survey, XL-CHHS}. For Artin groups which have non-dihedral finite-type parabolic subgroups, obtaining a hierarchy structure for irreducible finite-type Artin groups is almost certainly a prerequisite to resolving this question. Theorem \ref{thmx:main} provides strong evidence in favor of the existence of such a structure for irreducible finite-type Artin groups modulo their centers, and Proposition 5.14 in \cite{HHG-Z-central-extensions} together with Lemma 3.1.10 in \cite{Ragosta-thesis} shows that hierarchical hyperbolicity of irreducible finite-type Artin groups modulo their centers would imply hierarchical hyperbolicity of irreducible finite-type Artin groups themselves. Most hierarchically hyperbolic groups are asymptotically $\text{CAT}(0)$ \cite{HHG-asymp-CAT0}; as mentioned earlier in the introduction, whether finite-type Artin groups are $\text{CAT}(0)$ remains a major unsolved problem in the field with only a few special cases known \cite{5-strand-braid-CAT0, 6-strand-braid-CAT0, 7-strand-braid-CAT0}. 

\subsubsection{Structure of the paper}
Markings, which form the vertex set of the marking graph, are defined in Section \ref{sec:markings}. Their stabilizers are classified in Section \ref{sec:marking stabilizers}. The classification of marking stabilizers requires new insight into the structure of maximal simplices in $C_{parab}$ and their stabilizers; these are addressed in Section \ref{sec:simplices standardizers ribbons}. In the marking graph for surfaces, two markings are adjacent if they are connected by one of two kinds of elementary moves: a flip move or a twist move. We generalize elementary moves to the finite-type Artin setting in Section \ref{sec:elementary moves}. Generalizing flip moves requires significant technical innovation. Specifically, it requires control over the possible transverse elements associated to a particular base simplex and the definition of a new projection from these possible transverse elements to the base simplex. These steps are all technical, and they are completed in Section \ref{sec:transversals}. We conclude the proof of Theorem \ref{thmx:main} by showing that $A_{\Gamma}/Z(A_{\Gamma})$ acts on the associated marking graph by isometries and with a compact fundamental domain in Sections \ref{sec:by isometries} and \ref{sec:geometric action} respectively.\\

\noindent\textbf{Acknowledgments.} I am grateful to my Ph.D. supervisor, Carolyn Abbott, for numerous helpful suggestions, for her comments on several drafts of this paper, and for her guidance in general. I would also like to thank Mark Hagen for helpful conversations related to this project, and I would like to thank both Mark Hagen and Ruth Charney for their feedback on my thesis, which contained this work as the main component. I would like to thank Mar\'{i}a Cumplido for answering my questions about ribbons, which play an important technical role in this paper, and Abdul Zalloum for helpful suggestions at the early stages of this project. Finally, I am grateful to the anonymous referee for their careful reading of the paper and for their thoughtful feedback, which led to several improvements. This work was supported by NSF grant DMS-2139752, NSF grant DMS-2106906, Basque Government grant IT1913-26, and PID2025-171325NB-I00 grant funded by MICIU/AEI/10.13039/501100011033 and FEDER/EU.

\noindent\textbf{AI Declaration.} The author did not use AI in the preparation of this article.

\section{Background}

\subsection{Markings on surfaces}\label{sec:subsurfaces and markings}

In this subsection, we recall the definition of the marking graph of a surface. We will assume that the reader is familiar with surfaces, mapping class groups, and the curve complex. The \emph{complexity} of a surface with genus $g$ and $p$ punctures is $\xi(S) = 3g+p$. When we refer to the curve complex of a surface $S$ of complexity 4, unless otherwise specified, we mean the flag simplicial complex in which \emph{minimally} intersecting curves are adjacent. Specifically, if $S$ is a one-holed torus, then curves are adjacent if they intersect once, and if $S$ is a 4-punctured sphere, then curves are adjacent if they intersect twice.

We will restrict our attention to the case of surfaces with complexity at least 4 and without boundary. The restriction to surfaces without boundary may initially seem to be an odd choice given that we are interested in the braid group, which is the mapping class group of a punctured disk. In fact, it is acceptable for our purposes to consider the braid group modulo its center, which is a finite-index subgroup of the mapping class group of a punctured sphere. We explain why this is a sensible simplification in Subsection \ref{sec:braid groups}.

It is a well-known fact that the mapping class group of a thrice-punctured sphere is trivial. Given this, one might reasonably ask whether an element of the mapping class group of a surface $S$ is determined entirely by its action on a collection of curves which cuts $S$ into thrice punctured spheres. Such a collection is called a \emph{pants decomposition} of $S$. 

\begin{figure}[h!]
        \centering
        \includegraphics[width=8cm]{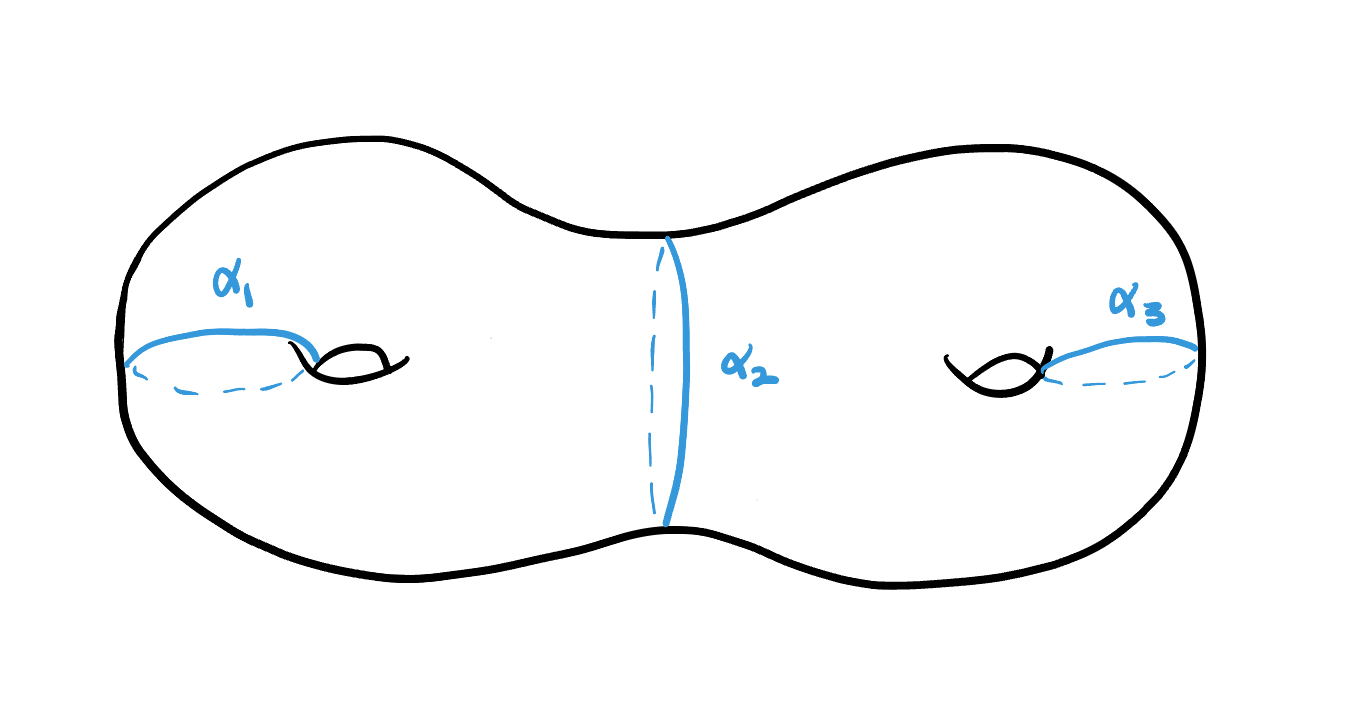}
        \captionsetup{margin=.5cm,justification=centering}
        \caption{A pants decomposition of a genus 2 surface.}
\end{figure}

In fact, a pants decomposition may be stabilized by two kinds of mapping classes. Firstly, some elements of the mapping class group may permute the curves in the pants decomposition. These permutations do not cause problems; any pants decomposition contains finitely many curves, so there are finitely many possible permutations. Second, any power of a Dehn twist about one of the curves in the pants decomposition stabilizes the pants decomposition. Dehn twists are infinite order, so if we hope to determine an element of the mapping class group, we must add in some way to ``keep track'' of Dehn twists. In \cite{Masur-Minsky-marking-graph-qi}, Masur and Minsky showed that there is a suitable way of adding extra curves to a pants decomposition which accomplishes this goal. First, it is easier to see that there are arcs in the annular subsurface with core curve $\alpha$ which are not fixed by such Dehn twists. 

\begin{definition}
    A \emph{complete marking} is a set $\{p_1, \cdots, p_k\}$ such that each $p_i$ is an ordered pair $(\alpha_i, t_i)$ where the $\{\alpha_i\}$ form a pants decomposition and each $t_i$ is a diameter 1 set of vertices of $C(\alpha_i)$, the annular curve graph of $\alpha_i$ as defined in \cite{Masur-Minsky-marking-graph-qi}.
\end{definition}
This definition can be modified slightly to obtain a subset of markings which are built from curves, as follows.

\begin{definition}
    Given an essential simple closed curve $\alpha$ on $S$, an essential simple closed curve $\beta$ is called a \emph{clean transverse curve} for $\alpha$ if when $\alpha$ and $\beta$ are placed in minimal position, they have the minimal possible non-zero intersection number, i.e., a regular neighborhood of $\alpha \cup \beta$ is either a 1-holed torus or a 4-holed sphere.

    A marking $\mu$ is called \emph{clean} if every $p_i$ is of the form $(\alpha_i , \pi_{\alpha_i}(\beta_i))$ where $\beta_i$ is a clean transverse curve for $\alpha_i$ which is disjoint from $\alpha_j$ for any $j \neq i$.
\end{definition}

Any non-clean complete marking can be associated to some clean markings as follows.

\begin{definition}
    A complete marking $\mu = \{(\alpha_i, t_i)\}$ is said to be \emph{compatible} with a clean marking $\mu' = \{(\alpha_i, t'_i)\}$ if $\mu'$ is also complete, the base collections of $\mu$ and $\mu'$ are the same, and $d_{\alpha_i}(t_i, t'_i)$ is minimal amongst all possible choices of curve $t'_i$.
\end{definition}
\cite[Lemma~2.4]{Masur-Minsky-marking-graph-qi} shows that every complete marking is uniformly close to a complete clean marking in each $C_{\alpha_i}$, which allows us to restrict our attention to clean markings.
\begin{figure}[h!]
        \centering
        \includegraphics[width=8cm]{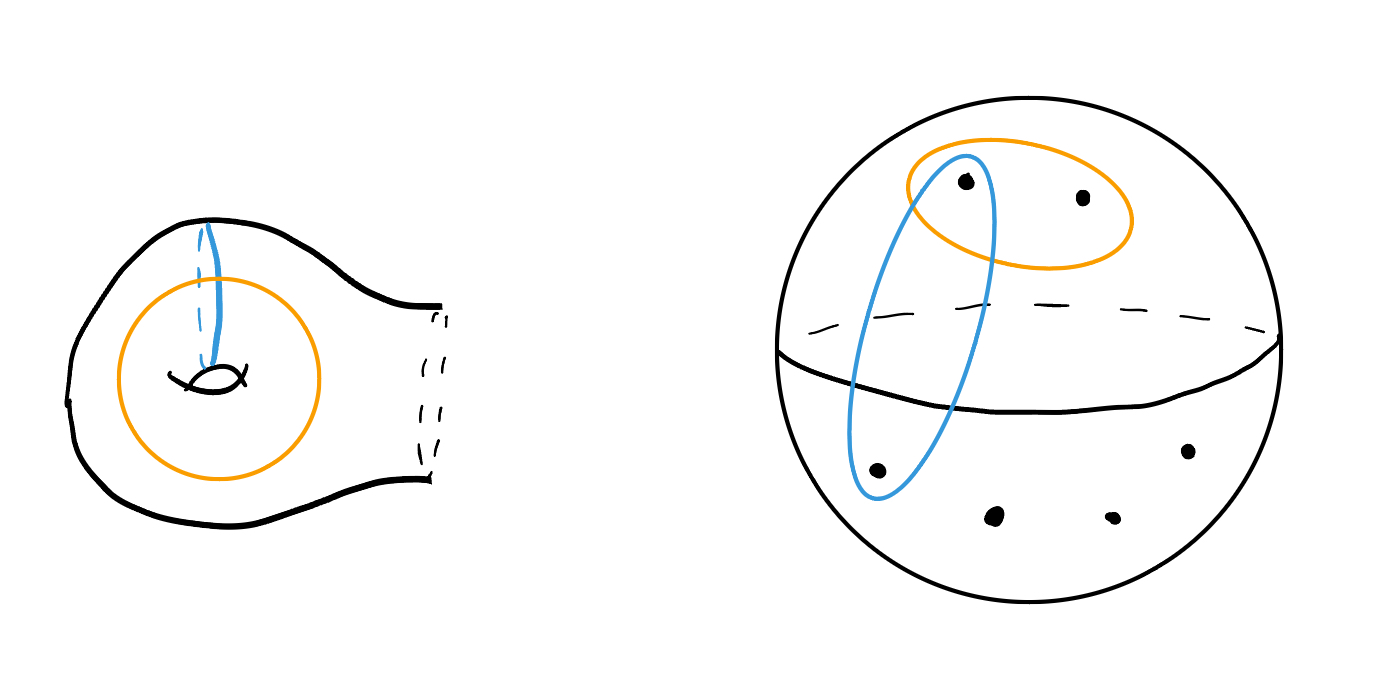}
        \captionsetup{margin=.5cm,justification=centering}
        \caption{The blue and orange curves intersect minimally in both figures.}
\end{figure}

\begin{figure}[h!]
        \centering
        \includegraphics[width=10cm]{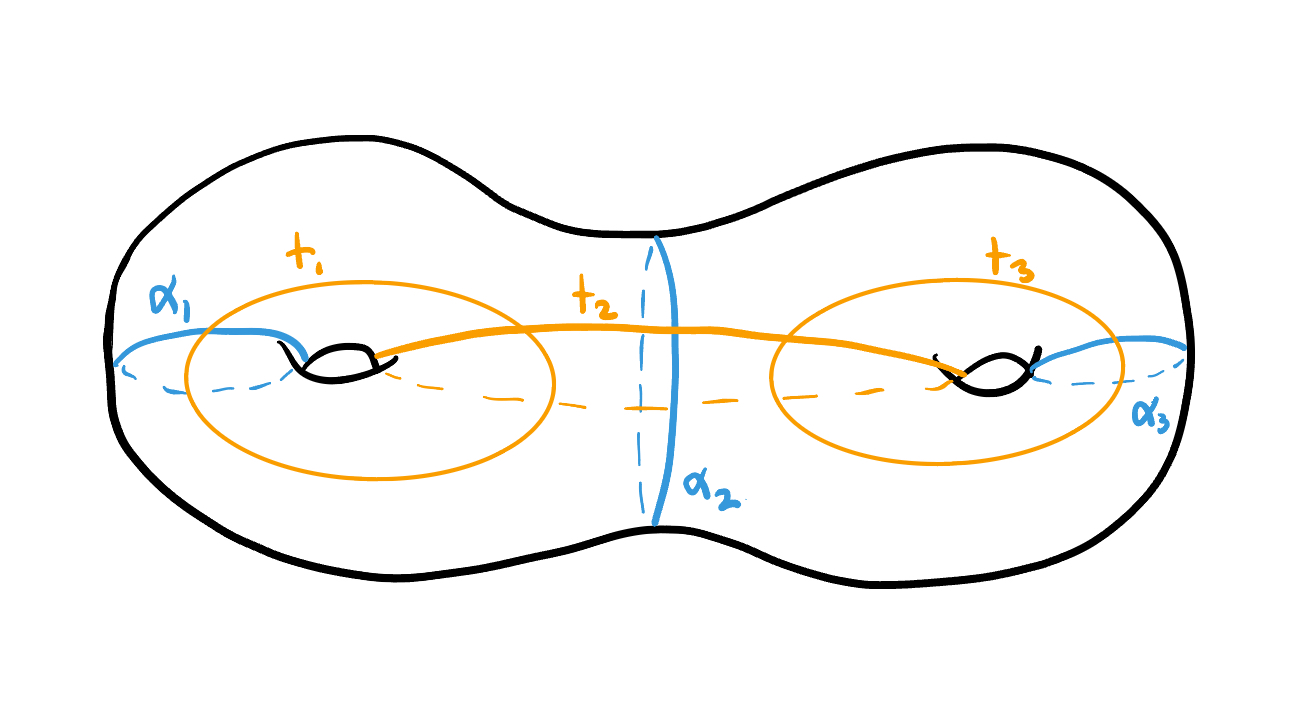}
        \captionsetup{margin=.5cm,justification=centering}
        \caption{A clean marking with base curves $\{\alpha_i\}$ and tranvserse curves $\{t_i\}$.}
        \label{fig:nonunique-clean-markings}
\end{figure}

We can now explain the technical motivation for working on surfaces without boundary: on a surface $S$ with boundary, a Dehn twist about a component of $\partial S$ is an infinite order element of the mapping class group, but there are no curves on the surface which intersect the boundary, so we cannot use transverse curves to ``keep track'' of these twists. This issue can be resolved with the introduction of annular curve graphs and non-clean markings, but this will not be necessary for our purposes. 

Masur and Minsky also defined \emph{elementary moves} between complete clean markings. 
\begin{definition}
    Let $\mu = \{(\alpha_i, t_i)\}$ be a complete clean marking. There are two types of elementary moves which transform $\mu$ into a new clean marking:
\begin{enumerate}
    \item Twist: Replace $t_i$ with $t'_i$, where $t'_i$ is obtained from $t_i$ by a Dehn twist or half-twist around $\alpha_i$.
    \item Flip: Replace a fixed pair $(\alpha_i , t_i)$ in $\mu$ by $(t_i , \alpha_i)$ to obtain a marking $\mu''$ which is not necessarily clean, and replace $\mu''$ by a compatible clean marking $\mu'$.

    \begin{figure}[h!]
        \centering
        \includegraphics[width=8cm]{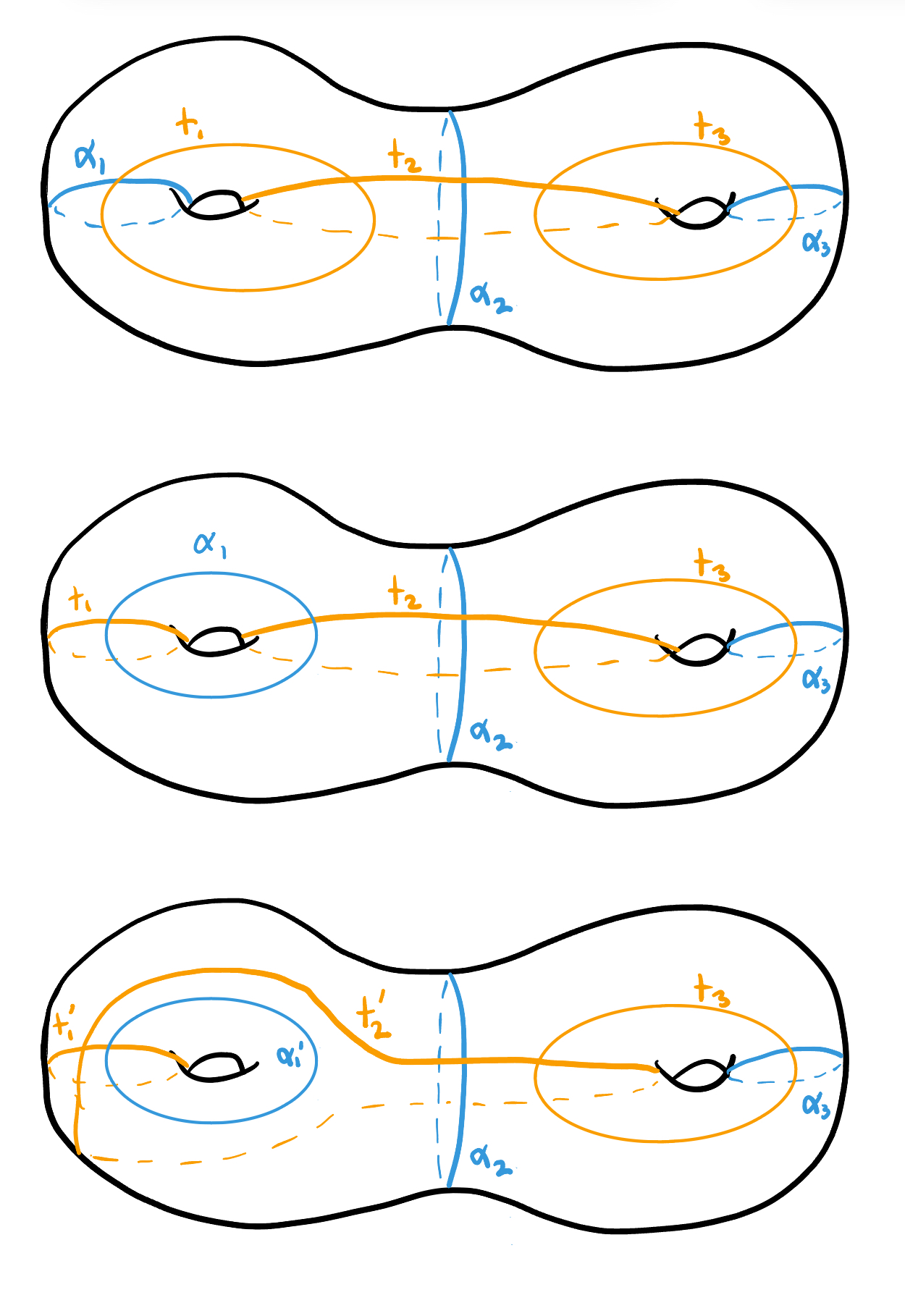}
        \captionsetup{margin=.5cm,justification=centering}
        \caption{The three steps of performing a flip move across $(\alpha_1, t_1)$}
        \label{fig:flip-fig}
\end{figure}
\end{enumerate}
\end{definition}
Intuitively, one should think of a flip move as interchanging the roles of a particular $(\alpha_i , t_i)$ pair, and replacing $t_j$ for $j \neq i$ with $t'_j$ such that $t_j$ and $t'_j$ look almost the same when we restrict to an annulus with core curve $\alpha_j$. We can now define the marking graph.

\begin{definition}
    The \emph{marking graph} of a surface $S$ is the graph whose vertex set consists of complete clean markings on $S$ and where two vertices are connected via an edge if one of the corresponding clean markings can be obtained from the other via an elementary move.
\end{definition}

Masur and Minsky showed that there are finitely many clean markings $\mu'$ which can be obtained from a given clean marking $\mu$ via a flip move. It is clear that there are only finitely many markings which can be obtained from $\mu$ via twist moves, so the marking graph is locally finite. The following theorem says that the clean transverse curves successfully encode the information which was missing from pants decompositions. Throughout this paper, an action is \emph{geometric} if it is properly discontinuous, cocompact, and by isometries.

\begin{theorem}\cite{Masur-Minsky-marking-graph-qi}
The natural action of $MCG(S)$ on the marking graph is geometric, i.e., the marking graph is quasi-isometric to $MCG(S)$.
\end{theorem}

\subsection{Finite-type Artin groups}\label{sec:finite type artin groups}

An Artin group is a group with a presentation of the form \begin{align*}
            \langle \sigma_1,...\sigma_n | \underbracket{\sigma_i\sigma_j\sigma_i...}_{m_{ij}} = \underbracket{\sigma_j \sigma_i \sigma_j ...}_{m_{ij}} \rangle
        \end{align*}
        for some constants $m_{ij}$. In general, some choices of $i$ and $j$ may have no relation, but this does not occur in the class of Artin groups we consider in this paper. Note that $m_{ij} = m_{ji}$.

        The relation constants $m_{ij}$ are often encoded in a defining graph. Each Artin generator corresponds to a vertex in the graph. If $m_{ij} > 3$, then the vertices corresponding to $\sigma_i$ and $\sigma_j$ are connected by an edge labeled with $m_{ij}$. The vertices corresponding to $\sigma_i$ and $\sigma_j$ are connected by an unlabeled edge if $m_{ij} = 3$, and they are not connected by any edge when $m_{ij}=2$. A finite-type Artin group is called \emph{reducible} if it can be decomposed as the direct product of two finite-type Artin groups where neither factor is trivial. An irreducible finite-type Artin group, for our purposes, can be thought of as an Artin group whose defining graph is contained in Figure \ref{fig:defgraphs} \cite{Coxeter-reflections}. We write $A_{\Gamma}$ or $A_{S}$ to refer to the Artin group with defining graph $\Gamma$ or to an Artin group with Artin generating set $S$.

        \begin{figure}[h]
        \centering
        \includegraphics[width=12cm]{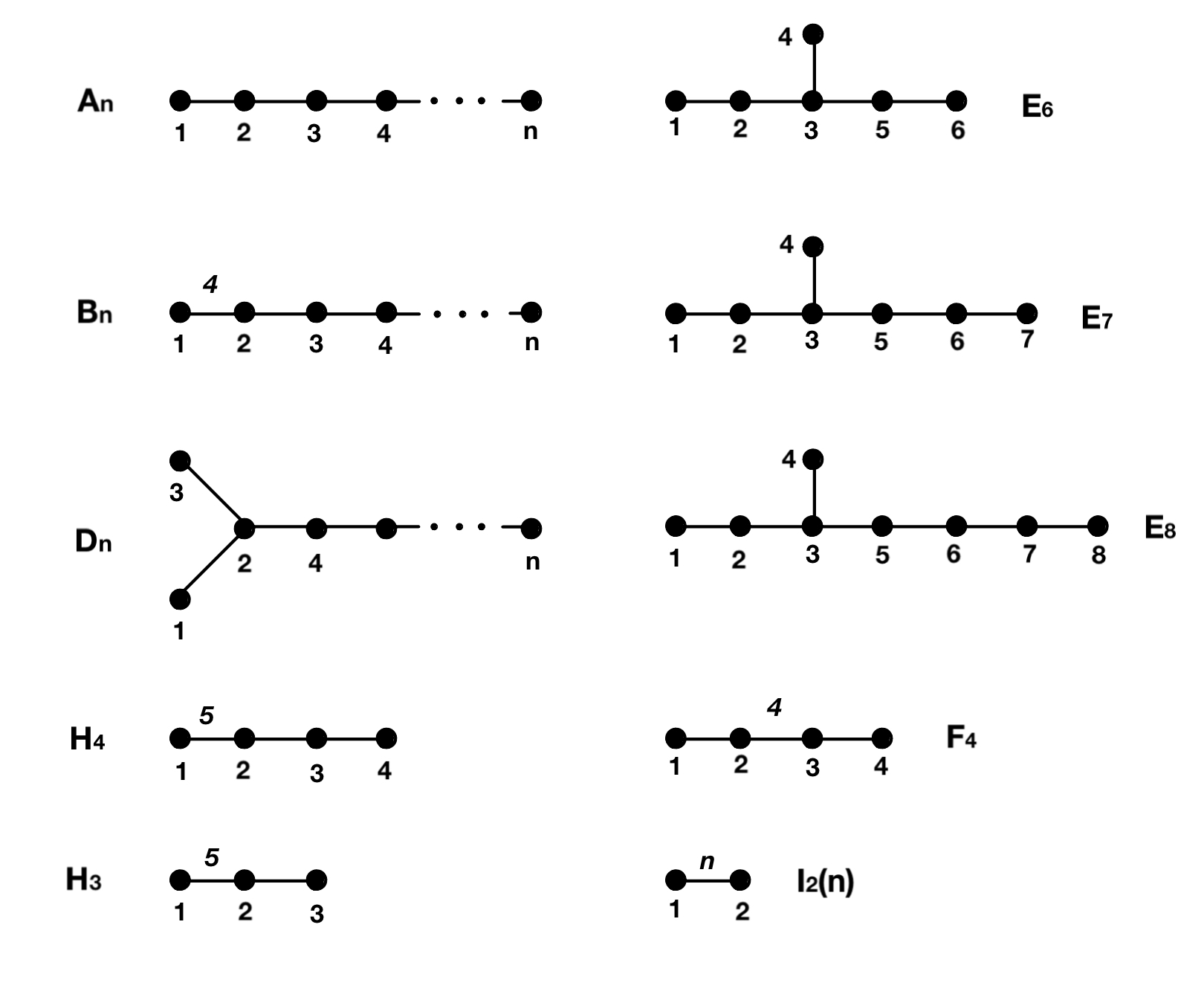}
        \captionsetup{margin=.5cm,justification=centering}
        \caption{Labeled defining graphs $\Gamma$ of irreducible finite-type Artin groups $A_{\Gamma}$}
        \label{fig:defgraphs}
    \end{figure}

\subsubsection{Garside elements and parabolic subgroups}
        Finite-type Artin groups were the motivating example for Garside groups, a class of groups which have been studied independently. We will restrict our attention to the finite-type Artin group case, but the Garside structure is nevertheless a useful framework.

        A \textit{Garside monoid} is a pair $(M, \Delta)$ where $M$ is a monoid which is left and right-cancellative, has the property that any two elements of $M$ have both left and right least common multiple and greatest common divisor, and which admits a map $\lambda: M \rightarrow \mathbb{N}$ such that $\lambda(fg) \geq \lambda(f) + \lambda(g)$ and $\lambda(g) \neq 0$ when $g \neq 1$. In addition, $\Delta$ is an element of $M$, called the \textit{Garside element}, such that the left and right divisors of $\Delta$ coincide and generate $M$. For our purposes, the set of divisors of $\Delta$ will be finite.

        A group $G$ is said to be a \textit{Garside group} if there exists a Garside monoid $(M, \Delta)$ such that $G$ is the group of fractions for $M$. The monoid admits two natural partial orderings: the prefix order and the suffix order. We say that $a$ is a prefix of $b$, denoted $a \preccurlyeq b$ if there is some $c \in M$ such that $ac = b$, and $a$ is a suffix of $b$ if there is some $c \in M$ such that $ca = b$. Elements of $G$ which are also in $M$ are called \emph{positive}. Either ordering can be naturally extended to the group by defining $a$ to be a prefix of $b$ if $a^{-1}b$ is a positive element or a suffix of $b$ if $b a^{-1}$ is a positive element.

        It was shown in \cite{ft-artin-groups-geodesically-biautomatic} that every element $g$ in a finite-type Artin group can be written as a product $a b^{-1}$ where $a$ and $b$ are both positive elements, and that there is a unique choice of $a$ and $b$ such that $a$ and $b$ have no nontrivial shared suffix. The product $ab^{-1}$ is called the \emph{$pn$-normal form} of $g$.

        If $g$ is a positive element of a finite-type Artin group, the same subset of Artin generators appear in any positive word representing $g$. This collection is called the \emph{support} of $g$, denoted $\text{supp}(g)$. The support of a general element is defined to be the union of the supports of positive elements $a$ and $b$ such that $ab^{-1}$ is the $pn$-normal form for $g$.

        \begin{example}
            Let $A_{S}$ be the Artin group of type $I_{2}(4)$. The element $s_1 s_2 s_1 s_2^2$ has both $s_1 s_2 s_1$ and $s_2 s_1 s_2$ as prefixes.
            \begin{align*}
                (s_1 s_2 s_1)^{-1}s_1 s_2 s_1 s_2^2 &= s_2^{2}\\
                (s_2 s_1 s_2)^{-1} s_1 s_2 s_1 s_2^2 &= (s_2 s_1 s_2)^{-1} s_2 s_1 s_2 s_1 s_2\\
                &= s_1 s_2\text{.}
            \end{align*}
            The element $s_1 s_2 s_1 s_2^2$ does not have $s_2^6$ as a prefix. One way to see this is that any two positive words which represent the same element of the Artin group have the same word length with respect to the Artin generating set, and $|s_1 s_2 s_1 s_2^2| = 5$ while $|s_2^6| = 6$. The element $s_2^6$ does not have $s_1 s_2 s_1 s_2^2$ as a prefix either; both $s_1 s_2 s_1 s_2^2$ and $s_2^6$ are positive elements, and $\text{supp}(s_1 s_2 s_1 s_2^2) = \{s_1, s_2\}$ while $\text{supp}(s_2^{6}) = \{s_2\}$.
        \end{example}

        A finite-type Artin group, $A_S$, is a Garside group where $M$ is the Artin monoid $A_S^{+}$ generated by the Artin generating set $S$, and $\Delta$ is the least common multiple of the Artin generators in $S$. The least common multiple may be taken with respect to either the prefix or suffix order; the same element is obtained. Explicit expressions for $\Delta$ can be found in \cite{Brieskorn-Saito}. 
            
        Irreducible finite-type Artin groups have the special property that the center of the group, $Z(A_S)$, is generated by $\Delta^2$ when the defining graph is of type $A_n$ for $n \geq 2$, $D_n$ for $n \geq 5$ and odd, $E_6$, or $I_2(n)$ for $n \geq 5$ and odd (the cases $n = 3$ are excluded only to avoid overlap in the naming convention). In all other cases, the center of $A_{\Gamma}$ is generated by $\Delta$ itself \cite{Brieskorn-Saito}.

        When $\Delta$ is not central, conjugation by $\Delta$ induces a permutation on the generating set that corresponds to the following label-preserving automorphism of the defining graph. 
        \begin{itemize}
                \item If $A_S$ is of type $A_n$, then $\Delta \sigma_i \Delta^{-1} = \sigma_{n + 1 - i}$, i.e., the first generator is sent to the final one, the second is send to the penultimate, etc.
                \item If $A_S$ is of type $D_n$ for $n \geq 5$ and odd, then conjugation by $\Delta$ permutes the two Artin generators corresponding to vertices on the ends of the prongs in the defining graph and fixes all other generators.
                \item If $A_S$ is of type $E_6$, then conjugation by $\Delta$ fixes the generator corresponding to the prong vertex and induces the same permutation as in the $A_n$ case on the subgraph of type $A_5$.
                \item If $A_S$ is of type $I_2(n)$ for $n$ odd, then conjugation by $\Delta$ permutes the two Artin generators.
        \end{itemize}

        Let $A_S$ be a finite-type Artin group with Artin generating set $S$, and consider $T \subseteq S$. The subgroup of $A_{S}$ generated by $T$ is called a \textit{standard parabolic subgroup} of $A_S$. When $g$ is a non-trivial element of $A_S$, the conjugate $g A_T g^{-1}$ is called a \textit{parabolic subgroup} of $A_S$. Much investigation has been done on the topic of parabolic subgroups, and we recall some key results needed in what follows.

        \begin{theorem}\cite[Theorem~4.13]{van-der-lek}\cite[Theorem~3.1]{Paris-subgroup-conjugacy}\label{parbs are artin groups}
            Let $P = \alpha A_{X} \alpha^{-1}$ be a parabolic subgroup of an Artin group. The group $P$ is an Artin group, and its defining graph is the induced subgraph of $\Gamma$ with vertex set $X$.
        \end{theorem}

        In the case where $A_{\Gamma}$ is of finite-type, Theorem \ref{parbs are artin groups} additionally implies that $\alpha A_X \alpha^{-1}$ is an Artin group of finite-type because all subgraphs of $\Gamma$ are defining graphs of finite-type Artin groups. Furthermore, each standard parabolic subgroup has its own Garside element. When the subgraph of $\Gamma$ with vertex set $X$ is connected, the parabolic subgroup $A_X$ is an irreducible finite-type Artin group, and $\Delta_X$ is defined precisely as for the finite-type Artin group $A_X$. When the subgraph of $\Gamma$ with vertex set $X$ is disconnected, the parabolic subgroup $A_X$ is a reducible finite-type Artin group, and the Garside element $\Delta_X$ is the product, in any order, of the Garside elements of the irreducible components.

        \begin{theorem}\cite[Theorem~9.5]{C-parab-definition}\label{int of parabolics}
            Let $P$ and $Q$ be two parabolic subgroups of an Artin group of finite type. Then $P \cap Q$ is also a parabolic subgroup of $A_{\Gamma}$.
        \end{theorem}

        \begin{theorem}\cite[Theorem~10.3]{C-parab-definition}
            The set of parabolic subgroups of a finite-type Artin group is a lattice with respect to the partial order determined by inclusion.
        \end{theorem}

        \begin{definition}\label{def:conj-graph}
            Let $A_{\Gamma}$ be a finite-type Artin group, and let $S = V(\Gamma)$. The associated graph $G$ is defined by the following data.
            The vertices of $G$ are the subsets $X \subseteq S$.
            An edge of $G$ is a triple $(Y, t, t')$ satisfying the following conditions.
                \begin{itemize}
                    \item $Y \subseteq S$.
                    \item There exists a connected component $\Gamma_0$ of $\Gamma_Y$ such that both $t$ and $t'$ are vertices of $\Gamma_0$.
                    \item $\Gamma_0 \in \{A_l : l \geq 2\} \cup \{D_l : l \geq 5\text{ and }l\text{ odd}\} \cup \{E_6\} \cup \{I_2(p) : p \geq 5\text{ and }p\text{ odd}\}$.
                    \item Let $Y_0$ be the set of vertices of $\Gamma_0$, and let $\delta_0 : Y_0 \rightarrow Y_0$ be the permutation such that $\Delta_{Y_0}s\Delta_{Y_0}^{-1} = \delta_0(s)$ for all $s \in Y_0$. Then $t' = \delta_0(t)$, and $t \neq t'$.
                \end{itemize}
            The edge $(Y, t, t')$ joins $X = Y - \{t\}$ with $X' = Y - \{t'\}$. 
        \end{definition}

        \begin{theorem}\cite[Theorem~4.1]{Paris-subgroup-conjugacy}\label{parab subgroup conjugacy}
            Let $A_{\Gamma}$ be a finite-type Artin group, and let $G$ be as in Definition \ref{def:conj-graph}. There exists $\beta \in A_{\Gamma}$ such that $\beta A_{X} \beta^{-1} = A_{X'}$ if and only if $X$ and $X'$ are in the same connected component of $G(\Gamma)$.
        \end{theorem}

        Notice that any $X$ and $X'$ which are connected by an edge in this graph must have $|X| = |X'|$. This implies that $|X| = |X'|$ for every $X$ in a given connected component of $G$. In particular, if $A_{X}$ is conjugate to $A_{X'}$, then $|X| = |X'|$.

        \begin{example}
            In the Artin group of type $E_8$, the subgroup $\langle s_1, s_2, s_3, s_4 \rangle$ is conjugate to the subgroup $\langle s_5, s_6, s_7, s_8 \rangle$. One possible path is that $\langle s_1, s_2, s_3, s_4 \rangle$ is contained in the standard $D_5$ subgroup $\langle  s_1 , s_2 , s_3, s_4, s_5 \rangle$, and the permutation $\delta_0$ corresponding to this subgroup sends $\langle s_1, s_2, s_3, s_4 \rangle$ to $\langle s_1, s_2, s_3, s_5 \rangle$. Inclusion in subgroups of type $A_5$ yields the remainder of the path, as shown in Figure \ref{fig:actual-G_path}.

            \begin{figure}[h!]
            \centering
            \includegraphics[width=8cm]{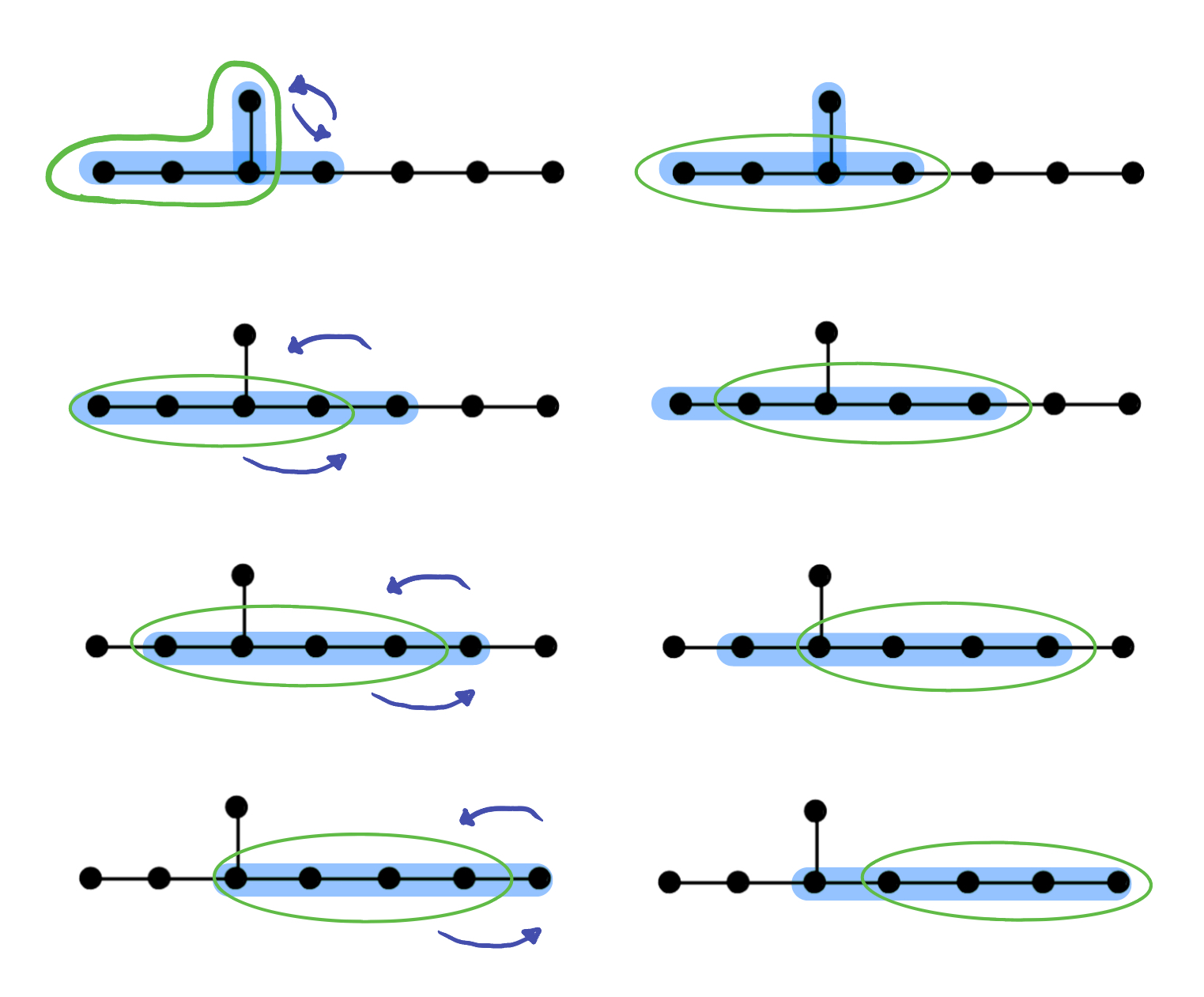}
            \captionsetup{margin=.5cm,justification=centering}
            \caption{A path from $\langle s_1, s_2, s_3, s_4 \rangle$ to $\langle s_5, s_6, s_7, s_8 \rangle$ in $G$.}
            \label{fig:actual-G_path}
        \end{figure}     
            
            While paths in $G$ determine when two subgroups are conjugate, they are not necessarily the only nor the most efficient method of conjugating $\langle s_1, s_2, s_3, s_4 \rangle$ to $\langle s_5, s_6, s_7, s_8 \rangle$. The subgroup $\langle s_1, s_2, s_3, s_4 \rangle$ is contained in the unique standard $E_6$ subgroup, and the permutation $\delta_0$ corresponding to this $E_6$ sends $\langle s_1, s_2, s_3, s_4 \rangle$ to $\langle s_3, s_4, s_5, s_6 \rangle$. The subgroup $\langle s_3, s_4, s_5, s_6, s_7, s_8 \rangle$ is of type $A_6$, and the permutation $\delta_0$ of the generators of $A_6$ sends $\langle s_3, s_4, s_5, s_6 \rangle$ to $\langle s_5, s_6, s_7, s_8 \rangle$, as shown in Figure \ref{fig:alt-path}.

        \begin{figure}[h!]
            \centering
            \includegraphics[width=8cm]{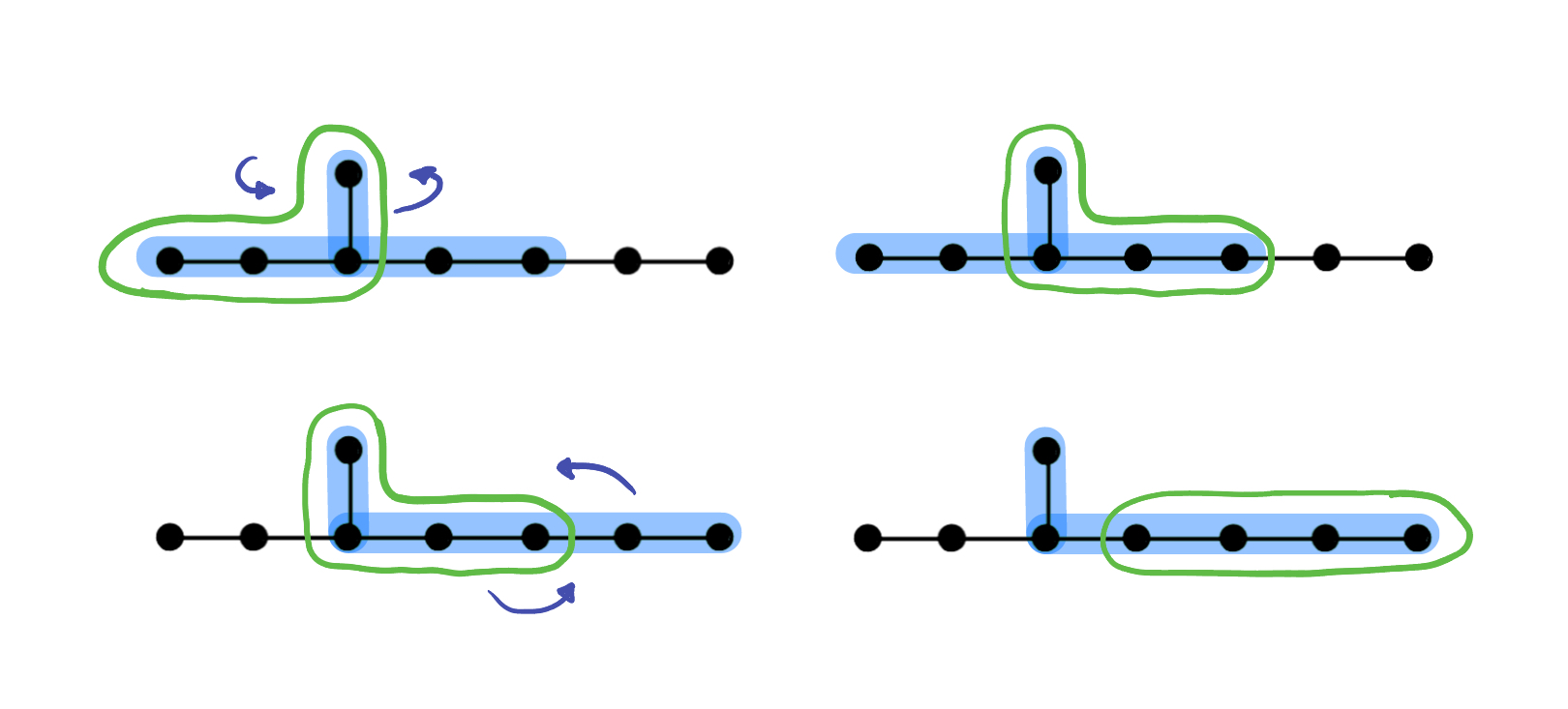}
            \captionsetup{margin=.5cm,justification=centering}
            \caption{Another path from $\langle s_1, s_2, s_3, s_4 \rangle$ to $\langle s_5, s_6, s_7, s_8 \rangle$.}
            \label{fig:alt-path}
        \end{figure}


        \end{example}

        \begin{example}
            In the Artin group of type $A_n$, the irreducible standard parabolic subgroups $A_X$ and $A_Y$ are conjugate if $|X| = |Y|$. Let $X = \{s_i, \cdots, s_{i+m}\}$ and $Y = \{s_j, \cdots, s_{j+m}\}$. If $i = j$, then $X = Y$ and we are done, so suppose without loss of generality that $i < j$.

            Consider the standard parabolic subgroup generated by $\{s_i, \cdots, s_{i+m+1}\}$. All irreducible standard parabolic subgroups of $A_n$ are of type $A_{n'}$, so this subgroup is of type $A_{|m+1|}$. By construction, the permutation $\delta_0$ sends $s_i$ to $s_{i+m+1}$, $s_{i+1}$ to $s_{i+m}$, and so on. This yields an edge in $G$ between $X$ and the set $\{s_{i+1}, \cdots , s_{i+m+1}\}$. Repeating this trick until we reach generator $s_{j+m}$ yields a path in $G$ from $X$ to $Y$.
        \end{example}

    While parabolic subgroups are Artin groups themselves, there is not always an obvious unique choice of Garside element such that $g A_X g^{-1} = P$ implies $g \Delta_X g^{-1} = \Delta_P$. For example, if $X$ contains the generator $s_1$, then $s_1 \Delta_X s_1^{-1}$ does not in general equal $\Delta_X$ when the center of $A_X$ is generated by $\Delta_X^2$. There is, however, a unique way of choosing a central power of a Garside element of $P$.

        \begin{theorem}\cite[Proposition 35]{Cumplido-minimal-standardizers}
            Let $A_S$ be a finite-type Artin group, and let $P = g A_X g^{-1} = h A_{X'} h^{-1}$ be a parabolic subgroup. Let $z_X$ denote $\Delta_X$ or $\Delta_X^2$, whichever is the minimal central power of $\Delta_X$. Then $g z_X g^{-1} = h z_{X'} h^{-1}$, and we call this element $z_P$. If $P$ is irreducible, $z_P$ is the unique element which generates the center of $P$ and is conjugate to a positive element. 
        \end{theorem}

    There may be many different ways of writing a parabolic subgroup as a conjugate of a standard parabolic subgroup. The following theorem shows that there is a unique way which is in some sense ``simplest''. 
    
        \begin{theorem}\cite[Theorem 4]{Cumplido-minimal-standardizers}\label{thm:cumplido-minl-pos}
            Let $P$ be a parabolic subgroup of a finite-type Artin group $A_S$, and let $ab^{-1}$ be the $pn$-normal form of $z_P$. There is a subset $X \subseteq S$ such that $b A_X b^{-1} = P$, and for any positive element $h$ and subset $Y \subseteq S$ with $h A_Y h^{-1} = P$, $b$ is a prefix of $h$. We call $b$ the \emph{minimal positive standardizer} of $P$. 
        \end{theorem}

\subsubsection{Braid groups as Artin groups and mapping class groups}\label{sec:braid groups}

The finite-type Artin group of type $A_n$ is the braid group on $n+1$ strands. The generator $s_i$ is the braid which crosses strand $i$ over strand $i+1$. The Artin relation $s_i s_{i+1} s_i = s_{i+1} s_i s_{i+1}$ can be visualized as the braid relation, depicted in the figure below.

    \begin{figure}[h!]
            \centering
            \includegraphics[width=8cm]{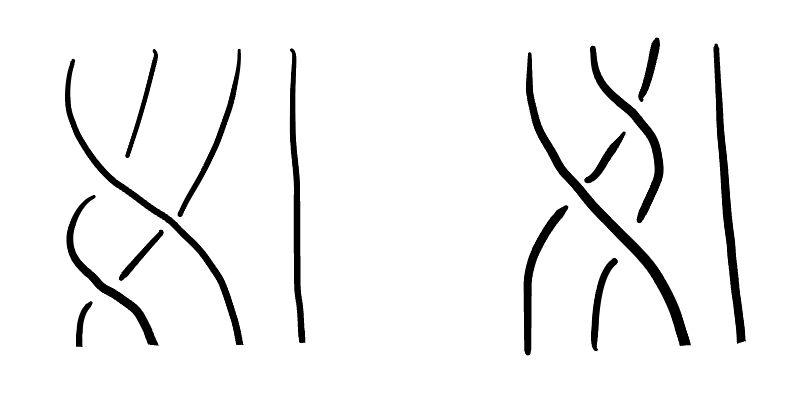}
            \captionsetup{margin=.5cm,justification=centering}
            \caption{The braid relation $s_1 s_2 s_1 = s_2 s_1 s_2$ in $A_3$.}
    \end{figure}

    The braid group on $n$ strands is also the mapping class group of an $n$-punctured disk. The generator $s_i$ is a half-twist about the curve enclosing punctures $i$ and $i+1$, and the Garside element $\Delta$ is a half-twist about the boundary of the disk which swaps the first puncture and the last puncture, the second puncture and the second-to-last puncture, etc.

    \begin{figure}[h!]
            \centering
            \includegraphics[width=8cm]{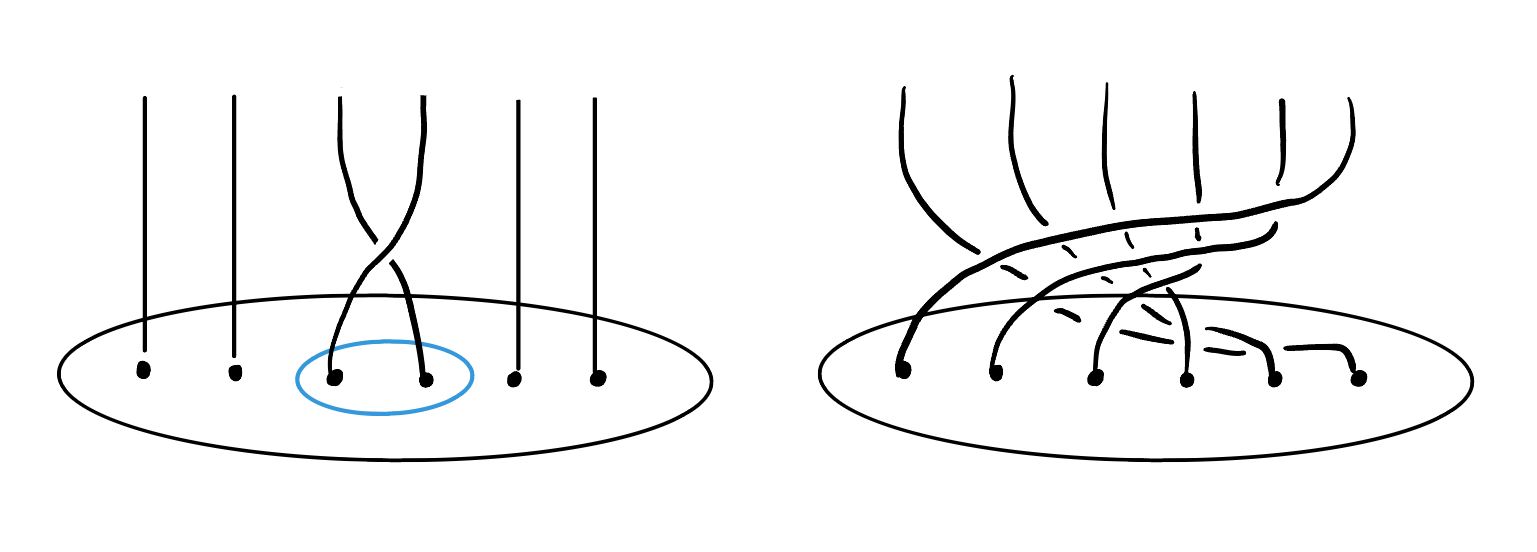}
            \captionsetup{margin=.5cm,justification=centering}
            \caption{The Artin generator $s_3$ and $\Delta$ viewed as elements of $MCG(D_n)$.}
    \end{figure}

    There is a natural correspondence between the standard parabolic subgroups and \emph{round curves} on the punctured disk: the standard parabolic subgroup $\langle s_i, s_{i+1}, \cdots s_j \rangle$ corresponds to the round curve enclosing punctures $i, i+1 \cdots, j+1$. We will not define a round curve precisely since these objects are only referenced for the sake of intuition; our meaning should be sufficiently clear from the examples in Figures \ref{fig:round-curve} and \ref{fig:rounding-a-curve-two-ways}. When viewed as an element of the mapping class group of the $n$-punctured disk, every element of the standard parabolic subgroup is supported inside of the subdisk bounded by this curve. In addition, the Garside element of the standard parabolic subgroup is a half-twist about the curve.

    \begin{figure}[h!]
            \centering
            \includegraphics[width=8cm]{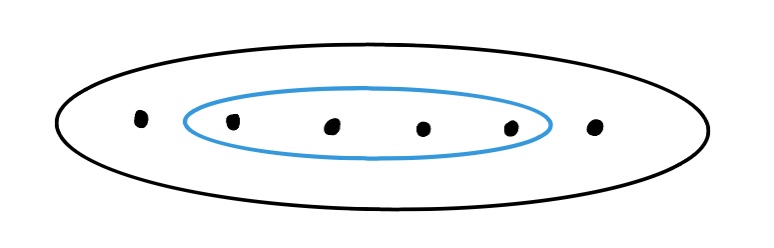}
            \captionsetup{margin=.5cm,justification=centering}
            \caption{The round curve enclosing the support of $\langle s_2, s_3, s_4 \rangle$.}
            \label{fig:round-curve}
    \end{figure}

    In fact, this correspondence naturally extends to non-standard parabolic subgroups. Consider any essential simple closed curve $C$ on the $n$-punctured disk. There is some homeomorphism $\alpha$ of the $n$-punctured disk which sends $C$ to the round curve enclosing the first $m$ punctures for some $m$, so the subset of braids supported on the subsurface enclosed by $C$ is precisely $\alpha A_{X} \alpha^{-1}$ where $X = \{s_1, \cdots, s_{m-1}\}$. To see this, notice that braids supported inside $C$ can be obtained by first sending $C$ to a round curve via $\alpha$, performing any braid supported in the round curve, and then returning to $C$ via $\alpha^{-1}$.

    One can check that every irreducible parabolic subgroup arises in this way. Namely, $\alpha A_{X} \alpha^{-1}$ consists of the braids supported inside of the image under $\alpha^{-1}$ of the round curve which bounds the support of the Artin generators in $X$, since $\alpha$ sends this to a round curve. While the number of punctures in this round curve is unique, the particular choice of round curve is not. There are automorphisms of the disk which send any round curve containing $m$ punctures to any other round curve containing $m$ punctures, so there are multiple automorphisms which send a particular curve $C$ to a round curve.

        \begin{figure}[h!]
            \centering
            \includegraphics[width=8cm]{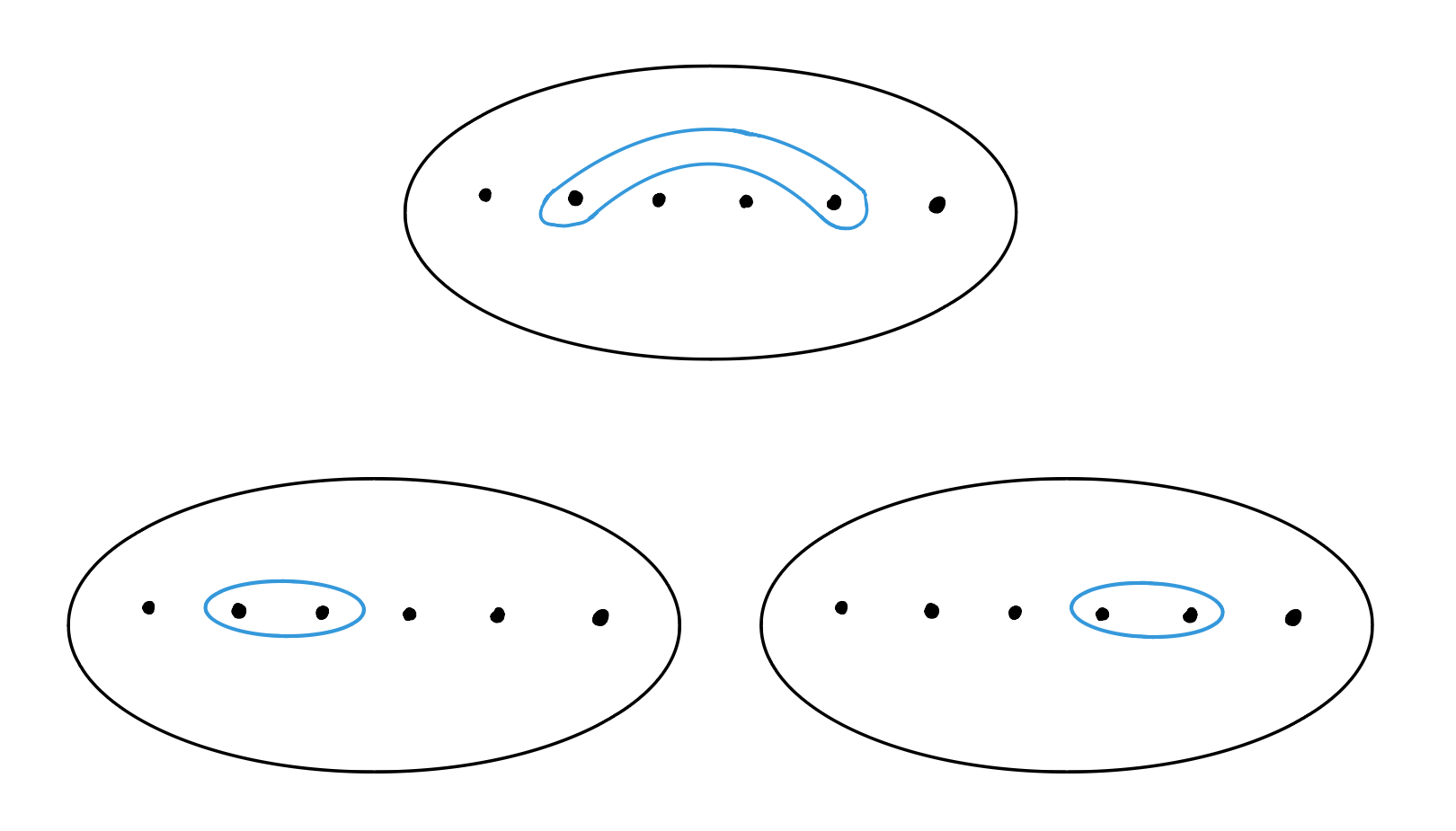}
            \captionsetup{margin=.5cm,justification=centering}
            \caption{The non-round curve can be transformed to the round curve enclosing punctures $2$ and $3$ via $s_4^{-1} s_3^{-1}$ or to the round curve enclosing punctures $4$ and $5$ via $s_2 s_3$.}
            \label{fig:rounding-a-curve-two-ways}
        \end{figure}

    In \cite{C-parab-definition}, the authors showed that if $\alpha$ sends $C$ to a round curve enclosing punctures $i \cdots i+m+1$ and $\alpha'$ sends $C$ to a round curve enclosing punctures $j \cdots j +m+1$, then the parabolic subgroups $\alpha A_{s_i, \cdots, s_{i+m}} \alpha^{-1}$ and $\alpha' A_{s_j, \cdots, s_{j+m}} \alpha'^{-1}$ are equal. Thus the correspondence between irreducible parabolic subgroups of the braid group on $n$ strands and essential simple closed curves on the $n$-punctured disk is a bijection.

    The group $A_{\Gamma}$ acts naturally on the collection of irreducible parabolic subgroups via conjugation. In the case of the braid group, this agrees with the action on simple closed curves: if $\beta$ is an element of the braid group and $\alpha A_X \alpha^{-1}$ is a parabolic subgroup, then the curve $\beta \alpha(X)$ clearly contains the support of $\beta \alpha A_{X} \alpha^{-1}\beta^{-1}$.

    Notice that the central element $z_{\Gamma}$ acts trivially on every parabolic subgroup, so if our goal is to construct a complex of parabolic subgroups analogous to the marking graph, we can at best hope to construct a complex which is quasi-isometric to $A_{\Gamma}/Z(A_{\Gamma})$. The analogy with the marking graph is still reasonable since, as we have previously noted, the braid group on $n$ strands modulo its center is a finite-index subgroup of the mapping class group of an $(n+1)$-punctured sphere. Despite the fact that an $(n+1)$-punctured sphere is more appropriately associated to $A_{\Gamma}$ modulo its center than an $n$-punctured disk, we will often provide intuition via figures drawn on the disk rather than the sphere in later sections. We do this for two reasons. Firstly, the curves are more clearly visible in a planar diagram. Second, we will see that subgroups which are properly included in one another map to disjoint curves. It is often useful to keep track of the direction of this subgroup containment, and this is more readily displayed via curves on the disk.
      
\subsubsection{The complex of irreducible parabolic subgroups}\label{sec:Cparab}
        In the previous section, we saw that there is a natural correspondence between irreducible proper parabolic subgroups of the braid group and isotopy classes of essential simple closed curves on the disk. In \cite{C-parab-definition}, the authors provide group theoretic descriptions of irreducible parabolic subgroups which correspond to disjoint curves. Specifically, they show that given two such subgroups $P$ and $Q$, the corresponding curves on the disk are disjoint if and only if one of the following holds:
            \begin{enumerate}
                \item $P \leq Q$
                \item $Q \leq P$
                \item $P \cap Q = \{1\}$ and $P$ and $Q$ commute.
            \end{enumerate}
            
        Remarkably, the authors showed that in \emph{any} irreducible finite-type Artin group, one of the above conditions holds if and only if $z_P$ commutes with $z_Q$. They then defined the \emph{complex of irreducible parabolic subgroups}, which generalizes the curve complex for braid groups (and other mapping class groups).

            \begin{definition}\label{def:c-parab-def}
                Let $A_{\Gamma}$ be an irreducible finite-type Artin group. The \emph{complex of irreducible parabolic subgroups}, $C_{parab}(A_{\Gamma})$ is the simplicial complex with vertex set equal to the collection of irreducible, proper parabolic subgroups of $A_{\Gamma}$ and where a collection of $n$ vertices $\{P_i\}$ spans an $(n-1)$-simplex if $z_{P_i}$ commutes with $z_{P_j}$ for all $i, j \in \{1, \cdots, n\}$.
            \end{definition}

        When the choice of $A_{\Gamma}$ is clear, we will often write simply $C_{parab}$. It was shown in \cite{hyp-structures-survey} that this complex is connected when $A_{\Gamma}$ is not of dihedral-type, and it is a consequence of Theorem 2 in \cite{parabs-acting-on-CAL} and Theorem 1.1 in \cite{CAL2} that this complex is infinite diameter when $A_{\Gamma}$ is irreducible. When $A_{\Gamma}$ is reducible, it has diameter at most 2. Since conjugation preserves commuting relations and subgroup inclusion, the natural action of $A_{\Gamma}$ on the collection of parabolic subgroups extends to an action on $C_{parab}$.
        
\section{Simplices, standardizers, and ribbons}\label{sec:simplices standardizers ribbons}
Throughout the paper, we adopt the convention that \emph{conjugating $x$ by $g$} means replacing $x$ with $g x g^{-1}$.

\subsection{Standardizers and maximal simplices}

As a step in the proof of Theorem 2.2 in \cite{C-parab-definition}, the authors showed that any two adjacent vertices in $C_{parab}$ are \emph{simultaneously standardizable} parabolic subgroups.
            
\begin{definition}
    Two parabolic subgroups $P$ and $Q$ are \emph{simultaneously standardizable} if $P = g A_X g^{-1}$ and $Q = g A_Y g^{-1}$ for some element $g \in A_{\Gamma}$ and some standard parabolic subgroups $A_{X}, A_Y \leq A_{\Gamma}$.
\end{definition}

In fact, their method of proof yields a slightly stronger statement, and this strengthening will be useful in what follows. For the sake of clarity, we provide an explicit proof.

\begin{lemma}\label{lem:pos-sim-std}
    Let $P$ and $Q$ be parabolic subgroups of $A_{\Gamma}$ such that $z_P$ and $z_Q$ commute. There are subsets $X, Y \subseteq S$ and a positive element $g \in A_{\Gamma}$ such that $P = g A_X g^{-1}$ and $Q = g  A_Y g^{-1}$. 
\end{lemma}
Note that the parabolic subgroups here may be reducible.
\begin{proof}
    By \ref{thm:cumplido-minl-pos}, there is a positive element $g$ such that $g^{-1} P g$ is a standard parabolic subgroup $A_X$. Let $ab^{-1}$ be the $pn$-normal form for $z_{Q'}$ where $Q' = g^{-1} Q g$.  The maximal shared suffix of $a$ and $b$ is trivial by definition, and suffix order is invariant under right multiplication, so the maximal shared suffix of $z_{Q'}$ and the identity is $b^{-1}$. By Theorem \ref{thm:cumplido-minl-pos}, $b^{-1} Q' b = A_Y$ for some $Y \subseteq S$.  
    
    Commutativity is preserved by conjugation, so $z_X$ and $z_{Q'}$ commute. Furthermore, the element $z_X$ is positive by construction. Thus $z_{Q'} z_X z_{Q'}^{-1}$ and $1 z_X 1^{-1}$ are both positive. By Lemma 5.8 of \cite{C-parab-definition}, this implies that conjugating $z_X$ by the maximal shared suffix of $z_{Q'}$ and the identity is also positive, i.e., $b^{-1} z_X b$ is positive. By Corollary 6.5 of \cite{C-parab-definition}, there is an $X' \subseteq S$ such that $b^{-1} A_X b = A_{X'}$.

    In summary, $gb$ is a positive element such that $P = g b A_{X'} b^{-1} g^{-1}$ and $Q = gb A_Y b^{-1} g^{-1}$.
\end{proof}

In the special case where $P$ and $Q$ are not nested, we obtain the following immediate corollary.
\begin{corollary}\label{cor:commuting-to-reducible}
    Let $P$ and $Q$ be two parabolic subgroups such that $z_P$ and $z_Q$ commute but $P \not\leq Q$ and $Q \not\leq P$. Then $\langle P, Q \rangle$ is a reducible parabolic subgroup whose irreducible components are precisely the irreducible components of $P$ and $Q$. 
\end{corollary}
 
    
    For the purposes of this paper, we will require the still stronger statement that every simplex in $C_{parab}$ is simultaneously standardizable. This is a relatively straightforward generalization of the proof of Theorem 2.2 in \cite{C-parab-definition}, but we prove it here because it is of critical importance in the remainder of the paper. To do this, we require the following lemma.

        \begin{lemma}\label{minimal standardizer inside parabolic}
            If $g$ is the minimal positive standardizer of a parabolic subgroup $g A_Y g^{-1} = P$ such that $P \leq A_{X}$ for some standard parabolic subgroup $A_X$, then $g \in A_X$. 
        \end{lemma}
        \begin{proof}
            First, notice that there is some $g' \in A_X$ which standardizes $g A_Y g^{-1}$. This follows from the fact that if a parabolic subgroup $P$ of $A_{\Gamma}$ is contained in $A_{X}$, then it is also a parabolic subgroup of the finite-type Artin group $A_{X}$ \cite{Godelle-spherical-normalizers,Paris-subgroup-conjugacy}. One can then apply the existence of minimal standardizers in $A_X$ to obtain a minimal positive $g' \in A_{X}$ which standardizes the subgroup. 

            The minimal positive standardizer is a prefix of any other positive element which standardizes $P = g A_Y g^{-1}$; in particular $g \preccurlyeq g'$. Since $g$ and $g'$ are both positive, this implies that $\text{supp}(g) \subseteq \text{supp}(g') \subseteq X$, which implies $g \in A_X$.
        \end{proof}

        We will also need to introduce some terminology. Let $\{P_i\}$ span a simplex $\Sigma$ in $C_{parab}$. A subset of the collection $\{P_i\}$ is called a \emph{nesting chain} if it can be ordered into a chain of proper subgroup inclusions $P_{i_k} < \cdots < P_{i_1}$. A nesting chain is called \emph{complete} if there are no other vertices of $\Sigma$ which can be added. If some $P_i$ is minimal (respectively, maximal) in any nesting chain containing it, we say that $P_i$ is \emph{minimal} (respectively, maximal) in the simplex. 
        
        In general, both minimality and maximality are dependent on the other vertices in the simplex. If $\Sigma$ is a maximal simplex, however, each minimal element will be a conjugate of a cyclic subgroup generated by a single Artin generator. These subgroups are minimal in any simplex containing them. If some larger parabolic subgroup $g A_X g^{-1}$ were minimal, then the conjugate by $g$ of any Artin generator in $X$ would be in the link of $\Sigma$, which contradicts maximality. We will also need to reference the positions of vertices which are neither maximal nor minimal in complete nesting chains.

        \begin{definition}
            Let $\Sigma$ be a maximal $C_{parab}$ simplex with vertex set $\{P_i\}$. A particular vertex $P$ is said to be \emph{in the $k$th level} of $\Sigma$ if there is some complete nesting chain containing $P$ where $k-1$ elements of the chain properly contain $P$. The collection of all such vertices is called the \emph{$k$th level} of $\Sigma$.
        \end{definition}

    As an example, the first level of any simplex is the set of maximal elements. Notice that every vertex of $\Sigma$ is contained in some level. In fact, the levels of $\Sigma$ form a partition of the vertex set.
        
        \begin{lemma}
            Let $P$ be a parabolic subgroup, and let $\Sigma$ be a maximal $C_{parab}$ simplex containing it. Any two complete nesting chains containing $P$ differ only on the parabolic subgroups which are properly contained in $P$.
        \end{lemma}
        \begin{proof}
            Suppose $P$ is in both the $k$th and $j$th levels of a maximal $C_{parab}$ simplex $\Sigma$, i.e., $P$ appears at the $k$th position in one maximal nesting chain and the $j$th position in another. 
            
            Consider the parabolic subgroups $P_{k-1}$ and $P'_{j-1}$ at the $(k-1)$st and $(j-1)$st positions in the two nesting chains. The intersection of $P_{k-1}$ and $P'_{j-1}$ is nontrivial, since it contains $P$. Since $P_{k-1}$ and $P'_{j-1}$ are adjacent in $C_{parab}$, this implies that either they are equal or one properly contains the other. 
            
            Proper containment would contradict completeness of one of the nesting chains, so they must be equal. Applying the same argument at each step of the two nesting chains shows that in fact, they must be the same nesting chain.            
        \end{proof}
        \begin{corollary}
            The levels of a maximal $C_{parab}$ simplex $\Sigma$ are non-intersecting.
        \end{corollary}
        
    \begin{example}
        Let $A_{\Gamma}$ be the Artin group of type $E_6$ with the labelling given in Figure \ref{fig:defgraphs}. The collection $\{ \langle s_1 \rangle, \langle s_1, s_2 \rangle, \langle s_4 \rangle, \langle s_5, s_6 \rangle, \langle s_6 \rangle \}$ spans a simplex $\Pi$ in $C_{parab}$. We will see later in Proposition \ref{maximal simplex characterization} that $\Pi$ is maximal.
        
        The complete nesting chains in $\Pi$ are the pairs $\{ \langle s_1 \rangle, \langle s_1, s_2 \rangle\}$ and $\{ \langle s_6 \rangle,$ $\langle s_5, s_6 \rangle\}$ as well as the singleton $\{\langle s_4 \rangle \}$. The minimal elements are $\langle s_1\rangle$, $\langle s_4 \rangle$, and $\langle s_6 \rangle$. The maximal elements are $\langle s_1, s_2 \rangle$, $\langle s_4 \rangle$, and $\langle s_5, s_6 \rangle$. The partition into levels is as follows.

        Level 1: $\{ \langle s_1, s_2 \rangle, \langle s_4 \rangle, \langle s_5, s_6 \rangle \}$

        Level 2: $\{ \langle s_1\rangle, \langle s_6 \rangle \}$
        
    \end{example}


    \begin{example}
        Let $A_{\Gamma}$ be of type $B_n$ with the labeling in Figure \ref{fig:defgraphs}. The collection $\{ \langle s_1 \rangle, \langle s_1, s_2 \rangle, \cdots \langle s_1, \cdots s_{n-1} \rangle\}$ spans a simplex $\Sigma$ in $C_{parab}$. We will see later in Proposition \ref{maximal simplex characterization} that $\Sigma$ is maximal.

        The simplex $\Sigma$ has only one complete nesting chain: the entire vertex set of $\Sigma$. The minimal element is $\langle s_1 \rangle$, and the maximal element is $\langle s_1, \cdots s_{n-1} \rangle$. The $k$th level is $\langle s_1 , \cdots, s_{n-k}\rangle$.
    \end{example}

    We can now construct a positive standardizing element for a $C_{parab}$-simplex of dimension larger than $1$. The following is Proposition \ref{propx:sim-std} of the introduction.
            
        \begin{proposition}\label{simplex sim standard}
            Let $\{P_i\}$ span a simplex in $C_{parab}$. There is a positive element $g$ such that $g^{-1} P_i g$ is standard for every $P_i$.
        \end{proposition}

        \begin{proof}
            To begin, we relabel the parabolic subgroups in the simplex as $P_{(L, i)}$ where $L$ denotes a particular level, and $i$ indexes the parabolic subgroups within that level. We prove the statement by constructing a suitable $g$, as follows.
            \begin{enumerate}
                \item If the first level contains a unique subgroup, then we set $P_{(1,1)} = P_{(1,-)}$, and we proceed to the final paragraph of this step. Otherwise, consider the subgroups $P_{(1,1)}$ and $P_{(1,2)}$. These subgroups are disjoint and commute with one another, so by Corollary \ref{cor:commuting-to-reducible}, $P_{(1,1)}$ and $P_{(1,2)}$ are the irreducible components of a reducible parabolic subgroup $P_{(1,1-2)}$. If there is no $P_{(1,3)}$, we proceed to the next step in the construction. If a subgroup $P_{(1,3)}$ does exist, then necessarily $z_{P_{(1,3)}}$ commutes with $z_{P_{(1,1-2)}} = z_{P_{(1,1)}}z_{P_{(1,2)}}$. By Corollary \ref{cor:commuting-to-reducible}, $P_{(1,1)}$, $P_{(1,2)}$, and $P_{(1,3)}$ form the irreducible components of a single reducible parabolic subgroup $P_{(1,1-3)}$.
                
                We can repeat this process on the remaining elements in level 1 of the simplex to obtain a single reducible parabolic subgroup $P_{(1,-)}$ whose irreducible components are precisely the $P_{(1,i)}$. Note that this implies there must be finitely many such maximal elements because there are finitely many Artin generators. 
                
                Every parabolic subgroup has a minimal positive standardizer, so choose $g_1$ to be the minimal positive standardizer of $P_{(1,-)}$. Let $X$ be the subset of the Artin generating set such that $P_{(1,-)} = g_1 A_{X} g_1^{-1}$.
                
                \item Now consider the subset $\{P_{(2,i)}\}$ in the second level. Conjugating by $g_1$ still results in a simplex, so by the same reasoning as above, $\{g_1 ^{-1} P_{(2,i)}g_1\}$ are the irreducible components of a single reducible parabolic subgroup $g_1^{-1} P_{(2,-)} g_1$. By construction, $g_1^{-1} P_{(2,-)} g_1$ is a parabolic subgroup of $A_{X} = g_1 ^{-1} P_{(1,-)} g_1$. Let $g_2$ denote the minimal positive standardizer for $g_1^{-1} P_{(2,-)}g_1$.
                
                By Lemma \ref{minimal standardizer inside parabolic}, the minimal positive standardizer $g_2$ is contained in the standard reducible parabolic subgroup $A_X$. This implies that the positive element $g_1 g_2$ standardizes the elements of $\{P_i\}$ in both the first and second levels.

                \item Iterate this process to obtain a positive element $g_1 \cdots g_k$ which standardizes the first $k$ levels of the collection. Since there are finitely many standard parabolic subgroups, this process must terminate at some finite $K$. The desired element is thus $g = g_1 \cdots g_K$.
            \end{enumerate} 
        \end{proof}
\begin{figure}[h]
                \centering
                \includegraphics[width=8cm]{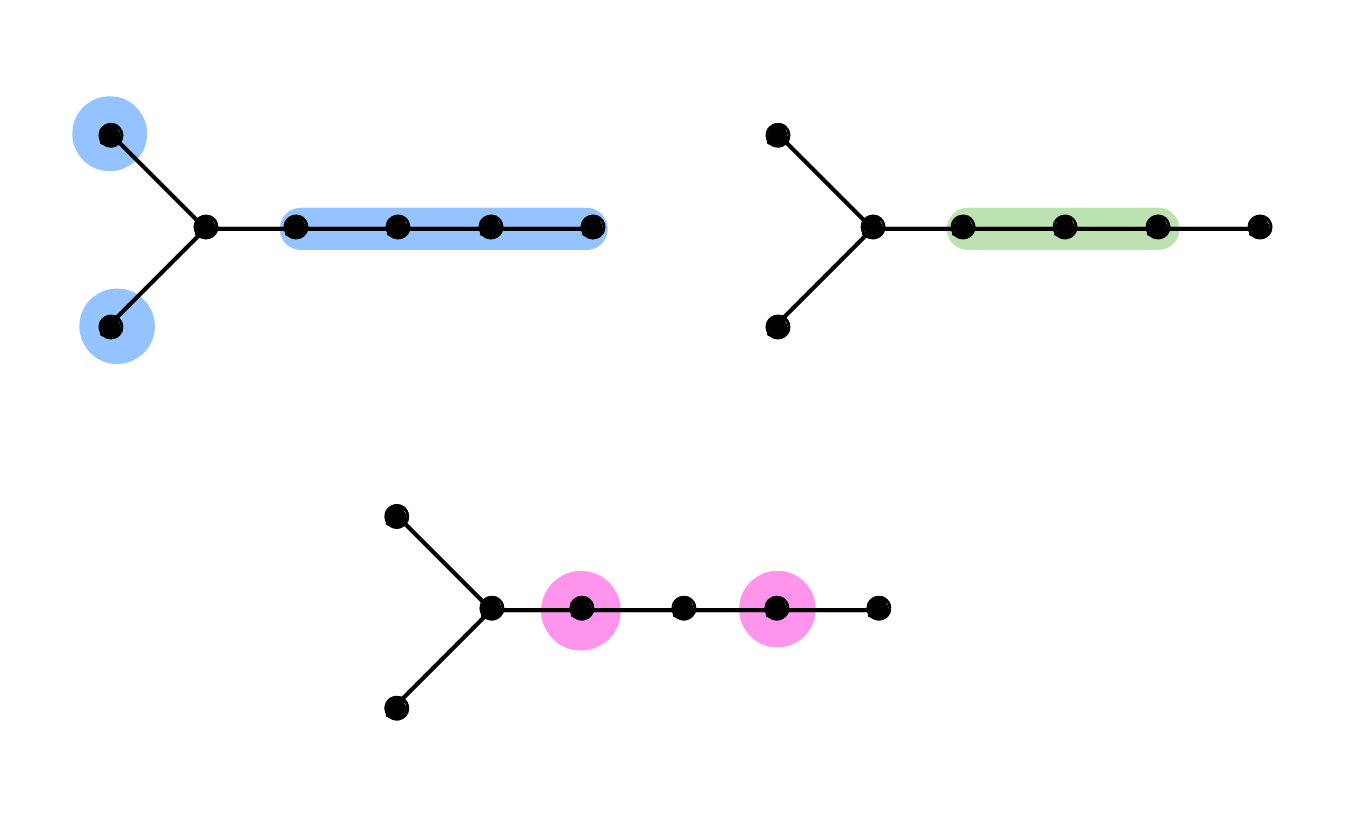}
                \captionsetup{margin=.5cm,justification=centering}
                \caption{Possible maximal subgroups at steps \textcolor{blue}{1}, \textcolor{green}{2}, and \textcolor{pink}{3} of the standardizer construction}
            \end{figure}
        
    \begin{corollary}
            The complex of irreducible parabolic subgroups is finite-dimensional.
        \end{corollary}

    Notice that the element $g$ constructed in this way is unique for each simplex $\{P_i\}$. We call it the \emph{canonical positive standardizer} for the simplex $\{P_i\}$, and we often write $\underline{g}$ to distinguish it from other standardizing elements. We refer to the collection $\{X_i\}$ with $P_i = \underline{g} A_{X_i} \underline{g}^{-1}$ as the \emph{canonical positive standardization}. We use the descriptor ``canonical'' rather than ``minimal'' because we will not determine whether this element is a prefix of other positive elements which standardize the simplex. Any canonical way of choosing a standardizing element is sufficient for our purposes. The canonical positive standardizer of a simplex is unique only in the sense that it is the unique standardizing element obtained \emph{via this construction}. There are many other standardizers for a simplex. Some of these standardizers may send the simplex to different standardizations. We give one example which arises frequently below.

        \begin{example}
            Let $\{P_i\}$ be a $C_{parab}$ simplex, and let $\{g A_{X_i} g^{-1}\}$ be a choice of standardization. For any $i$ and any power $k$, the element $g \Delta_{X_i}^k$ is also a standardizer for the collection, but we may have $g A_{X_j} g^{-1} = g \Delta_{X_i}^k A_{X'_j} \Delta_{X_i}^{-k} g^{-1}$ for some standard parabolic subgroup $A_{X'_j}$ which was not in the collection $\{A_{X_j}\}$.

            For example, let $A_{\Gamma}$ be of type $B_6$, and let $\{P_i\} = \{\langle s_1, s_2 \rangle , \langle s_4, s_5, s_6 \rangle, $ $\langle s_5, s_6 \rangle\}$. Here the standardizing element $g$ can be taken to be the trivial element. We could also choose $\Delta_{456}^{2}$. Since $\Delta_{456}^{2}$ is central in $\langle s_4, s_5, s_6 \rangle$ and commutes with $\langle s_1, s_2 \rangle$, conjugation by $\Delta_{456}^2$ normalizes each parabolic subgroup in the simplex.

            The element $\Delta_{456}$ does not normalize $\langle s_5, s_6 \rangle$, but it is still a standardizer for the simplex. Because $\langle s_5, s_6 \rangle \leq \langle s_4, s_5, s_6 \rangle$, conjugation by $\Delta_{456}^{-1}$ sends $\langle s_5, s_6 \rangle$ to some standard parabolic subgroup in $\langle s_4, s_5, s_6 \rangle$. In this case, $\Delta_{456}^{-1} \langle s_5, s_6 \rangle \Delta_{456} = \langle s_4, s_5 \rangle$. Since $\Delta_{456}$ normalizes both $\langle s_1, s_2 \rangle$ and $\langle s_4, s_5, s_6 \rangle$, the simplex can also be written in the following form.
            \begin{align*}
                \langle s_1, s_2 \rangle &= \Delta_{456} \langle s_1, s_2 \rangle \Delta_{456}^{-1}\\
                \langle s_4, s_5, s_6 \rangle &= \Delta_{456} \langle s_4, s_5, s_6 \rangle \Delta_{456}^{-1}\\
                \langle s_5, s_6 \rangle &= \Delta_{456} (\Delta_{456}^{-1} \langle s_5, s_6 \rangle \Delta_{456})\Delta_{456}^{-1}\\
                &= \Delta_{456} \langle s_4, s_5 \rangle\Delta_{456}^{-1}
            \end{align*}
            Thus $\Delta_{456}$ is a simultaneous standardizer for the simplex.
        \end{example}

        We now have the tools to describe the maximal simplices of $C_{parab}$ concretely. Since all simplices are simultaneously standardizable, it is no loss of generality to describe the maximality of simplices spanned by standard parabolic subgroups.

        \begin{lemma}\label{3 maxl base components}
            Let $\{A_{X_i}\}$ be a maximal simplex in $C_{parab}$. Each maximal element of $\{X_i\}$ is a connected component of $\Gamma - \{v\}$ for some vertex $v \in \Gamma$. In particular, there are at most as many maximal elements as the maximal valence of a vertex in $\Gamma$.
        \end{lemma}
        \begin{proof}
            Let $A_X$ be a maximal element in the simplex. This is a proper parabolic subgroup, so $X$ cannot contain all vertices of $\Gamma$. Let $v \in \Gamma - \{x\}$ in $\Gamma$ be adjacent to some $x \in X$. Since $X$ is irreducible and does not contain $v$, it is contained in a single connected component of $\Gamma - \{v\}$.

            Suppose there is some other maximal element $A_Y$. Since $A_Y$ commutes with $A_X$ and $v$ is adjacent to $X$, $A_Y$ is an irreducible parabolic subgroup which does not contain $v$. Thus $A_Y$ is in a connected component of $\Gamma - \{v\}$. Consequently, every element of $\{A_{X_i}\}$ is contained in a connected component of $\Gamma - \{v\}$, since each is contained in some maximal element.

            We now claim that $X$ must coincide with the connected component containing it. The connected components of $\Gamma - \{v\}$ define commuting parabolic subgroups, and every $A_{X_i}$ is contained in one of these, so the parabolic subgroup defined by a connected component of $\Gamma - \{v\}$ either contains or commutes with each $A_{X_i}$. It must be contained in the simplex by maximality of the simplex, and it must coincide with $A_X$ by maximality of $A_X$ within the simplex. Applying the same argument to each maximal $A_Y$ completes the proof.

        \end{proof}
        Since no finite-type defining graph has a vertex of valence more than 3, we immediately obtain the following corollary.
        \begin{corollary}\label{cor:3-max-cpts}
            Any maximal simplex in $C_{parab}$ has at most 3 maximal base elements.
        \end{corollary}
    Throughout, we will implicitly apply \cref{cor:3-max-cpts} when referring to the maximal elements of a $C_{parab}$ simplex for ease of notation, i.e., we will often simply say ``Let $P_1$, $P_2$, and $P_3$ be the maximal components of a maximal $C_{parab}$-simplex $\Sigma$.'' Unless otherwise specified, up to two of the named components $P_1$, $P_2$, and $P_3$ in such phrasing may be empty. 

    The following is Proposition \ref{propx:maxl-characterization} of the introduction.

        \begin{proposition}\label{maximal simplex characterization}
            Let $\{A_{X_i}\}$ span a simplex $\Sigma$ in $C_{parab}$. The simplex $\Sigma$ is maximal precisely when the following hold.
            \begin{itemize}
                \item $\bigcup_{i} X_i = V(\Gamma) - \{t\}$ for some Artin generator $t$.
                \item For every $A_{X_i}$ in the simplex, the union of all $X_j \subsetneq X_i$ is equal to $X_i - \{t_i\}$ for some Artin generator $t_i$ contained in $X_i$.
            \end{itemize}
        \end{proposition}
        The generator $t$ in the first condition depends on the parabolic subgroups which make up the first level of the simplex. If $A_{X_i}$ is at level $k$, the generator $t_i$ depends on the parabolic subgroups in the simplex which are at level $k+1$ and are contained in $A_{X_i}$.
        \begin{proof}
            The first condition is necessary by Lemma \ref{3 maxl base components}. The same statement applied to the induced subgraph of $\Gamma$ spanned by each $X_i$ shows that the second condition is necessary.
            


            Now, we check that these conditions are sufficient. First, notice that if both conditions hold, then there is no \emph{standard} parabolic subgroup which can be added to such a collection. We cannot add a new maximal element because it would have to contain the missing Artin generator $t$, and we have already seen that this does not define a simplex. Similar reasoning shows that we cannot add a standard parabolic subgroup contained in any $X_i$. 

            If there were some irreducible parabolic subgroup $P$ in the link of $\Sigma$, then it could be simultaneously standardized. Paris's algorithm for checking conjugacy of standard parabolic subgroups implies that two standard parabolic subgroups can only be conjugate if they contain the same number of generators \cite{Paris-subgroup-conjugacy}. In particular, any other standardization of a collection $\{A_{X_i}\}$ with the properties in the statement will have the same properties. There are no suitable standard parabolic subgroups conjugate to $P$, so there is no such $P$, and $\Sigma$ is maximal.
        \end{proof} 
        The following is an immediate corollary.
        \begin{corollary}
            A maximal simplex in $C_{parab}(A_{\Gamma})$ has dimension $n-2$, where $n = |V(\Gamma)|$.
        \end{corollary}
        \begin{proof}
            It suffices to consider standard simplices $\{A_{X_i}\}$. Let $v$ be the unique vertex which is not contained in any $X_i$. Proposition \ref{maximal simplex characterization} shows there is a bijection between the vertices of $\Gamma - \{v\}$ and parabolic subgroups in $\{A_{X_i}\}$ given by sending $w \in \Gamma - \{v\}$ to the minimal subgroup $A_{X_i}$ containing it, so $\{A_{X_i}\}$ has $n-1$ vertices.
        \end{proof}

\subsection{Ribbons}\label{sec:ribbons}
In much of what follows, a key point will be understanding normalizers of certain parabolic subgroups of finite-type Artin groups. Much is known about these normalizers already: Godelle showed in \cite{Godelle-spherical-normalizers} that the normalizer of a standard parabolic subgroup $A_X$ is the semidirect product $H_X \ltimes A_X$, where $H_X$ is a subgroup consisting of special elements called \emph{$X$-ribbons-$X$}; see Definition \ref{ribbon def}. For most choices of $A_X$, there are many such ribbons. When $X$ includes almost all of the Artin generators, however, there are relatively few ribbons. We devote this subsection to classifying the possibilities in that special case, which will prove useful for understanding vertex stabilizers in our marking graph.

    We adopt the notation that $x^{\square}$ can represent $x^k$ for any integer $k$, since the particular exponent will often be unimportant for our purposes. The main goal of this subsection is to prove the following.

    \begin{proposition}\label{product of deltas}
        Let $X \subsetneq S$ be such that $X \cup \{t\} = S$. Every $X$-ribbon-$X$ can be written as a product of the form $\Delta_{X_1}^{\square} \Delta_{X_2}^{\square} \Delta_{X_3}^{\square} \Delta_{\Gamma}^{\square}$, where $X_i$ are the indecomposable components of $X$, up to two of which may be empty.
    \end{proposition}

    We first recall the precise definition of $X$-ribbons-$X$. The following definition is due to Godelle, but the same collection of elements was defined earlier by Paris in \cite{Paris-subgroup-conjugacy}, where they were called $(X,X)$-conjugators.
    
    \begin{definition}[\cite{Godelle-spherical-normalizers} Definition 0.4]\label{ribbon def}
        Let $A_{\Gamma}$ be an irreducible finite-type Artin group with Artin generating set $S$. Let $X \subseteq S$ and $t \in S$, and let $X(t)$ be the indecomposable component of $A_{X \cup \{t\}}$ containing $t$. If $t \not\in X$, define $d_{X,t} = \Delta_{X(t)} \Delta^{-1}_{X(t) - \{t\}}$. If $t \in X$, define $d_{X,t} = \Delta_{X(t)}$. In either case, there is a unique subset $Y$ of $X \cup \{t\}$ such that $Y = d_{X,t} X d_{X,t}^{-1}$, and we say that $d_{X,t}$ is a \emph{positive elementary $Y$-ribbon-$X$}. 

        For $X, Y \subseteq S$, we say that $g \in A_{\Gamma}^{+}$ is a \emph{positive $Y$-ribbon-$X$} if $g = g_n \cdots g_1$ where $g_i$ is a positive elementary $X_i$-ribbon-$X_{i-1}$ with $X_0 = X$ and $X_n = Y$. If we allow inverses of positive elementary $X_{i-1}$-ribbons-$X_i$ in a chain of the above form, we obtain a \emph{$Y$-ribbon-$X$}.
    \end{definition}
    Observe that while $g A_X g^{-1} = A_Y$ in general only implies that $g z_X g^{-1} = z_Y$, an elementary $Y$-ribbon-$X$ sends the \emph{generating set} $X$ to $Y$ and thus conjugates $\Delta_X$ to $\Delta_Y$.

    When $t \not\in X$ is adjacent to $X$ in $\gamma$, the element $d_{X,t}$ is a positive element of the Artin group, as noted in Definition 4.1 of \cite{C-parab-definition}.

    \begin{example}
        Let $A_{\Gamma}$ be of type $B_n$. First, let $s_1$ be the Artin generator which is connected to $s_2$ via a 4-labeled edge and which commutes with all other Artin generators as in Figure \ref{fig:defgraphs}. Let $X = \{s_1\}$. The element $d_{X, s_2}$ is $\Delta_{12} s_1^{-1}$. In a dihedral group with even edge label, the Garside element is central. Thus $d_{X,s_2} s_1 d_{X,s_2}^{-1} = s_1$, so $d_{X,s_2}$ is a positive elementary $X$-ribbon-$X$.

        Now let $Y = \{s_2\}$, and let $s_3$ be the Artin generator connected to $s_2$ by a 3-labeled edge. By the same reasoning as above, the element $d_{Y, s_1}$ is an elementary $Y$-ribbon-$Y$. The Garside element of the dihedral braid group $\langle s_2, s_3 \rangle$ is $s_2 s_3 s_2 = s_3 s_2 s_3$, so the ribbon $d_{Y,s_3}$ is $s_2 s_3$, and 
        \begin{align*}
            d_{Y, s_3} s_2 d_{Y, s_3}^{-1} &= s_2 s_3 s_2 s_3^{-1} s_2^{-1}\\
            &= s_3 s_2 s_3 s_3^{-1} s_2^{-1}\\
            &= s_3\text{.}
        \end{align*}
        The element $d_{Y,s_3}$ is a positive elementary $\{s_3\}$-ribbon-$\{s_2\}$. The product $d_{Y, s_3} d_{Y, s_1}$ is an $\{s_3\}$-ribbon-$\{s_2\}$ with $X_0 = X_1 = \{s_2\}$ and $X_2 = \{s_3\}$.
    \end{example}

    We will need the following lemma.

    \begin{lemma}\label{lem:ribbon-support}
        Let $A_X$ be an irreducible standard parabolic subgroup, and let $t \in \Gamma - X$ be an Artin generator which is adjacent to some $x \in X$ in $\Gamma$. The support of the positive elementary ribbon $d_{X,t}$ contains the set $X$.
    \end{lemma}
    \begin{proof}
        Let $a \in X$, and suppose $a$ is not in $\text{supp}(d_{X,t})$. Recall that the element $d_{X,t}$ sends Artin generators in $X$ to Artin generators in some subset $Y$. Let $b$ be the Artin generator $d_{X,t} a d_{X,t}^{-1}$. Observe that the element $d_{X,t}$ is non-trivial; it has $t$ as a prefix by Lemma 4.4 of \cite{C-parab-definition}.

        First, suppose $b \neq a$. Since $d_{X,t} a d_{X,t}^{-1} = b$, the supports of both sides of this equation must be the same. The right hand side does not contain $a$ and every expression for the positive element $d_{X,t}a$ does contain $a$, so equality of the supports can occur only if there is cancellation between the positive and negative factors in the left hand side which removes all $a$ terms. This is only possible if $d_{X,t}^{-1}$ contains $a^{-1}$, contradicting our assumption.

        Now suppose $b = a$ so that $d_{X,t} a d_{X,t}^{-1} = a$. This necessarily implies that the Artin group elements $d_{X,t} a$ and $a d_{X,t}$ are equal. These are both positive elements, so they are equal as group elements if and only if they are positive equivalent, meaning one can be transformed into the other by sequential applications of the Artin relations \cite{paris2002monoidinj}.

        If $a$ appears in $d_{X,t}$, we are done, so suppose it does not. There is only one Artin relation where one side is of the form $\cdots ca$ and contains only one $a$: the commuting relation $ca = ac$. In particular, if $d_{X,t}$ contains no $a$ terms but has the property that $d_{X,t} a = a d_{X,t}$, then $a$ must commute with every generator which appears in $d_{X,t}$. This implies that no generator $a' \in \text{Lk}(a) \cap (X \cup \{t\})$, which must be a non-empty set because $X \cup \{t\}$ is connected and strictly larger than $\{a\}$, can be in the support of $d_{X,t}$ either. 

        Repeating this process on each $a' \in \text{Lk}(a) \cap X$ yields a further strict reduction in the potential support for $d_{X,t}$. Since $X \cup\{t\}$ is finite, iterating this process eventually contradicts the fact that $d_{X,t}$ is non-trivial.
    \end{proof}

    It is shown in \cite{C-parab-definition} that if $g$ conjugates $A_X$ to $A_Y$, then $g z_X g^{-1} = z_Y$. We require a slightly stronger result for certain choices of $g$.

    \begin{lemma}\label{Delta conjugation}
        Let $A_X$ be a standard parabolic subgroup of a finite-type Artin group $A_{\Gamma}$, and let $X \subseteq T$. Suppose that $\Delta_T$ conjugates $A_{X}$ to the standard parabolic subgroup $A_{Y}$. Then $\Delta_T$ conjugates $\Delta_{X}$ to $\Delta_{Y}$ and $\Delta_{Y}$ to $\Delta_{X}$.
    \end{lemma}
    \begin{proof}
        Conjugation by $\Delta_T$ induces a permutation of the Artin generators in $A_T$, so if $\Delta_T$ conjugates $A_X$ to $A_{Y}$, then $\Delta_T X \Delta_T^{-1} = Y$. Both Theorem 5.1 in \cite{Paris-subgroup-conjugacy} and the proof of Lemma 2.2 in \cite{Godelle-spherical-normalizers} show that such an element is a product of elementary ribbons. As we observed when we defined ribbons, a product of elementary ribbons which conjugates $X$ to $Y$ sends $\Delta_X$ to $\Delta_Y$. 

        The square of the Garside element is always central. This implies that $\Delta_T Y \Delta_T^{-1} = X$, and the same argument reveals $\Delta_T \Delta_{Y} \Delta_T^{-1} = \Delta_X$.
    \end{proof}

Observe that because $\Delta_T^2$ is always central, the above lemma immediately implies that if $\Delta_T^{-1}$ conjugates $X$ to $Y$, then $\Delta_T^{-1}\Delta_X = \Delta_Y \Delta_{T}^{-1}$. We can now proceed to the proof of Proposition \ref{product of deltas}.
    \begin{proof}
        Choose some $t' \in S$. We start by classifying the possible elements $d_{X,t'}$. If $t' \in X$, then it is in exactly one $X_i$, and $d_{X,t} = \Delta_{X_i}$. In this case, we have $d_{X,t} A_{X} d_{X,t}^{-1} = A_X$. There is only one choice of $t' \not\in X$, and the indecomposable component of $A_{X \cup \{t'\}}$ containing $t'$ is necessarily all of $A_{\Gamma}$. Thus if $t' \not \in X$, then $d_{X, t'} = \Delta_{\Gamma} \Delta_X^{-1} = \Delta_{\Gamma} \Delta_{X_1}^{-1}\Delta_{X_2}^{-1}\Delta_{X_3}^{-1}$.

        If $t' \not\in X$, then there is a subset $Y$ of $S$ such that $Y = d_{X, t'} X d_{X, t'}^{-1}$. Expanding this, we find that \begin{align*}
            Y =  \Delta_{\Gamma} \Delta_X^{-1} X \Delta_X \Delta_{\Gamma}^{-1} = \Delta_{\Gamma} X \Delta_{\Gamma}^{-1}\text{.} \tag{*}
        \end{align*}

        \noindent\textbf{\underline{Case 1:}} If $Y = X$, then we have determined that the only choice of path $\{Z_i\}$ connected via elementary ribbons with $Z_0 = X = Z_n$ is the path with $Z_i = X$ for all $i$. Every elementary positive $Y$-ribbon-$X$ is in fact an elementary positive $X$-ribbon-$X$, and they are precisely the elements $\Delta_{X_i}$ for each $i$ and $\Delta_{\Gamma} \Delta_X^{-1}$. 
        
        To see that this proves the desired result, note that Lemma \ref{Delta conjugation} applied with $\Gamma$ playing the role of $T$ and $X$ playing the role of both $X$ and $Y$ shows that $\Delta_{\Gamma}$ and $\Delta_{X}^{-1}$ in fact commute, so 
        \begin{align*}
           \Delta_{\Gamma} \Delta_X^{-1} = \Delta_{X}^{-1} \Delta_{\Gamma} = \Delta_{X_1}^{-1}\Delta_{X_2}^{-1}\Delta_{X_3}^{-1} \Delta_{\Gamma}\text{.} 
        \end{align*}

        Furthermore, using the descriptions of possible conjugation actions of $\Delta_{\Gamma}$ in finite-type Artin groups and investigation of the defining graphs, it is straightforward to check that either $\Delta_{\Gamma}$ fixes $X$ pointwise or it permutes two connected components of $X$ and fixes the third (if it is nonempty). Then a product of these elements and their inverses can be placed in the desired form by simply gathering shifting all $\Delta_{\Gamma}$ factors to the right and gathering like terms.
        
        If $\Delta_{\Gamma}$ permutes $X_1$ and $X_2$ and fixes $X_3$, for example, then $\Delta_{\Gamma} \Delta_{X_1} = \Delta_{X_2} \Delta_{\Gamma}$ and $\Delta_{\Gamma} \Delta_{X_2} = \Delta_{X_1} \Delta_{\Gamma}$ and similarly for their inverses by Lemma \ref{Delta conjugation}, so $(\Delta_{\Gamma}\Delta_{X}^{-1})^{-1} \Delta_{X_1}$ can be written in the desired form as follows.
        \begin{align*}
            \Delta_{X} \Delta_{\Gamma}^{-1} \Delta_{X_1} &= \Delta_{X} \Delta_{X_2} \Delta_{\Gamma}^{-1} \\
            &= \Delta_{X_1} \Delta_{X_2} \Delta_{X_3} \Delta_{X_2} \Delta_{\Gamma}^{-1}\\
            &= \Delta_{X_1} \Delta_{X_2}^{2} \Delta_{X_3}\Delta_{\Gamma}^{-1}\text{.}
        \end{align*}
        

        \noindent\textbf{\underline{Case 2:}} Now, suppose $Y \neq X$. Note that this is only possible if $\Delta_{\Gamma}$ is not central. To build a path $\{Z_i\}$ with $Z_0 = Z_n = X$, we have exactly two choices for each $Z_i$: $Y$ or $X$. To see this, note that any $t' \in Y$ will give a $d_{Y,t'}$ which is contained in $Y$. If $t'$ is the unique generator not contained in $Y$, then 
        \begin{align*}
            d_{Y,t'} Y d_{Y,t'}^{-1} = \Delta_{\Gamma} \Delta_{Y}^{-1} Y \Delta_{Y}\Delta_{\Gamma}^{-1} = \Delta_{\Gamma} Y \Delta_{\Gamma}^{-1} = \Delta_{\Gamma}^{2} X \Delta_{\Gamma}^{-2}= X
        \end{align*}
        since $\Delta_{\Gamma}^2$ is always central. The same holds for $d_{Y,t'}^{-1}$.

        If we choose $Z_1 = Y$ and $Z_2 = X$, then a positive ribbon with this path is the product $\Delta_{\Gamma} \Delta_{Y}^{-1} \Delta_{\Gamma} \Delta_{X}^{-1}$. Applying Lemma \ref{Delta conjugation} to (*) yields $\Delta_{\Gamma} \Delta_{X}^{-1} = \Delta_{Y}^{-1} \Delta_{\Gamma}$, so 
        \begin{align*}
            \Delta_{\Gamma} \Delta_{Y}^{-1} \Delta_{\Gamma} \Delta_{X}^{-1} = \Delta_{\Gamma} (\Delta_{\Gamma} \Delta_{X}^{-1}) \Delta_X^{-1} = \Delta_{\Gamma}^2 \Delta_{X}^{-2}\text{.}
        \end{align*}
        
        If we choose $Z_2 = Y$ and choose a ribbon $d_{Y,t'} = \Delta_{Y_i}$ for some indecomposable component $Y_i$ of $Y$, then there is some indecomposable component $X_i$ of $X$ such that $\Delta_{\Gamma} \Delta_{X_i} = \Delta_{Y_i} \Delta_{\Gamma}$. Then the first two steps of the ribbon, $\Delta_{Y_i} \Delta_{\Gamma} \Delta_{X}^{-1}$, can be replaced with $\Delta_{\Gamma} \Delta_{X_i} \Delta_{X}^{-1}$. 

        A similar argument shows that terms can be rearranged if we take inverses of positive ribbons. All possible combinations are listed below. 

        \begin{center}
\renewcommand{\arraystretch}{1.4}
    \begin{tabular}{| c | c |}
    \hline
    $Z_1 = Z_2 = X$ & $Z_1 = X$ and $Z_2 = Y$ \\ 
    \hline
    $\begin{aligned}
            &\Delta_{X_i}^{\pm 1} \Delta_{X_j}^{\pm 1} \\
            &\Delta_{X_i}^{\pm 2}\\
            & 1
        \end{aligned}$ &  $\begin{aligned}
            \Delta_{\Gamma} \Delta_X^{-1} \Delta_{X_i} &= \Delta_{\Gamma}\Delta_{X_j}^{-1}\Delta_{X_k}^{-1}\\
            \Delta_{\Gamma} \Delta_X^{-1} \Delta_{X_i}^{-1} &= \Delta_{\Gamma}\Delta_{X_i}^{-2}\Delta_{X_j}^{-1}\Delta_{X_k}^{-1}\\
            \Delta_Y \Delta_{\Gamma}^{-1} \Delta_{X_i} &= \Delta_{\Gamma}^{-1}\Delta_{X_i}^{2}\Delta_{X_j}\Delta_{X_k}\\
            \Delta_Y \Delta_{\Gamma}^{-1} \Delta_{X_i}^{-1} &= \Delta_{\Gamma}^{-1}\Delta_{X_j}\Delta_{X_k}
        \end{aligned}$  \\  [2pt]
    \hline
    $Z_1 = Y$ and $Z_2  = Y$ & $Z_1 = Y$ and $Z_2 = X$\\
    \hline
    $\begin{aligned}
            \Delta_{Y_i} \Delta_Y \Delta_{\Gamma}^{-1} &= \Delta_{\Gamma}^{-1}\Delta_{X_i}^{2}\Delta_{X_j} \Delta_{X_k} \\
            \Delta_{Y_i}^{-1} \Delta_Y \Delta_{\Gamma}^{-1} &= \Delta_{\Gamma}^{-1}\Delta_{X_j} \Delta_{X_k}\\
            \Delta_{Y_i} \Delta_{\Gamma} \Delta_{X}^{-1} &= \Delta_{\Gamma}\Delta_{X_j}^{-1} \Delta_{X_k}^{-1} \\
            \Delta_{Y_i}^{-1} \Delta_{\Gamma} \Delta_{X}^{-1} &= \Delta_{\Gamma}\Delta_{X_i}^{-2}\Delta_{X_j}^{-1} \Delta_{X_k}^{-1}
        \end{aligned}$ & $\begin{aligned}
            \Delta_{\Gamma}\Delta_Y^{-1}\Delta_{\Gamma}\Delta_X^{-1} &= \Delta_X^{-2}\Delta_{\Gamma}^2\\
            \Delta_{\Gamma}\Delta_Y^{-1} \Delta_Y \Delta_{\Gamma}^{-1} &= 1\\
            \Delta_{X}\Delta_{\Gamma}^{-1} \Delta_{\Gamma} \Delta_{X}^{-1} &= 1\\
            \Delta_X \Delta_{\Gamma}^{-1} \Delta_Y \Delta_{\Gamma}^{-1} &= \Delta_X^{2}\Delta_{\Gamma}^{-2}
        \end{aligned}$\\
    \hline
\end{tabular}
\end{center}

Notice that a $\Delta_{\Gamma}$ term is introduced \emph{precisely} when $Z_i$ is distinct from $Z_{i-1}$. In particular, the exponent of the $\Delta_{\Gamma}$ term in the expression is even if we have circled back to $X$ and odd if we have ended at $Y$. 

If we wish to left-multiply by another elementary ribbon, then, there are two possibilities. First, we consider the possibility that our existing element conjugates $X$ to $Y$. In this case, the parity of the exponent of the $\Delta_{\Gamma}$-term is odd. The possible elementary ribbons we could multiply by are $\Delta_{Y_i}^{\pm 1}$, $\Delta_{\Gamma}\Delta_Y^{-1}$, and $\Delta_X \Delta_{\Gamma}^{-1}$. For odd $k$, $\Delta_{Y_i}^{\pm 1} \Delta_{\Gamma}^k = \Delta_{\Gamma}^k \Delta_{X_i}^{\pm 1}$ and $\Delta_{\Gamma}\Delta_{Y}^{-1} \Delta_{\Gamma}^k = \Delta_{\Gamma}^{k+1} \Delta_{X}^{-1}$, so one can perform the same kind of rearrangements as in the chart. Likewise, $\Delta_X \Delta_{\Gamma}^{-1} \Delta_{\Gamma}^k = \Delta_X \Delta_{\Gamma}^{k-1} = \Delta_{\Gamma}^{k-1} \Delta_X$. Notice also that $k$ remains odd if we have multiplied by a $Y$-ribbon-$Y$ and changes parity if we have multiplied by an elementary $X$-ribbon-$Y$. 

If instead our existing ribbon conjugated $X$ to $X$, then $k$ is even, and we have $\Delta_{X_i}^{\pm 1} \Delta_{\Gamma}^k = \Delta_{\Gamma}^k \Delta_{X_i}^{\pm 1}$, $\Delta_{\Gamma}\Delta_{X}^{-1}\Delta_{\Gamma}^k = \Delta_{\Gamma}^{k+1} \Delta_X$, or $\Delta_{Y}\Delta_{\Gamma}^{-1} \Delta_{\Gamma}^k = \Delta_{\Gamma}^{k-1} \Delta_X$. Again, $k$ is unchanged if we have multiplied by an $X$-ribbon-$X$ and changes parity if we have multiplied by an elementary $Y$-ribbon-$X$.

Thus our ribbon is of the form $\Delta_{\Gamma}^k \Delta^{\square}_{X_i} \Delta_{X_j}^{\square} \Delta_{X_k}^{\square} \Delta_X^{\square}$, and if it is an $X$-ribbon-$X$, $k$ is even. Since $\Delta_{\Gamma}^2$ is central, 
\begin{align*}
    \Delta_{\Gamma}^k \Delta^{\square}_{X_i} \Delta_{X_j}^{\square} \Delta_{X_k}^{\square} \Delta_X^{\square} = \Delta^{\square}_{X_i} \Delta_{X_j}^{\square} \Delta_{X_k}^{\square} \Delta_X^{\square} \Delta_{\Gamma}^k\text{.}
\end{align*}
Since $\Delta_X$ itself is a product of $\Delta_{X_i}$s and the $\Delta_{X_i}$s commute with one another, this shows that all $X$-ribbons-$X$ are of the desired form.

    \end{proof}

    The above results also allow us to prove the following useful fact about simplices in $\{C_{parab}\}$.

    \begin{lemma}\label{changing standardization}
        Let $\{P_i\}$ be a maximal simplex in $C_{parab}$. Suppose that $\{A_{X_i}\}$ and $\{A_{Y_i}\}$ are two standardizations of $\{P_i\}$. There is a positive element $r$ such that $\{rA_{X_i} r^{-1}\}= \{A_{Y_i}\} $ and such that $r$ is a product of the form $\Delta_{Y_k}^{\square} \cdots \Delta_{Y_1}^{\square}\Delta_{\Gamma}^{\square} = \Delta^{\square}_{\Gamma} \Delta^{\square}_{X_1} \cdots \Delta^{\square}_{X_k}$ where $\square$ may denote 0 or 1.
    \end{lemma}
    The proof is by constructing a suitable $r$. The order of the $\Delta_{Y_i}$ in the element $r$ we construct are not random but rather are stratified by level; this is an example of an \emph{ascending product} as formalized in Definition \ref{def:ascending-prod}.
    \begin{proof}
        By assumption, the simplices $\{A_{X_i}\}$ and $\{A_{Y_i}\}$ are conjugate. Since subgroup inclusion and commuting are both preserved under conjugation, the two simplices have the same number and length of nesting chains. Consider the maximal elements of the two nesting chains. In both simplices, the maximal elements form conjugate, standard reducible parabolic subgroups $A_Z$ and $A_{Z'}$ which are almost maximal in the sense that there are Artin generators $t$ and $t'$ such that $Z \cup \{t\} = Z' \cup \{t'\} = V(\Gamma)$. It is possible that $A_{Z} = A_{Z'}$ and the two standardizations differ at subsequent levels. As we saw in the proof of Proposition \ref{product of deltas}, either they are equal or $\Delta_{\Gamma} A_{Z} \Delta_{\Gamma}^{-1} = A_{Z'}$ and $\Delta_{\Gamma} A_{Z'} \Delta^{-1}_{\Gamma} = A_{Z}$. 

        If there is an ordering such that $\Delta^{e}_{\Gamma} A_{X_i} \Delta_{\Gamma}^{-e} = A_{Y_j}$ for every $i$ and the appropriate $e \in \{0,1\}$, then $r = \Delta^e_{\Gamma}$. Now suppose that there is some level 2 element $A_{X_j} \leq A_{X_1}$ such that $\Delta_{\Gamma}^e A_{X_j} \Delta_{\Gamma}^{-e} \not\in \{A_{Y_i}\}$. 
        
        Without loss of generality, suppose $\Delta_{\Gamma}^e A_{X_1} \Delta_{\Gamma}^{-e} = A_{Y_1}$. Since conjugation preserves subgroup inclusion, $\Delta_{\Gamma}^e A_{X_j} \Delta_{\Gamma}^{-e} \leq A_{Y_1}$. We can apply the same argument inside of $A_{Y_i}$ to see that we must have $\Delta_{Y_1} \Delta_{\Gamma}^{e} A_{X_j} \Delta_{\Gamma}^{-e} \Delta_{Y_1}^{-1} \in \{A_{Y_i}\}$.

        Applying a similar argument inside the irreducible components at each level of the nesting order shows that there is some element $r = \Delta_{Y_k}^{\square}\cdots \Delta_{Y_1}^{\square}\Delta_{\Gamma}^{\square}$ such that $\{r A_{X_i}r^{-1} \}= \{A_{Y_i}\}$, where each $\square$ may denote 0 or 1.

        Lemma \ref{Delta conjugation} implies 
        \begin{align*}
            \Delta_{Y_k}^{\square}\cdots\Delta^{\square}_{Y_1}\Delta^{\square}_{\Gamma} = \Delta^{\square}_{\Gamma} \Delta^{\square}_{X_1} \cdots \Delta^{\square}_{X_k}\text{.}
        \end{align*}
    \end{proof}

    At a few points throughout the paper, we will also require the following lemma.
    \begin{lemma}\label{conjugation implies containment}
        Let $A_X$, $A_Y$, and $A_Z$ be irreducible standard parabolic subgroups such that $z_X$ does not commute with $z_Z$. If $\Delta_{Z}^i A_X \Delta_Z^{-i} = A_Y$ for some integer $i$, then $i = 0$ and $X = Y$.
    \end{lemma}
    \begin{proof}
        Recall that if $\Delta_Z^i$ conjugates $A_X$ to $A_Y$, then it also conjugates $z_X$ to $z_Y$.
        
        First, notice that if $i \neq 0$, then $z_{Z}$ also does not commute with $z_Y$. If it did, then $z_X = \Delta_{Z}^{-i} z_Y \Delta_{Z}^i = z_Y$ would yield a contradiction because the right hand side would commute with $z_Z$ while $z_X$ was assumed not to. Thus the roles of $X$ and $Y$ are symmetric, so we can assume without loss of generality that either $i = 0$ and $X = Y$ or the exponent $i$ is positive.

         In the proof of Lemma 21 in \cite{parabs-acting-on-CAL}, it is shown that any positive element which conjugates the positive element $z_X$ to the positive element $z_Y$, such as $\Delta_Z^i$, can be decomposed as a product of positive elements $c_1 \cdots c_r$ of the following form for some sequence of subsets $X \rightarrow X_1 \rightarrow \cdots X_n \rightarrow Y$ such that:
            \begin{enumerate}
                \item $c_1 \in A_Y$, $c_3 \in X_{n}$, $c_5 \in X_{n-1}$, etc. Some may be empty.
                \item $c_r \in A_X$. This term may be empty.
                \item $c_2$ is a positive elementary ribbon conjugating $X_n$ to $Y$, $c_4$ is a positive elementary ribbon conjugating $X_{n-1}$ to $X_{n}$, etc. Each of these terms is nonempty.
            \end{enumerate}
         

        Suppose $i \neq 0$ and $X \neq Y$. The sequence of subsets $X \rightarrow \cdots X_i \rightarrow \cdots Y$ must contain some $X_i \neq X$, so there must be some $i$ such that $c_{2i}$ is of the form $d_{X,t}$ where $t$ is adjacent to $X$ but $t \not\in X$. The subset $X$ is contained in the support of such an element by Lemma \ref{lem:ribbon-support}. Since the product is positive, this implies that $X \subseteq \text{supp}(\Delta_Z^i) = Z$. This is a contradiction because it implies $A_X \leq A_Z$, and we assumed $z_X$ did not commute with $z_Z$.

        Suppose instead that $i \neq 0$ but $X = Y$. If there is any $c_{2j}$ of the form $d_{X,t}$ where $t$ is adjacent to $X$ but $t \not\in X$, we conclude by the same reasoning as above. Similarly, if there are no $c_{2j}$ elements, or equivalently if $\Delta_Z^{i} = c_1$, then $\Delta_Z^i \in A_X$. This implies that $\Delta_Z^{2i}$ commutes with $z_X$, but $\Delta_Z^{2i}$ is a power of $z_Z$. This implies that $z_Z$ commutes with $z_X$ by Lemma 4.6 of \cite{C-parab-definition}, contradicting our assumption. 
        
        Otherwise, our $c_{2j}$ terms may be of the form $d_{X,t}$ where $t \not\in X$ is not adjacent to $X$ in $\Gamma$. In this case, the element $d_{X,t}$ is precisely the generator $t$. Each term in the product $\Delta_Z^i$ is then either $t$ such that $t$ commutes with $A_{X}$ or an element contained in $A_X$. A product of such elements commutes with $z_X$, so we obtain a contradiction as in the previous paragraph. We conclude that $i = 0$.

    \end{proof}

\subsection{Simplex Stabilizers}
    The natural action of $A_{\Gamma}$ by conjugation on $C_{parab}$ extends to an action on the set of maximal simplices of $C_{parab}$: an element $g \in A_{\Gamma}$ acts on $\{P_i\}$ by replacing each $P_i$ with its conjugate by $g$. In the mapping class group, stabilizers of pants decompositions were well-understood before the introduction of markings. If maximal $C_{parab}$ simplices are to play the role of pants decompositions in our marking graph, we will first need to understand their stabilizers.

    Recall that any element of the normalizer of a standard parabolic subgroup $X$ is an element of the form $g_x r$ where $g_x \in A_X$ and $r$ is an $X$-ribbon-$X$ \cite{Godelle-spherical-normalizers}.

    \begin{definition}\label{def:ascending-prod}
        Let $\Sigma$ be a $C_{parab}$ simplex with vertex set $\{A_{X_i}\}$. A word $p$ representing a product of powers of $\Delta_{X_i}$ and $\Delta_{\Gamma}$ is called an \emph{ascending product} if it is of the form 
        \begin{align*}
            p = \Delta_{X_{i_1}}^{\square} \cdots \Delta_{X_{i_m}}^{\square} \Delta_{\Gamma}^{\square}
        \end{align*}
        where all factors of $\Delta_{\Gamma}$ appear at the right end of $p$, $X_i \subseteq X_q$ implies the $\Delta_{X_i}$ term is to the left of the $\Delta_{X_q}$ term in the product, and $\square$ may denote any integer. For technical reasons, we assume that there is a term for each element $A_{X_i}$ of $\Sigma$ and allow $\square$ to be 0.
    \end{definition}

    The element $r$ produced in Lemma \ref{changing standardization}, for example, was an ascending product. We will need the following lemma.
    
    \begin{lemma}\label{conj preserving gen sets}
        Let $\{A_{X_i}\}$ be a maximal standard simplex in $C_{parab}$, and let an ascending product $g$ of the $\{\Delta_{X_i}\}$ stabilize $\{A_{X_i}\}$. Let $g_k$ denote the subword of $g$ of the form $g_k = \Delta_{X_{(1,k-1)}}^i \Delta^j_{X_{(2,k-1)}}\cdots \Delta_{\Gamma}^l$, containing precisely the $\Delta_{X_{(i,L)}}$ factors such that $X_{(i,L)}$ is in level $L \leq k-1$ of the simplex and written in ascending order so that $A_{X_{(\square, k-1)}}$ is in the $(k-1)$th level and the final factor before $\Delta_{\Gamma}$ is maximal. Conjugation by $g_k$ stabilizes the $k$th level of the simplex.
    \end{lemma}

    In the case which is simplest to state, this lemma says that if $g$ is an ascending product which stabilizes a simplex, then the $\Delta_{\Gamma}$-factor of $g$ must stabilize the first level of the simplex. Intuitively, this is because all of the lower level terms are contained in the first level, so if the $\Delta_{\Gamma}$-factor of $g$ did not stabilize the first level, no product of lower level terms would be able to ``undo'' the change.

    \begin{proof}
        Let $p$ be the initial subword of $g$ such that $g = p g_k$. By construction, for each term $\Delta_{X_i}$ in $p$ and each $\Delta_{(j,k)}$ in the $k$th level, either $A_{X_i}$ is contained in $A_{X_{(j,k))}}$ or the two commute. In either case, $\Delta_{X_i}^{\square} A_{X_{(j,k))}} \Delta_{X_i}^{-\square} = A_{X_{(j,k))}}$ where $\square$ may be any integer.

        Observe that if $g$ stabilizes the simplex, so does $g^{-1}$. Since conjugation preserves subgroup inclusion, both $g$ and $g^{-1}$ in fact stabilize each level of the simplex, i.e., $g^{-1} A_{X_{(j,k)}} g = A_{X_{(j',k)}}$ for some $j'$. For each $j$, we then have the following.
        \begin{align*}
            A_{X_{(j',k)}} = g^{-1} A_{X_{(j,k)}} g = g_k^{-1} p^{-1} A_{X_{(j,k)}} p g_k = g_k^{-1} A_{X_{(j,k)}} g_k\text{.}
        \end{align*}
        We conclude that $g_k^{-1}$, and thus also $g_k$, stabilizes the $k$th level.
    \end{proof}

    We can now prove Theorem \ref{thmx:simplex-stab} of the introduction.

    \begin{theorem}\label{almost ascending product}
        If $g$ stabilizes a maximal $C_{parab}$ simplex $\{A_{X_i}\}$, then $g$ can be written as an ascending product of powers of $\Delta_{X_i}$ and $\Delta_{\Gamma}$.
    \end{theorem}
    \begin{proof}
        We prove this by induction on the length of nesting chains in the marking. 
        
        \textbf{Base case:} Suppose that every nesting chain in the simplex has length 1. No simplex has more than 3 maximal elements, so let $A_X$, $A_Y$ and $A_Z$ be the maximal parabolic subgroups which appear (up to two of $X$, $Y$, and $Z$ may be empty). Note that they must commute with one another, and that by Proposition \ref{maximal simplex characterization}, there is one $t$ such that $X \cup Y \cup Z = V(\Gamma) - \{t\}$.

        If $g$ stabilizes the simplex, then it stabilizes the collection $\{A_X , A_Y, A_Z\}$. In particular, this implies that $g$ normalizes the reducible standard parabolic subgroup $A_{X \cup Y \cup Z}$. Thus $g$ is of the form $g'r$ where $g' \in A_{X \cup Y \cup Z}$ and $r$ is a $(X \cup Y \cup Z)$-ribbon-$(X\cup Y \cup Z)$. 
        
        By Proposition \ref{product of deltas}, the element $r$ is a product of the form $\Delta^{\square}_X \Delta^{\square}_Y \Delta^{\square}_Z \Delta^{\square}_{\Gamma}$. Since $A_X$, $A_Y$, and $A_Z$ are also minimal elements of a maximal simplex, they must consist of a single generator, which is also the Garside element. Thus any element $g' \in A_{X \cup Y \cup Z}$ is a product of the form $\Delta^{\square}_X \Delta^{\square}_Y \Delta^{\square}_Z$. The elements $\Delta_{X}$, $\Delta_Y$, and $\Delta_Z$ commute by construction, so $g'r$ is also of the desired form.

        \textbf{Induction step:} Suppose the statement is true for all simplices where nesting chains have length at most $n-1$, and consider a simplex with nesting chains of length $n$. Let $A_X$, $A_Y$ and $A_Z$ be the maximal components, up to two of which may be empty. As before, $A_X$, $A_Y$ and $A_Z$ all commute with each other, and there is a single $t \in S$ such that $t \not\in X \cup Y \cup Z$.

        As in the base case, since $g$ stabilizes the simplex, the element $g$ is of the form $g'r$ where $g' \in A_{X \cup Y \cup Z}$ and $r$ is a ribbon of the form $\Delta^{\square}_X \Delta^{\square}_Y \Delta^{\square}_Z \Delta^{\square}_{\Gamma}$.

        By construction, $A_{X \cup Y \cup Z} = A_X \times A_Y \times A_Z$, so there are unique $g_x \in A_X$, $g_y \in A_Y$, and $g_z \in A_Z$ such that $g' = g_x g_y g_z$.
        
        The non-maximal elements $\{A_{X_i}\} - \{A_{X}, A_Y , A_Z\}$ can be subdivided into maximal standard simplices in the parabolic subgroup complexes associated to $A_X$, $A_Y$, and $A_Z$. By Lemma \ref{conj preserving gen sets}, conjugation by $r$ permutes the maximal elements of each of these subsimplices such that $\{r A_{X_i} r^{-1}\} - \{A_{X}, A_Y , A_Z\}$ can also be subdivided into maximal simplices in the parabolic subgroup complexes associated to each of $A_{X}$, $A_Y$, and $A_Z$. 

        The elements $g_y$ and $g_z$ commute with every element of $A_{X}$. Thus since $g = g_x g_y g_z r$ stabilizes the simplex $\{A_{X_i}\}$, $g_x$ must stabilize the subset of  $\{r A_{X_i} r^{-1}\} - \{A_{X}, A_Y , A_Z\}$ which forms a maximal simplex in $C_{parab}(A_X)$. This simplex has strictly shorter nesting chains than the starting simplex, so by the induction hypothesis, $g_x$ can be decomposed as an ascending product of the Garside elements of parabolic subgroups contained in $A_X$. 

        Analogous arguments for $g_y$ and $g_z$ reveal that they can also be written as appropriate ascending products. No parabolic subgroup contained in $A_{X}$ can be contained in any parabolic subgroup of $A_Y$ since $A_X$ and $A_Y$ are disjoint and commute, nor can the analogous situation occur in $A_X$ and $A_Z$ or $A_Y$ and $A_Z$, so $g_x g_y g_z$ is an ascending product. 

        Since every element of $A_{X}$ commutes with every element of $A_Y$ and $A_Z$, the complete expression $g = g' r = g_x g_y g_z \Delta_{X}^i \Delta_{Y}^j \Delta_Z^k \Delta_{\Gamma}^l$ can be rearranged to an ascending product expression by moving all $\Delta_{Y}$ factors of $g_{y}$ to the right of the non-$\Delta_{Z}$ factors of $g_z$ and moving all $\Delta_{X}$ factors of $g_x$ to the right of the non-$\Delta_{Y}$ and non-$\Delta_{Z}$ factors of $g_yg_z$.
        \end{proof}

In Section \ref{sec:transversals}, we will make use of the following observation to compare different standardizing elements of a maximal $C_{parab}$-simplex. It is a straightforward consequence of several earlier results, but we state it explicitly for clarity.

\begin{lemma}\label{p comparison}
    Let $\{P_i\}$ be a maximal $C_{parab}$-simplex, and let $g$ and $h$ be two standardizing elements so that $P_i = g A_{X_i} g^{-1} = h A_{Y_i} h^{-1}$ for each $i$ and some $\{A_{X_i}\}$ and $\{A_{Y_i}\}$. Let $r'$ be the element of Lemma \ref{changing standardization} which conjugates $\{A_{X_i}\}$ to $\{A_{Y_i}\}$, and let $r$ be the element of Lemma \ref{changing standardization} which conjugates $\{A_{Y_i}\}$ to $\{A_{X_i}\}$. 
    
    There is an ascending product $p_1$ of $\Delta_{Y_i}$ and $\Delta_{\Gamma}$ and an ascending product $p'_1$ of $\Delta_{X_i}$ and $\Delta_{\Gamma}$ such that $h = grp_1 = g p'_1 r$. There is also an ascending product $p_2$ of $\Delta_{X_i}$ and an ascending product $p'_2$ of $\Delta_{Y_i}$ such that $gr = h p'_2$ and $gp_2 = h r'$.
\end{lemma}
\begin{proof}
    Notice that the elements $g^{-1}hr'$ and $h^{-1}gr$ stabilize the maximal simplices $\{A_{X_i}\}$ and $\{A_{Y_i}\}$ respectively. Specifically, if $\sigma$ is the permutation such that $r' A_{X_i} r'^{-1} = A_{Y_{\sigma(i)}}$, then
    \begin{align*}
        g^{-1} h r' A_{X_i} r'^{-1} h^{-1} g &= g^{-1} (h A_{Y_{\sigma(i)}} h^{-1}) g\\
        &= g^{-1} (g A_{X_{\sigma(i)}} g^{-1}) g\\
        &= A_{X_{\sigma(i)}}
    \end{align*}
    
    Theorem \ref{almost ascending product} implies that there is an ascending product $p_2$ of the $\Delta_{X_i}$ and $\Delta_{\Gamma}$ with $g^{-1} h r' = p_2$ and an ascending product $p'_2$ of the $\Delta_{Y_i}$ and $\Delta_{\Gamma}$ with $h^{-1} g r = p'_2$. 
    
    Similarly, the element $r^{-1}g^{-1}h$ stabilizes the simplex $\{A_{Y_i}\}$. If $\tau$ is the permutation with $r A_{Y_i} r^{-1} = A_{X_{\tau(i)}}$, then
    \begin{align*}
        r^{-1}g^{-1}(h A_{Y_i} h^{-1}) g r &= r^{-1}g^{-1}(g A_{X_i} g^{-1}) g r\\
        &= r^{-1} A_{X_i} r\\
        &= r^{-1} (r A_{Y_{\tau^{-1}(i)}} r^{-1}) r\\
        &= A_{Y_{\tau^{-1}(i)}}\text{.}
    \end{align*}
    Theorem \ref{almost ascending product} then implies that there is an ascending product $p_1$ of the $\Delta_{Y_i}$ and $\Delta_{\Gamma}$ with the following property.
    \begin{align*}
        r^{-1} g^{-1} h &= p_1\\
        h &= grp_1
    \end{align*}
    Since $r$ is a product of Garside elements of standard parabolic subgroups containing each $A_{Y_i}$ and which conjugate each $A_{Y_i}$ to some $A_{X_j}$, repeated applications of Lemma \ref{Delta conjugation} imply that there is a suitable $p'_1$ with $rp_1 = p'_1 r$.
\end{proof}

Notice that in each of these cases, it is clear from the position of $p$ or $p'$ whether the factors in the ascending product are $\Delta_{X_i}$ and $\Delta_{\Gamma}$ or $\Delta_{Y_i}$ and $\Delta_{\Gamma}$.

We conclude this section with the following fact, which will play a key technical role in Section \ref{sec:marking stabilizers}.
\begin{proposition}\label{differ by p}
    Let $\{P_i\}$ be a maximal $C_{parab}$ simplex, and let $g$, $h$, $\{A_{X_i}\}$, and $\{A_{Y_i}\}$ be such that $g A_{X_i} g^{-1} = P_i = h A_{Y_i} h^{-1}$ for each $i$. There is an ascending product $x$ of $\Delta_{X_i}$ and $\Delta_{\Gamma}$ such that $gx = h$.
\end{proposition}
\begin{proof}
    Let $r$ be the ascending product of $\Delta_{X_i}$ and $\Delta_{\Gamma}$ provided by Lemma \ref{changing standardization} so that $\{r A_{Y_i} r^{-1}\} = \{A_{X_i}\}$. Let $\sigma$ be the permutation so that $r A_{Y_i} r^{-1} = A_{X_{\sigma(i)}}$.
    
    Applying Lemma \ref{p comparison} shows that there is some ascending product $p'_1$ of the $\Delta_{X_i}$ and $\Delta_{\Gamma}$ such that $gp'_1r = h$ and $p'_1$ stabilizes the collection $\{A_{X_i}\}$. For simplicity of notation, we will rename $p'_1 = p$. It suffices to check that any product $pr$ where $p$ stabilizes $\{A_{X_i}\}$ can be written as an ascending product $p'$ of the desired form.

    Consider the leftmost factor of $r$, say $\Delta_{X_j}$, and suppose $A_{X_j}$ is at level $L_j$ in the marking. In Lemma \ref{conj preserving gen sets}, we saw that there is a permutation $\tau$ of the $L_j$th level of the simplex such that, if $p_{L_j}$ is the terminal subword of $p$ containing all terms at higher levels of the marking than $A_{X_j}$, then $p_{L_j} \Delta_{X_j} = \Delta_{X_{\tau(j)}} p_{L_j}$. All terms of $p$ within a level commute, so we can remove the $\Delta_{X_j}$ term from $r$ and instead include it in the $\Delta_{X_{\tau(j)}}$ term of $p$.

    The element $r$ has finitely many terms, so repeated applications of this process result in an $r$ which consists only of a power of $\Delta_{\Gamma}$. The right-multiple of an ascending product $p$ by any power of $\Delta_{\Gamma}$ is still an ascending product, which concludes the proof.
\end{proof}

\section{Markings}\label{sec:markings}
\subsection{Definition}
Recall that a clean marking on a genus 0 surface is a collection of ordered pairs of essential simple closed curves $\{(\alpha_i, t_i)\}$ such that the collection $\{\alpha_i\}$ forms a maximal simplex in the curve complex, $t_i$ intersects $\alpha_j$ if and only if $i = j$, and for each $i$, a neighborhood of $\alpha_i \cup \beta_i$ is homeomorphic to a 4-punctured sphere. In this section, we describe a natural extension of this idea to the parabolic subgroup complex for a finite-type Artin group. Since we will only work with analogues of complete clean markings, we drop the terms ``complete'' and ``clean''.

    \begin{definition}\label{def:marking}
        A \emph{marking} $M$ on $A_{\Gamma}$ is a collection $\{(P_i , Q_i)\}$ of ordered pairs of irreducible parabolic subgroups such that \begin{enumerate}
            \item $\{P_i\}$ spans a maximal simplex in $C_{parab}$,
            \item $z_{Q_i}$ commutes with $z_{P_j}$ if and only if $i \neq j$, and
            \item For each $j$, the collection $\{P_i\} \cup Q_j$ is simultaneously standardizable.
        \end{enumerate}
    \end{definition}

Following the naming convention for mapping class groups, we refer to $\{P_i\}$ as the \emph{base elements} of the marking and to $\{Q_i\}$ as the \emph{transverse elements}. If all of the base elements are standard parabolic subgroups, we say that the marking is \emph{standard}. We say that a collection $\{P_i\}$ is a \emph{nesting chain} in the marking if it is a nesting chain in the $C_{parab}$-simplex spanned by the base elements. Similarly, $P_i$ is \emph{minimal} (resp. \emph{maximal}) if it is a minimal (resp. maximal) element of the simplex spanned by the base elements. The transverse elements generally do not span a simplex in $C_{parab}$.

    The first two conditions in our definition of a marking are obtained by applying the idea that two parabolic subgroups correspond to disjoint curves if their centers commute, as in \cite{C-parab-definition}. The third condition is a natural analogue of the requirement that a base-transverse pair must fill a 4-punctured sphere.

    To see this, consider a pants decomposition $\{\alpha_i\}$ on the punctured disk, and let $t_j$ be a choice of transverse element for $\alpha_j$. Curves on the disk correspond to simultaneously standardizable parabolic subgroups exactly when there is an element $\varphi$ of the mapping class group which sends $\{\{\alpha_i\}, t_j\}$ to round curves. The fact that such a $\varphi$ exists for the pair $(\alpha_j, t_j)$ is precisely the condition that they fill a 4-punctured sphere. All other $\alpha_i$ are contained entirely in $D_n - S_{0,4}$, and they can therefore be sent to round curves by appropriate elements of the mapping class groups of the connected components of $D_n - S_{0,4}$. These elements fix $\alpha_j$ and $t_j$ by construction, so the desired element is simply the composition of $\varphi$ with these elements.

      \begin{figure}[h]
        \centering
        \includegraphics[width=12cm]{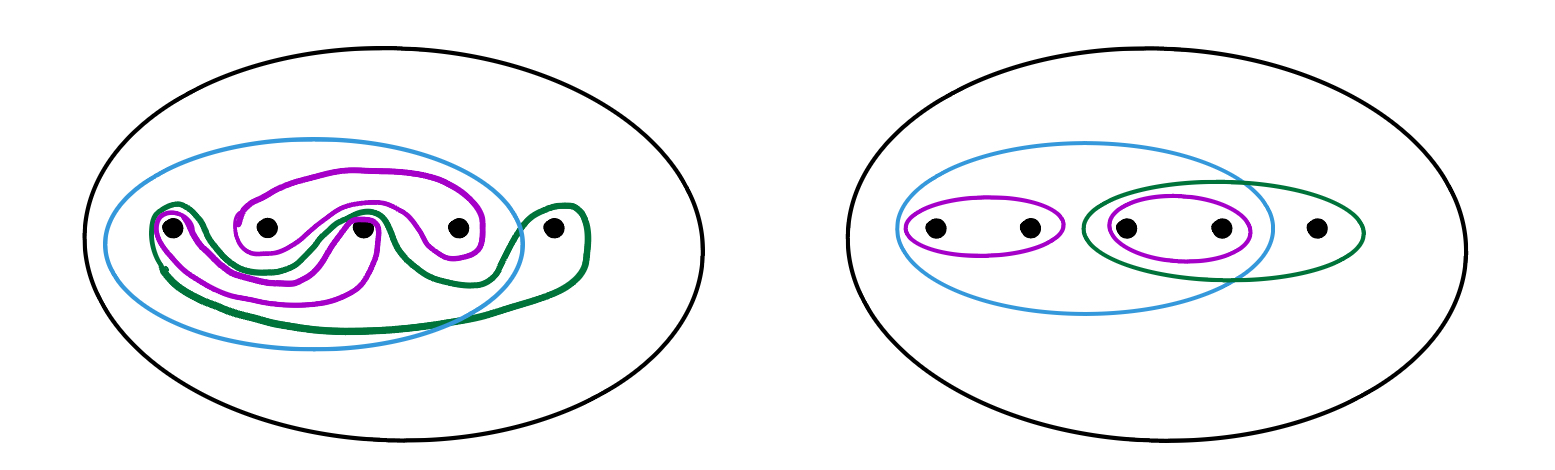}
        \captionsetup{margin=.5cm,justification=centering}
        \caption{A maximal simplex in the curve complex of the punctured disk (purple and green) and a transverse curve (blue) for the green base curve, pre and post-standardization}
        \label{fig:standardizingsimplices}
    \end{figure}

    Since maximality in $C_{parab}$, commutativity of subgroups, subgroup inclusion, and simultaneous standardizability are all preserved under conjugation, we obtain the following.
    
    \noindent\textbf{Key observation:} The property of being a marking is preserved under conjugation.

    \subsection{Existence}
    While the existence of suitable base collections is clear from the fact that $C_{parab}$ is finite-dimensional and non-empty, it is not obvious from our definition that a suitable choice of transverse elements exists for any choice of base elements. In this subsection, we show that this is in fact the case. Moreover, we show that there is always a choice of transverse elements that are simultaneously standardizable with both the base \emph{and each other}. These ``uniformly standardizable'' collections will be useful in later sections.

        \begin{proposition}\label{marking existence}
            Any maximal simplex $\{P_i\}$ in $C_{parab}$ is the collection of base elements for some marking. Moreover, the transverse elements can be chosen such that the entire collection $\{P_i\} \cup \{Q_i\}$ is simultaneously standardizable.
        \end{proposition}
        \begin{proof}
            It is no loss of generality to assume that $\{P_i\}$ is standard. To reflect this, we rename the collection $\{A_{X_i}\}$.

            Let $A_X$, $A_Y$, and $A_Z$ denote the maximal elements of $\{A_{X_i}\}$. First, suppose at least two are non-empty. By Proposition \ref{maximal simplex characterization}, each is contained in a different connected component of $\Gamma$ minus a single vertex $v$. In particular, $\Gamma - X$ is a connected subgraph which contains any other maximal components. The corresponding subgroup $A_{\Gamma - X}$ is disjoint from $X$ but does not commute with it because $v \in \Gamma - X$ is adjacent to $X$.
                
            If $v$ is not connected by an edge to any vertex of $X_i$ for all $X_i \leq X$, then we can choose $\Gamma - X$ to be the transverse element for $X$, since the vertex set of $\Gamma - v$ is precisely $v \cup Y \cup Z$ and $X_i$ also is not connected by an edge to any vertex of $Y$ or $Z$.
            
            Consider the maximal components $X'_1$, $X'_2$, and $X'_3$ \emph{inside} of $X$. There is a single vertex $w \in X$ such that $X_1'$, $X_2'$, and $X_3'$ are the connected components of $X - \{w\}$. If $v$ is connected by an edge to any $X_i \in X$, then it will necessarily be connected by an edge to one of these. Furthermore, the components $X'_1$, $X'_2$, and $X'_3$ are connected to each other via edges inside of $X$. Since $\Gamma$ has no cycles, this implies $v$ is connected by an edge to at most one of the $X_i$. If $v$ is connected to $X'_1$, then we instead choose the transverse element to be $(\Gamma - X) \cup X'_1$. By the same reasoning as before, this element has the desired commuting properties, and its intersection with $X$ is precisely $X'_1$.

            \begin{figure}[h]
                \centering
                \includegraphics[width=8cm]{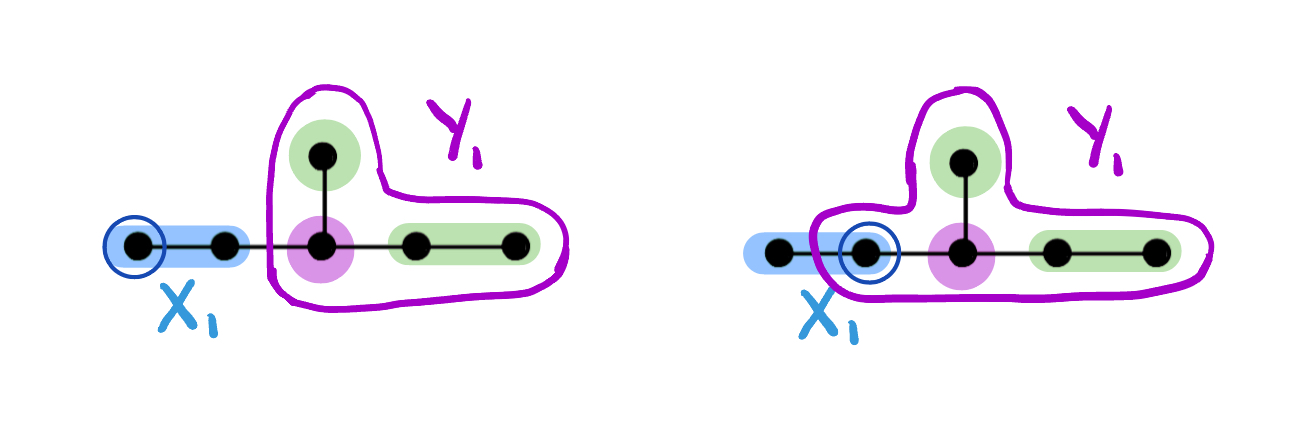}
                \captionsetup{margin=.5cm,justification=centering}
                \caption{Some maximal base elements (blue and green highlights), the missing vertex $v$ (purple highlights) and possible transverse elements (purple lines) for $A_{X_1}$ depending on the location of the next level inside $X_1$ (blue circles)}
            \end{figure}
            Repeating this process for $Y$ and $Z$, if they are nonempty, yields transverse elements for each maximal element.
                
            Now proceed to the elements which are at level 2 in the marking and contained in $X$: $X'_1$, $X'_2$, and $X'_3$. The set of $\{X_i\}$ which are contained in $X$ clearly spans a maximal simplex in $C_{parab}(X)$. Thus we can build transverse elements for $X'_1$, $X'_2$, and $X'_3$ \emph{inside} of $X$ in precisely the same way as we did for the maximal components. Note that the result will be contained entirely in $A_X$ and will therefore commute with both $A_Y$ and $A_Z$ (and consequently any subgroups they contain). We then repeat this process inside of $Y$ and $Z$, and then iterate for subsequent steps down all nesting chains in the collection. This process terminates because the size of the collection is finite.       
        \end{proof}  
  
    \subsection{Structure}
    The following facts about markings will frequently be useful.
    \begin{lemma}\label{maximal inclusion}
        Let $M = \{(A_{X_i}, Q_i)\}$ be a standard marking. Let the maximal base elements be $A_{X_1}$, $A_{X_2}$, and $A_{X_3}$. If $Q_1 = A_{Y_1}$, then $A_{X_2 \cup X_3} \leq A_{Y_1}$.
    \end{lemma}
    \begin{proof}
        If $A_{X_2}$ and $A_{X_3}$ are both trivial, then the statement is trivial, so we assume at least $A_{X_2}$ is nontrivial. The parabolic subgroups $A_{X_2}$, $A_{X_3}$ if it is nontrivial, and $A_{Y_1}$ are all adjacent in $C_{parab}$, so either they are disjoint and commute or there is inclusion in one direction. The subgroup $A_{Y_1}$ cannot be included in either of $A_{X_2}$ or $A_{X_3}$ because these subgroups commute with $A_{X_1}$, and $A_{Y_1}$ does not.
        
        Suppose that $A_{Y_1}$ is disjoint from $A_{X_2}$ and the two subgroups commute. In particular, this implies that $Y_1 \subseteq \Gamma - (X_2 \cup \partial(X_2))$. In Lemma \ref{3 maxl base components}, we saw that when $X_2$ is maximal, $\partial(X_2) = \{v\}$ for a single vertex $v$ and $\Gamma - (X_2 \cup \{v\}) = X_1 \cup X_3$. Any irreducible parabolic subgroup which is contained in $A_{X_1 \cup X_3}$ is connected to $A_{X_1}$ via an edge in $C_{parab}$. This contradicts the requirement that $z_{Y_1}$ does not commute with $z_{X_1}$. Thus $A_{X_2} \leq A_{Y_1}$. The same argument can be applied to $A_{X_3}$, if it is nontrivial, to obtain that $A_{Y_1}$ must contain both of $A_{X_2}$ and $A_{X_3}$. 
    \end{proof}
    \begin{corollary}\label{maximal transversal inclusion}
        Let $M = \{(P_i,Q_i)\}$ be a marking. If the base elements $P_1$, $P_2$, and $P_3$ are maximal and non-empty, then $P_2, P_3 \leq Q_1$.
    \end{corollary}
    \begin{proof}
        By the third condition of a marking, $M$ is conjugate to a marking $M'$ which satisfies the conditions of Lemma \ref{maximal inclusion}, and subgroup inclusion is preserved under conjugation.
    \end{proof}

    \begin{lemma}\label{lower level transversal inclusion}
        Let $M = \{(P_i , Q_i)\}$ be a marking. If $P_j \leq P_k$ and $j \neq k$, then $Q_j \leq P_k$ and either $P_j \leq Q_k$ or $P_j$ commutes with $Q_k$.
    \end{lemma}
    \begin{proof}
        Let $P_j \leq P_k$. The subgroup $Q_j$ must be adjacent to $P_k$ in $C_{parab}$. This implies that either $Q_j \leq P_k$, or $Q_j \cap P_k = \{1\}$ and $Q_j$ and $P_k$ commute, or $P_k \leq Q_j$. Either of the latter two cases would imply that $Q_j$ is also adjacent to every subgroup of $P_k$ in $C_{parab}$. This is a contradiction because $Q_j$ is a transverse element for $P_j$. Thus $Q_j \leq P_k$.

        Similarly, $P_j$ must be adjacent to $Q_k$ in $C_{parab}$, so either $P_j \leq Q_k$, $P_j \cap Q_k = \{1\}$ and $P_j$ commutes with $Q_k$, or $Q_k \leq P_j$. If $Q_k \leq P_j$, then $Q_k \leq P_k$. This is a contradiction because $Q_k$ is a transverse element for $P_k$.
    \end{proof}

    Combining the results above with the proof of Proposition \ref{marking existence} yields the following.
    \begin{lemma}\label{lem:std-transv-uniqueness}
        Let $\{A_{X_i}\}$ be a maximal $C_{parab}$-simplex. For each $A_{X_j} \in \{A_{X_i}\}$, there is a unique choice of standard transverse element $A_{Y_j}$.
    \end{lemma}
    \begin{proof}
        The elements constructed in \cref{marking existence} are standard, so it suffices to show uniqueness. Furthermore, it suffices to demonstrate that this holds for the top level; Lemma \ref{lower level transversal inclusion} demonstrates that each element in a lower level is maximal in a submarking on some $A_{X_i}$ in the collection which also contains its transverse element, so the same argument can be applied with $X_i$ playing the role of $\Gamma$.

        Let $A_{X_1}$, $A_{X_2}$ and $A_{X_3}$ be the maximal elements of $\{A_{X_i}\}$, up to two of which may be empty. Let $A_{X_1'}$, $A_{X_2'}$, and $A_{X_3'}$ denote the maximal elements inside of $A_{X_1}$, any of which may be empty. Let $v$ and $w$ be the vertices of \cref{maximal simplex characterization} such that $\Gamma = X \cup Y \cup Z \cup \{v\}$ and $X = X_1' \cup X_2' \cup X_3' \cup \{w\}$. 
        
        Corollary \ref{maximal transversal inclusion} implies that $X_2$ and $X_3$, if they are non-empty, must be contained in $Y_1$. Since $A_{Y_1}$ is an irreducible standard parabolic subgroup which is neither contained in nor commutes with $A_{X_1}$, this implies $v \in Y_1$. 
        
        Notice that either $w$ and $v$ are adjacent and no $X_{i}'$ is adjacent to $v$, or there is some $X_i'$ which is adjacent to $v$. As in \cref{marking existence}, there is a unique such $X_i'$ if one exists. The set $Y_1$ cannot contain $w$. If all three of $X'_1$, $X'_2$, and $X'_3$ are empty, then $X_1 = \{w\}$, and $Y_1$ cannot contain $X_1$. Recall that by the definition of a transverse element, $A_{Y_1}$ must contain or be disjoint from and commute with each $A_{X'_i}$. If at least one of the $X'_i$ were nonempty and $A_{Y_1}$ contained $w$, then $A_{Y_1}$ could not commute with any of $A_{X_1'}$, $A_{X_2'}$, or $X_3'$, so it would have to contain all of them. Consequently, $A_{Y_1}$ would contain all of $A_{X_1}$, contradicting transversality. Since $w$ is necessarily a cut point in $\Gamma$, this implies that $Y_1$ is precisely $X_2 \cup X_3 \cup \{v\} \cup X_i'$, where $X_i'$ is the unique second level element which is adjacent to $v$ and where any of $X_2$, $X_3$, or $X_i'$ may be empty.
    \end{proof}

    \begin{lemma}\label{distinct twists}
        Let $\{A_{X_i}\}$ be the base of a marking, and let $Q$ be the transverse element for a particular $A_{X_j}$. The subgroup $\Delta_{X_j}^k Q \Delta_{X_j}^{-k}$ is equal to $Q$ only when $k = 0$.
    \end{lemma}
    \begin{proof}
        Suppose there is some $k$ with $\Delta_{X_j}^k Q \Delta_{X_j}^{-k} = Q$. By Proposition 35 of \cite{Cumplido-minimal-standardizers}, this implies the following equalities. 
    \begin{align*}
        \Delta_{X_j}^k z_Q \Delta_{X_j}^{-k} &= z_Q\\
        \Delta_{X_j}^k z_Q &= z_Q \Delta_{X_j}^k \\
        \Delta_{X_j}^{2k} z^2_Q &= z^2_Q \Delta_{X_j}^{2k}
    \end{align*}
    By Lemma 4.6 in \cite{C-parab-definition}, this implies that either $k = 0$ or $z_{X_j}$ and $z_Q$ commute. The latter contradicts transversality of $Q$, so $k=0$.
    \end{proof}

\section{Transverse projections}\label{sec:transversals}

Our ultimate goal in this section is to define a projection of a transverse element $Q$ onto its corresponding base element $P$. This will function in some sense like a ``winding number'' which detects how many times one curve on a surface is twisted around another, though the definition of our projections will be very different.

\subsection{Simplex projections of standardizers}

To begin, we define the projection to each vertex of a standardizer $g$ for some $C_{parab}$-simplex. We will later use the simultaneous standardizability of transverse elements and base simplices to extend this to projections of transverse elements. One might hope to simply examine the $\Delta_{X_i}$ suffixes of $g$ for each $i$, but in fact, this is not sufficiently well-behaved with respect to negative powers. Instead, we define our projections as follows.

    \begin{lemma}\label{unique p}
        Let $\{A_{X_i}\}$ be a maximal $C_{parab}$ simplex, and let $p$ be an ascending product of $\Delta_{X_i}$ and $\Delta_{\Gamma}$. The tuple of exponents of $\Delta_{X_i}$-factors in $p$ is unique up to reordering entries corresponding to commuting Garside elements that are in the same level of the marking.
    \end{lemma}
    \begin{proof}
        Let $p_1$ and $p_2$ be two such ascending product expressions, and suppose $p_1 = p_2$ as group elements. Suppose that the $\Delta_{\Gamma}$-factor of $p_1$ is $\Delta_{\Gamma}^k$, and consider the elements $p_1 \Delta_{\Gamma}^{-k}$ and $p_2 \Delta_{\Gamma}^{-k}$. By construction, the element $p_1 \Delta_{\Gamma}^{-k}$ has support contained in $\bigcup_{i} X_i$, which is $V(\Gamma) - \{t\}$ for some $t$. 
        
        If the $\Delta_{\Gamma}$ factor of $p_2$ is not $\Delta_{\Gamma}^k$, then the support of $p_2 \Delta_{\Gamma}^{-k}$ is necessarily all of $V(\Gamma)$. The support of an element is well-defined, so this implies that $p_1 \Delta_{\Gamma}^{-k} \neq p_2 \Delta_{\Gamma}^{-k}$ and consequently that $p_1 \neq p_2$.

        Thus the $\Delta_{\Gamma}$ terms must be the same. We can repeat the same process for $p_1 \Delta_{\Gamma}^{-k}$ and $p_2 \Delta_{\Gamma}^{-k}$ for the factor of $p_1$ corresponding to each maximal component of $\{A_{X_i}\}$, if their exponents are non-zero in at least one expression, to see that these terms must agree in $p_1$ and $p_2$ as well. Repeating this process at each subsequent level of the marking in ascending order gives the desired result.
    \end{proof}
    
    \begin{corollary}\label{unique p difference}
        Let $\{P_i\}_{i=1}^{N}$ span a maximal simplex in $C_{parab}$. If $g$ and $h$ are two elements of $A_{\Gamma}$ with $g A_{X_i} g^{-1} = P_i = h A_{Y_i} h^{-1}$ for every $i$, then there is a tuple $I = (n_{i_1}, \cdots n_{i_{N+1}})$ that is unique up to the ordering of $X_i$ within the same level such that $I$ represents the exponents of $\Delta_{X_i}$ and $\Delta_{\Gamma}$ in the ascending product $p$ with $gp = h$.
    \end{corollary}

    The element $p$ and its associated ascending product expression are unique for a \emph{pair} of elements $g$ and $h$. We would like to associate a unique $p$ to each individual standardizer. One way to do this is to find a unique element $g$ to which we can compare all $h$, so that an expression $p$ and a collection $\{P_i\}$ uniquely determine $h$. There is a natural choice of comparison element: the canonical positive standardizer.

    \begin{definition}
        Let $\{P_i\}$ be a maximal $C_{parab}$ simplex with canonical positive standardization $\underline{g} A_{X_i} \underline{g}^{-1} = P_i$, and let $h$ be an element such that $P_i = h A_{Y_i} h^{-1}$ for all $i$. By Corollary \ref{unique p difference}, there is a unique ascending product expression $p$ of $\Delta_{X_i}$ and $\Delta_{\Gamma}$ such that $\underline{g}p = h$. We define the power of $\Delta_{X_i}$ in $p$ to be the \emph{projection of $h$ to $P_i$}, and we write $\pi_{P_i}(h)$. 
    \end{definition} 

The following lemma says that these projections behave predictably with respect to right-multiplication by powers of suitable Garside elements.

\begin{lemma}\label{well behaved projections}
    Let $P_i = h A_{Y_i} h^{-1}$ be a maximal $C_{parab}$-simplex, and fix an index $j$. For any $k$, we have
    \begin{align*}
        \pi_{P_{i}}(h \Delta_{Y_j}^k) = \begin{cases}
                                        \pi_{P_i}(h) & \text{ if } i \neq j  \\
                     \pi_{P_i}(h)+k & \text{ if } i=j
                                    \end{cases}
    \end{align*}
\end{lemma}

\begin{proof}
    Let $\underline{g} A_{X_i} \underline{g}^{-1} = P_i$ be the canonical positive standardization of $\{P_i\}$. By Corollary \ref{unique p difference}, there is a unique ascending product $p$ of the $\Delta_{X_j}$ and $\Delta_{\Gamma}$ such that $\underline{g}p = h$. 

    Right-multiplying both sides of the equality $\underline{g}p = h$ by $\Delta_{Y_j}^k$, we obtain \begin{align*}
        \underline{g}p \Delta_{Y_j}^k = h \Delta_{Y_j}^k
    \end{align*}

    Notice that by the construction of $p$ and the fact that we chose to index the $A_{Y_i}$ in agreement with the $A_{X_i}$, we have $p A_{Y_j} p^{-1} = A_{X_{j}}$. Let $p_f$ be the terminal subword of $p$ consisting of all $\Delta_{X_k}$ of level at most the level of $X_j$. Arguing precisely as in the proof of Lemma \ref{conj preserving gen sets} shows that $p_f$ conjugates $A_{Y_j}$ to $A_{X_j}$. Since $p$ is constructed as an ascending product of Garside elements, repeated applications of Lemma \ref{Delta conjugation} show that $p_f$ conjugates $\Delta_{Y_j}$ to $\Delta_{X_j}$.

    Thus $\underline{g} p \Delta_{Y_j}^k = \underline{g} p_s \Delta_{X_j}^k p_f$, where $p_s$ denotes the initial subword of $p$ consisting of the terms which correspond to higher levels of the simplex than $A_{X_j}$. Including the additional $k$ factors of $\Delta_{X_j}^k$ in the existing $\Delta_{X_j}$ term of $p_f$ yields an ascending product $p'$ such that $\underline{g}p' = h \Delta_{Y_j}^k$. The projection of $h\Delta_{Y_j}^k$ to each $P_i$ is then given by examining the $\Delta_{X_i}$ term in $p'$. Comparing the expressions $p$ and $p'$ shows the result.
    \end{proof}

\subsection{Classifying and projecting transverse elements}

In the definition of a marking, we required that each transverse element $Q$ be simultaneously standardizable with the base simplex. This does not necessarily imply that every simultaneous standardizer for the base also standardizes $Q$. However, we will see that there is a unique way to obtain a simultaneous standardizer for both $Q$ and the base simplex from any $g$ which standardizes the base.

We will require the following lemma and its corollary.
\begin{lemma}\label{cancel transversal terms}
    Let $\{A_{X_i}\}$ be a maximal $C_{parab}$-simplex. Let $p$ be an ascending product of $\Delta_{X_i}$ and $\Delta_{\Gamma}$, so $p = \Delta_{X_1}^{i_1} \cdots \Delta_{X_n}^{i_n} \Delta_{\Gamma}^{i_{n+1}}$. Suppose $Q=p A_Y p^{-1}$ is a transverse element for a particular $A_{X_j}$. Then $Q = \Delta_{X_j}^{i_j} A_{Y'} \Delta_{X_j}^{-i_j}$ for some $Y' \subseteq V(\Gamma)$. 
\end{lemma}
\begin{proof}
    Recall that $\Delta_{\Gamma}$ conjugates standard parabolic subgroups to standard parabolic subgroups. Thus up to replacing $Y$ with a different subset $Y'$ of the Artin generators, we can assume the $\Delta_{\Gamma}$-factor in $p$ is 0. Our goal is to show that the same is true for all terms in the ascending product except $\Delta_{X_j}$ itself.


    Since $p A_Y p^{-1}$ is a transverse element, every $A_{X_i}$ for $i \neq j$ either contains $p A_Y p^{-1}$, commutes with $p A_Y p^{-1}$, or is contained in $p A_Y p^{-1}$. Notice that if $A_{X_i}$ is at a higher level of the marking than $A_{X_j}$, then $A_{X_i}$ cannot contain $p A_{Y}p^{-1}$ by Lemma \ref{lower level transversal inclusion}.

    First, each minimal $A_{X_i}$ is either contained in $p A_Y p^{-1}$ or commutes with it. In either case, $p A_Y p^{-1}$ is preserved under conjugation by any power of $\Delta_{X_i}$. Conjugating by the inverse of the power of $\Delta_{X_i}$ which appears in $p$ thus yields a new ascending product $p'$ with $p' A_Y p'^{-1} = p A_Y p^{-1}$ and such that $p'$ is precisely $p$ without its $\Delta_{X_i}$-factors.

    Repeating this process at each subsequent level yields an element $p'$ with $p' A_Y p'^{-1} = p A_Y p^{-1}$ and such that $p'$ is precisely $p$ with all factors at levels higher than $A_{X_j}$ removed. By same process, we can remove any non-$\Delta_{X_j}$ terms at the same level as $A_{X_j}$. If $X_j$ is maximal, then this concludes the argument. Otherwise, we proceed to removing the terms at lower levels.

    Conjugating by any power of $\Delta_{X_j}$ itself results in a simplex whose base elements are the same as the original simplex at all levels appearing in $p'$. This implies that if $p''$ is $p'$ without the $\Delta_{X_j}$-factor, then $p'' A_Y p''^{-1}$ is a transverse element for $A_{X_j}$ in a marking where the bottom levels agree with those of $\{A_{X_i}\}$. If we can show that any such $p'' A_{Y} p''^{-1}$ is equal to $A_{Y'}$ for some $Y'$, then we are done.
    
    We begin with the rightmost term of $p''$. Notice that conjugating by $(p'')^{-1}$ stabilizes the maximal elements of the base collection, since $p''$ has no $\Delta_{\Gamma}$-factors. Exactly one of these maximal elements contains $A_{X_j}$, and that maximal element must also contain $p'' A_Y p''^{-1}$ by Lemma \ref{lower level transversal inclusion}. Since conjugation preserves subgroup inclusion and conjugating by $p''^{-1}$ stabilizes the maximal elements, this subgroup also contains $p''^{-1} (p'' A_Y p''^{-1}) p'' = A_Y$. Every other maximal subgroup then necessarily commutes with both $p'' A_Y p''^{-1}$ and $A_Y$. 
    
    Let $A_{X_i}$ be a maximal subgroup. In the commuting case, any power of the Garside element $\Delta_{X_i}$ normalizes $A_Y$. If $A_Y \leq A_{X_i}$, then $\Delta_{X_i}$ conjugates $A_Y$ to some different standard parabolic subgroup $A_{Y'}$. Thus we can remove all factors of $p''$ corresponding to maximal components by possibly replacing $A_Y$ with $A_{Y'}$. We apply the same argument on each subsequent level and remove all terms of $p''$ to eventually obtain $p'' A_Y p''^{-1} = A_{Y'}$.
\end{proof}

We obtain the following corollary.

\begin{corollary}\label{transversal k existence}
    Let $M$ be a marking with base elements $\{P_i\}$, and let $Q$ be the transverse element for a particular $P_j$. Let $g$ be a simultaneous standardizer for $\{P_i\}$ with $g A_{X_i} g^{-1} = P_i$ for each $i$. Then there is a unique integer $k$ and a unique standard parabolic subgroup $A_Y$ such that $Q = g\Delta_{X_j}^k A_{Y} \Delta_{X_j}^{-k}g^{-1}$.
\end{corollary}
    \begin{proof}
        By the third property of a marking, we can choose a simultaneous standardizing element $h$ for the collection $\{\{P_i\}, Q\}$ with $P_i = h A_{Z_i} h^{-1}$ for each $i$ and $Q = h A_Y h^{-1}$.  By Proposition \ref{differ by p}, there is an ascending product $p$ of the $\Delta_{X_i}$ such that $h = gp$. By Lemma \ref{cancel transversal terms}, there are some $k$ and $Y'$ such that $Q = g\Delta_{X_j}^{k} A_{Y'} \Delta_{X_j}^{-k}g^{-1}$.

        To see that such an expression is unique, suppose that there are two such expressions
    \begin{align*}
        g \Delta_{X_j}^k A_Y \Delta_{X_j}^{-k} g^{-1} = Q = g \Delta_{X_j}^{l} A_{Y'} \Delta_{X_j}^{-l} g^{-1}\text{.}
    \end{align*}
    Conjugating both sides by $\Delta_{X_j}^{-l}g^{-1}$ gives the equality 
    \begin{align*}
        \Delta_{X_j}^{k-l} A_Y \Delta_{X_j}^{l-k} = A_{Y'}\text{.} 
    \end{align*}
    Observe that $z_{X_j}$ cannot commute with $z_{Y}$. If it did, then conjugating both $z_{X_j}$ and $z_Y$ by $g\Delta_{X_i}^k$ would still yield commuting elements, and this would contradict transversality of $Q$ and $P_j$. Lemma \ref{conjugation implies containment} then implies that $k-l = 0$. 
    \end{proof}

We can now prove the following.

\begin{proposition}\label{prop:last-before-proj-def}
    Let $\{P_i\}$ be the base of a marking $M$, and let $Q$ be the transverse element for some $P_j$. Let $h$ be a choice of standardizer for the base, and let $k_h$ and $Y$ be the unique integer and standard parabolic subgroup provided by Corollary \ref{transversal k existence} such that $Q = h \Delta_{X_j}^{k_h} A_Y \Delta_{X_j}^{-k_h} h^{-1}$. The value \begin{align*}\tag{$\star\star$}
        l = \pi_{P_j}(h \Delta_{X_j}^{k_h}) 
    \end{align*}
    is independent of the choice of standardizing element $h$.
\end{proposition}
\begin{proof}
    First, observe that the projection is well-defined because $h \Delta_{X_j}^{k_h}$ is also a standardizer for the simplex $\{P_i\}$, although the standardization may differ at levels below $X_j$. Indeed for each $i$, we have the following. 
    \begin{align*}
            \Delta_{X_j}^{-k_h} h^{-1} P_i h \Delta_{X_j}^{k_h} &= \Delta_{X_j}^{-k_h} A_{X_i} \Delta_{X_j}^{k_h}\text{.}
        \end{align*}
        
        If $A_{X_i}$ contains or commutes with $A_{X_j}$ or if $i =j$, then $\Delta_{X_j}^{-k_j} A_{X_i} \Delta_{X_j}^{k_j}$ is precisely $A_{X_i}$. If $A_{X_i}$ is contained in $A_{X_j}$, then $\Delta_{X_j}^{-k_h} A_{X_i} \Delta_{X_j}^{k_h}$ is some other standard parabolic subgroup $A_{X'_i}$ which is also contained in $A_{X_j}$. 

    Let $\{\underline{g} A_{X_i} \underline{g}^{-1}\}$ be the canonical positive standardization of the base simplex, and let $\{h A_{Y_i} h^{-1}\}$ be another standardization, both indexed so that $P_i = \underline{g} A_{X_i} \underline{g}^{-1} = h A_{Y_i} h^{-1}$ for each $i$. Let $Q$ be the transverse element for $\underline{g} A_{X_j} \underline{g}^{-1} = h A_{Y_j} h^{-1}$. 
    
    By Proposition \ref{differ by p}, there is a unique ascending product $p$ of $\Delta_{X_i}$ and $\Delta_{\Gamma}$ such that $\underline{g}p = h$. If $p_j$ is the exponent of the $\Delta_{X_j}$-factor of $p$, then by definition, $\pi_{P_j} (h) = p_j$. 

    By Corollary \ref{transversal k existence}, there are some integers $k_g$ and $k_h$ such that
    \begin{align*}
        Q &= \underline{g}\Delta_{X_j}^{k_g} A_Y \Delta_{X_j}^{-k_g} \underline{g}^{-1}\\
        &= h \Delta_{Y_j}^{k_h} A_{Y'} \Delta_{Y_j}^{-k_h} h^{-1}\\
        &= \underline{g}p \Delta_{Y_j}^{k_h} A_{Y'} \Delta_{Y_j}^{-k_h} p^{-1} \underline{g}^{-1}\text{.}
    \end{align*} 
    By the construction of $p$, if $p_s$ is the initial subword of $p$ containing all factors at levels strictly lower than $A_{X_j}$ and $p_f$ is the terminal subword of $p$ containing all factors at levels at least as large as $A_{X_j}$, then $p \Delta^{k_h}_{Y_j} = p_s \Delta_{X_j}^{k_h} p_f$. We can combine $X_j$ terms to obtain an ascending product of the $\Delta_{X_i}$ with additional terms in the $\Delta_{X_j}$ factor. Notice that $\underline{g}^{-1} Q \underline{g}$ is a transverse element for $A_{X_j}$ with respect to $\{A_{X_i}\}$, so applying Lemma \ref{cancel transversal terms} yields the following.
    
    \begin{align*}
        Q &= \underline{g} \Delta_{X_j}^{p_j + k_h} A_{Y''} \Delta_{X_j}^{-p_j - k_h} \underline{g}^{-1}\text{.}
    \end{align*}
    The uniqueness part of Corollary \ref{transversal k existence} implies that $Y'' = Y$ and $p_j + k_h = k_g$. Thus applying \cref{well behaved projections}, we obtain the following.
    \begin{align*}
        k_g &= p_j + k_h\\
        &= \pi_{P_j}(h) +k_h\\
        &= \pi_{P_j}(h \Delta_{Y_j}^{k_h})\text{.}
    \end{align*}
\end{proof}

Notice that if the base collection $\{P_i\}$ is standard, the canonical positive standardizer is the identity. If $Q$ is also standard, then $0= k_e$, so $\pi_{P_j}(Q) = 0$.

We can now give the following definition. 
\begin{definition}\label{def:projections}
    Let $\{P_i\}$ be the base of a marking $M$, and let $Q$ be the transverse element for $P_j$. The \emph{projection of $Q$ to $P_j$}, denoted $\pi_{\{P_i\}}(Q)$, is the integer $l$ in ($\star\star$).
\end{definition}
Equivalently, the projection of $Q$ to $P_j$ with respect to $\{P_i\}$ is the unique integer $k$ such that if $\{\underline{g} A_{X_i} \underline{g}^{-1}\}$ is the canonical positive standardization of $\{P_i\}$, then $\underline{g}\Delta_{X_j}^k$ is a standardizer for $Q$. We adopt the notation $\pi_{\{P_i\}}$ rather than $\pi_{P_j}$ to remind ourselves that this projection depends on the entire base simplex; it is unambiguous because a parabolic subgroup $Q$ is a transverse element for at most one $P_j$ in any simplex $\{P_i\}$. 

    \begin{figure}[h]
        \centering
        \includegraphics[width=8cm]{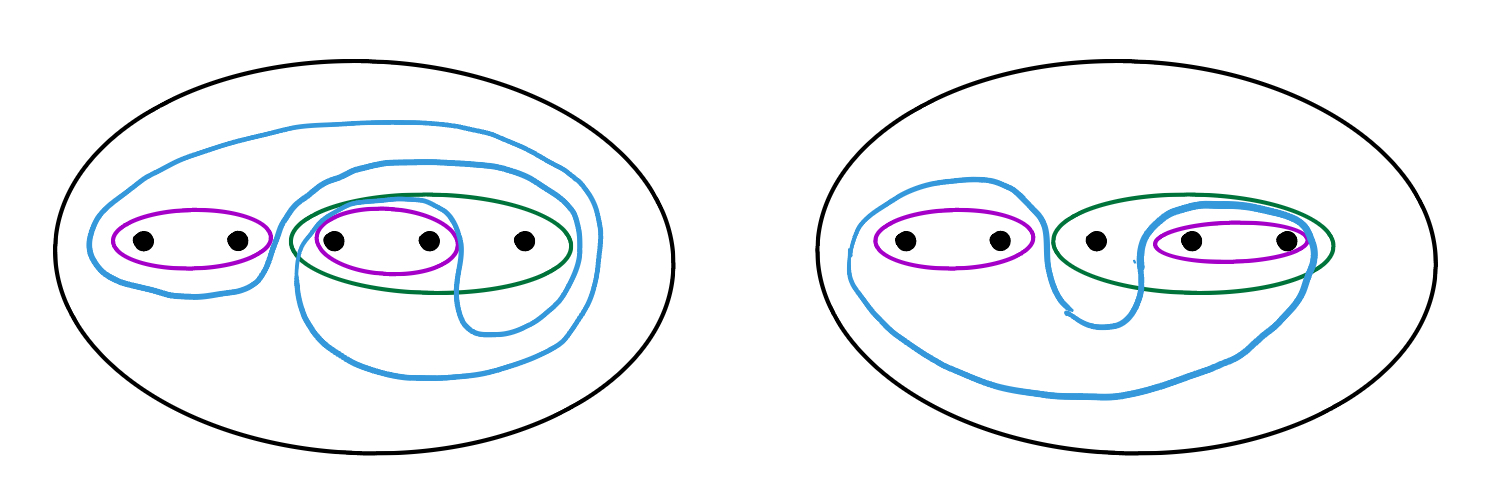}
        \captionsetup{margin=.5cm,justification=centering}
        \caption{A maximal $C_{parab}(A_n)$ simplex (purple and green) with two transverse elements $Q_1$ and $Q_2$ (blue) for the green element $P$ with $l = 2$ and $l = -1$}
    \end{figure}

The domain of these projections can be extended slightly, and the extension will be useful in Section \ref{sec:elementary moves}. The projection map $\pi_{\{P_i\}}$ may depend on the choice of base simplex containing $P_i$, since it is defined with respect to the canonical positive standardizer of the entire base simplex. If one interchanges the roles of some $P_j$ and $Q_j$ in a marking, the base is still a $C_{parab}$-simplex, and we would like a way of projecting transverse elements with respect to the new base that is naturally compatible with the projections to the original base defined above.

Recall that if $\{g A_{X_i} g^{-1}\} = \{P_j\}$ is the base of a marking $M$ and $Q = g \Delta_{X_j}^{k} A_Y \Delta_{X_j}^{-k} g^{-1}$ is a transverse element for $P_j$, then $g \Delta_{X_j}^k$ is a standardizer for $\{P_1 , \cdots , P_{j-1}, Q , P_{j+1} , \cdots\}$. 

\begin{definition}\label{def:extended-proj}
    Let $\{g A_{X_i} g^{-1}\} = \{P_j\}$ be the base of a marking $M$, and let $Q = g \Delta_{X_j}^{k_j} A_Y \Delta_{X_j}^{-k_j} g^{-1}$ be a transverse element for $P_j$. If $Q_i'$ is a transverse element for $P_i$ with respect to the base simplex $\{P_1 , \cdots , P_{j-1}, Q , P_{j+1} , \cdots\}$, then for $i \neq j$ we define
    \begin{align*}
        \tilde{\pi}_{\{P_i\}}(Q_i') = \pi_{P_i}(g\Delta_{X_j}^{k_j}) + k_i = \pi_{P_i}(g) + k_i
    \end{align*}
    where $k_i$ is the unique integer such that $Q_i' = g \Delta_{X_j}^{k_j} \Delta_{X'_i}^{k_i} A_{Y'_i} \Delta_{X_i'}^{-k_i}\Delta_{X_j}^{-k_j}g^{-1}$, $\{A_{X'_i}\}$ is the standardization of $\{P_1 , \cdots , P_{j-1}, Q , P_{j+1} , \cdots\}$ via $g \Delta_{X_j}^{k_j}$, and $\pi_{P_i}$ is the projection to $P_i$ with respect to $\{P_i\}$.
\end{definition}

The extended projection $\tilde{\pi}_{\{P_i\}}$ is well-defined for precisely the same reason $\pi_{\{P_i\}}$ itself is; the integer $k_i$ in such an expression is unique up to the choice of $g$, and $\Delta_{X_j}^{k_j} \Delta_{X_i'}^{k_i} = \Delta_{X_i}^{k_i} \Delta_{X_j}^{k_j}$ by the construction of $X_i'$, so choosing $g'$ with a different $\pi_{\{P_i\}}(g')$ yields a corresponding change in $k_i$. 

The projection $\tilde{\pi}_{\{P_i\}}$ depends on the original choice of base simplex $\{P_i\}$, and its domain depends on the choice of transverse element $Q$. In practice, the choice of transverse element will always be clear from context, so we do not include it explicitly in the notation.

As a brief aside, one might ask whether $\pi_{\{P_i\}}(g\Delta_{X_j}^{k_j})$ and $\pi_{\{P_1 , \cdots, Q , \cdots\}}(g\Delta_{X_j}^{k_j})$ are equal at indices $i \neq j$. We believe that this is likely not the case in general. The complication is that while the element $\underline{g}\Delta_{X_j}^{k_j}$ is a standardizer for the simplex $\{P_1 , \cdots , Q , \cdots\}$, it is not necessarily the canonical positive one. If $k_j$ is negative, for example, $\underline{g}\Delta_{X_j}^{k_j}$ is probably not even positive, and it is plausible that obtaining a positive standardizer might require the introduction of factors of higher-level $\Delta_{X_i}$. Similarly, if $k_j$ is positive, then letting $X'$ represent the standardization of the next level of the simplex inside of $X_j$, the element $\underline{g}\Delta_{X_j}^{k_j} \Delta_{X'}^{-k_j}$ is also a standardizer. The latter element seems more likely to be the canonical positive standardizer than $\underline{g}\Delta_{X_j}^{k_j}$. We suspect that the canonical positive standardizer is of the aforementioned form when $k_j$ is positive or of the form $\underline{g}\Delta_{X_j}^{k_j} \Delta_{X_k}^{-k_j}$ or $\underline{g}\Delta_{X_j}^{k_j} \Delta_{\Gamma}^{-k_j}$ when $k_j$ is negative, depending on whether $\underline{g}A_{X_j} \underline{g}^{-1}$ is maximal, but such a classification is not necessary for the remainder of the paper.

\section{Marking stabilizers}\label{sec:marking stabilizers}

Our classification of transverse elements yields the following.

\begin{proposition}\label{markings conj to standard}
    Every marking is conjugate to one in which all base and transverse elements are standard.
\end{proposition}
\begin{proof}
    Let $M = \{(P_i , Q_i)\}$ be a marking. The base is simultaneously standardizable by definition, so $M$ is conjugate to a marking $\{(A_{X_i} , Q'_i)\}$. We refer to this new marking as $M_1$. By Lemma \ref{transversal k existence}, there are some integers $k_i$ and standard parabolic subgroups $A_{Y_i}$ such that $Q'_i = \Delta_{X_i}^{k_i} A_{Y_i} \Delta_{X_i}^{k_i}$ for every $i$.

    Consider an $A_{X_j}$ at the the highest level of the base simplex. Since it is minimal, $A_{X_j}$ is either contained in $A_{X_i}$ or commutes with $A_{X_i}$ for every $i \neq j$. In particular, conjugation by any power of $\Delta_{X_j}$ normalizes $A_{X_i}$ for every $A_{X_i}$ in the base simplex. 

    Similarly, by Lemma \ref{lower level transversal inclusion}, $A_{X_j}$ is either contained in $Q'_i$ or commutes with $Q'_i$ for every $i \neq j$. This again implies that conjugation by any power of $\Delta_{X_j}$ normalizes $Q_{i}'$ for $i \neq j$. In particular, conjugating $M_1$ by $\Delta_{X_j}^{-k_j}$ for each minimal $A_{X_j}$ results in a marking $M_2$ which is exactly the same as $M_1$ except for the transverse elements associated to the minimal elements, which in $M_2$ have been replaced with the standard parabolic subgroups $A_{Y_j}$.

    Now consider an $A_{X_j}$ in the next level of the base simplex. The same reasoning as above shows that conjugation by any power of $\Delta_{X_j}$ normalizes $A_{X_i}$ and $Q_i'$ whenever $X_i$ is either in the same level as $A_{X_j}$ or in a lower level than $A_{X_j}$. 

    If $A_{X_i}$ is contained in $A_{X_j}$, then Lemma \ref{lower level transversal inclusion} shows that $A_{Y_i}$ is also contained in $A_{X_j}$. Then conjugation by any power of $\Delta_{X_j}$ sends both $A_{X_i}$ and $A_{Y_i}$ to some other standard parabolic subgroups of $A_{X_j}$. 
    
    If $A_{X_i}$ is at a higher level of the base simplex than $A_{X_j}$ but is not contained in $A_{X_j}$, then the same reasoning shows that both $A_{X_i}$ and $A_{Y_i}$ are contained in some other base element at the same level of the simplex as $A_{X_j}$. In particular, both commute with $A_{X_j}$. Then conjugating $M_2$ by $\Delta_{X_j}^{-k_j}$ for each $X_j$ in this level of the marking replaces $M_2$ with a marking $M_3$ which has all standard base elements and standard transverse elements for the first two levels.

    Repeating this process for each level of the marking proves the desired result.
\end{proof}

We can now classify stabilizers of markings.

\begin{theorem}\label{vx stabilizer}
        Let $M$ be a marking on a finite-type Artin group $A_{\Gamma}$. The stabilizer of $M$ is contained in a conjugate of $\langle \Delta_{\Gamma} \rangle$.
\end{theorem}

    \begin{proof}
        By Proposition \ref{markings conj to standard}, it suffices to show that the stabilizer of any marking of the form $\{(A_{X_i} , A_{Y_i})\}$ is contained in $\langle \Delta_{\Gamma} \rangle$. Suppose without loss of generality that the base elements are ordered such that if $j > i$, either $P_i \leq P_j$ or $P_i \cap P_j = \varnothing$ and the elements $z_{P_i}$ and $z_{P_j}$ commute. 

        Since $g$ stabilizes the base simplex, Theorem \ref{almost ascending product} shows that $g$ is an ascending product of the Garside elements of the $\{A_{X_i}\}$, i.e., there are some integers $n_i$ such that $g = \Delta_{X_1}^{n_1} \cdots \Delta_{X_N}^{n_N}\Delta_{\Gamma}^{n_{N+1}}$. 
        
        Let $\sigma$ be the permutation of $\{1, \cdots, N\}$ such that $g A_{X_i} g^{-1} = A_{X_{\sigma(i)}}$. Since $g$ also preserves the collection of transverse elements, it is clear from the definition of transversality that $g A_{Y_i} g^{-1} = A_{Y_{\sigma(i)}}$. The element $g$ is an ascending product, so Lemma \ref{cancel transversal terms} implies that for any $i$, $g A_{Y_i} g^{-1} = \Delta_{X_{\sigma(i)}}^{n_{\sigma(i)}} A_{Y'} \Delta_{X_{\sigma(i)}}^{-n_{\sigma(i)}}$ for some $Y' \subseteq V(\Gamma)$. 

        By assumption, however, $g A_{Y_i} g^{-1} = A_{Y_{\sigma(i)}}$ and $z_{X_{\sigma(i)}}$ does not commute with $z_Y'$. By Lemma \ref{conjugation implies containment}, $\Delta_{X_{\sigma(i)}}^{n_\sigma(i)} A_{Y'} \Delta_{X_{\sigma(i)}}^{-n_{\sigma(i)}} = A_{Y_{\sigma(i)}}$ implies $n_{\sigma(i)} = 0$. 

        Applying this argument to each non-$\Delta_{\Gamma}$-factor of $g$ shows that $n_i = 0$ for each $i \in \{1, \cdots, N\}$. Thus $g \in \langle \Delta_{\Gamma} \rangle$.
    \end{proof}

    The stabilizer of a marking may be a proper subgroup of a conjugate of $\langle \Delta_{\Gamma} \rangle$; in general, conjugation by $\Delta_{\Gamma}$ does not even preserve maximal $C_{parab}$-simplices.

         \begin{figure}[h!]
            \centering
            \includegraphics[width=8cm]{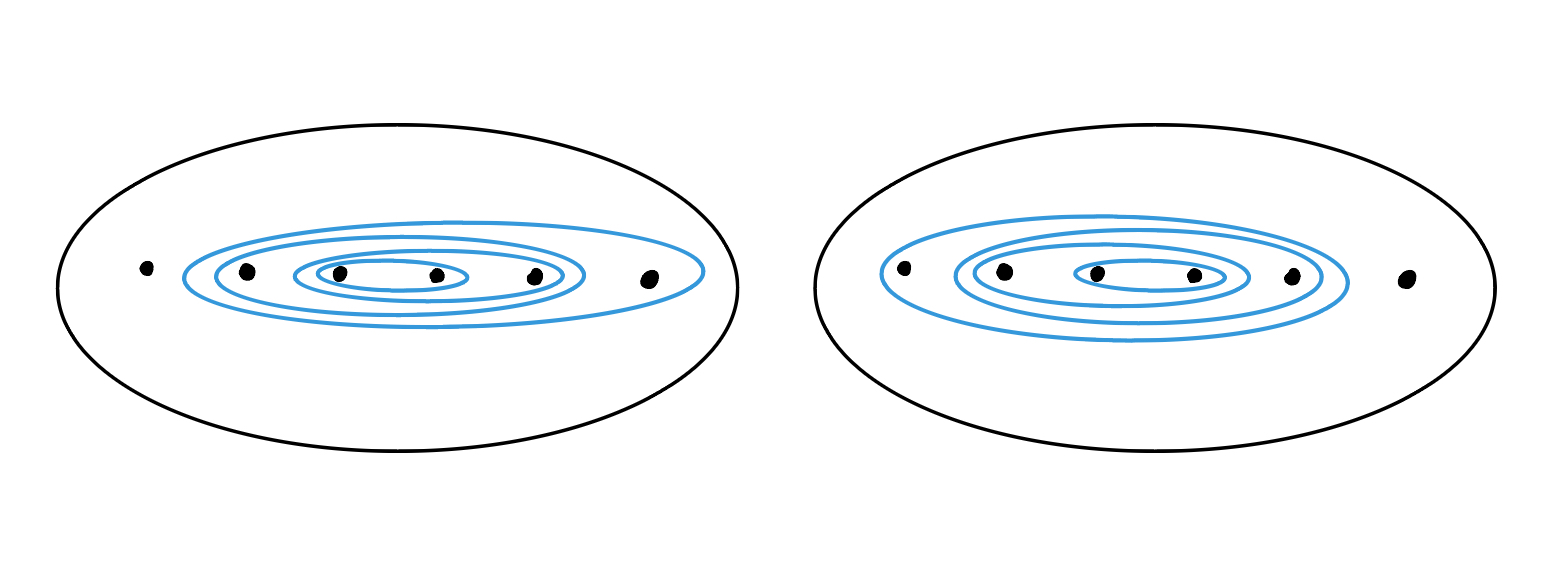}
            \captionsetup{margin=.5cm,justification=centering}
            \caption{Two collections of curves whose associated $C_{parab}$ simplices are maximal and conjugated to one another by $\Delta_{\Gamma}$.}
        \end{figure}

\section{Elementary moves}\label{sec:elementary moves}
\subsection{Definition}
In the mapping class group, edges in the marking graph are defined via two types of elementary moves on markings: twist moves and flip moves. In this section, we define twist and flip moves between markings on finite-type Artin groups which are inspired by the surface analogues. 

Traditionally, two markings on a surface are connected via a twist move if one can be obtained from the other by Dehn twisting a transverse curve by a base curve. This move has an intuitive analogue in our setting.
    \begin{definition}\label{def:twist}
        Two markings $M_1$ and $M_2$ on $A_{\Gamma}$ are connected via a \emph{twist} move if one can be obtained from the other by replacing some transverse element $Q_i$ with $z_{P_i} Q_i z_{P_i}^{-1}$ where $P_i$ is the associated base element.
    \end{definition}

It is straightforward to check that $M_2$ is indeed a marking with the same base collection as $M_1$. We could have equivalently defined a twist move to replace \emph{every} parabolic subgroup in the marking with its conjugate by $z_{P_i}$, since $z_{P_i}$ normalizes every element of the marking except $Q_i$.

The second kind of move is somewhat more complex. In the mapping class group, a flip move would interchange a $(P_i , Q_i )$ pair and then replace $Q_j$ where $j \neq i$ with some suitable choices which ensure that we still have a marking and where the new choices of $Q_j$ have appropriate winding numbers with the other curves in the marking. In this setting, we replace this with the projections $\pi_{\{P_i\}}$ and $\tilde{\pi}_{\{P_i\}}$ from Definitions \ref{def:projections} and \ref{def:extended-proj} to obtain the following analogue.

    \begin{definition}\label{def:flip}
        Let $M_1 = \{(P_i , Q_i)\}$ be a marking on $A_{\Gamma}$, and let $n = |\{P_i\}|$. We say that $M_1$ is connected by a \emph{flip} move to $M_2$ if
        \begin{align*}
            M_2 = \{(P_1, Q'_1)\cdots, (P_{i-1}, Q'_{i-1}), (Q_i, P_i),  (P_{i+1}, Q'_{i+1}), \cdots, (P_n, Q'_n)\}
        \end{align*} 
        with $|\pi_{\{P_i\}}(Q_j) - \tilde{\pi}_{\{P_i\}}(Q'_j)| \leq 1$ for all $j \neq i$. 
    \end{definition}

We are now ready to define the marking graph of a finite-type Artin group.
    \begin{definition}\label{def:marking graph}
        The marking graph $\mathcal{W}$ for $A_{\Gamma}$ is the graph with vertex set equal to the set of markings on $A_{\Gamma}$ and an edge between vertices $M_1$ and $M_2$ if the corresponding markings are connected by either a twist move or a flip move.
    \end{definition}
    
\subsection{A locally finite graph}
It follows from the finite dimensionality of $C_{parab}$ and uniqueness of the element $z_{P_i}$ for a given $P_i$ that there are finitely many markings $M'$ which can be obtained from a given marking $M$ via twist moves. Specifically, there are at most $2(D+1)$ where $D$ is the dimension of $C_{parab}$. In fact, there are also finitely many markings $M'$ which can be obtained from a given marking $M$ via flip moves, as we will see in Proposition \ref{finite flips}. We begin with two preliminary lemmas.

\begin{lemma}\label{unique transversal with a proj}
    Let $\{P_i\}$ be a maximal simplex. For each integer $n$, there is at most one parabolic subgroup $Q$ such that $Q$ is a transverse element for a particular $P_j$ with $\pi_{\{P_i\}}(Q) = n$. 
\end{lemma}
\begin{proof}
    Let $\{\underline{g} A_{X_i} \underline{g}^{-1}\}$ be the canonical positive standardization of $\{P_i\}$. Consider any subgroup $Q$ as in the statement. By Lemma \ref{transversal k existence}, there is a unique $A_Y$ such that $Q = \underline{g} \Delta_{X_j}^n A_Y \Delta_{X_j}^{-n} \underline{g}^{-1}$. 

    Suppose there is some other $Q'$ with $\pi_{\{P_j\}}(Q) = \pi_{\{P_j\}}(Q')$. The subgroup $Q'$ admits an expression of the same form, and by the definition of the projection, this expression is $Q' = \underline{g} \Delta_{X_j}^n A_{Y'} \Delta_{X_j}^{-n}\underline{g}^{-1}$. The subgroups $A_Y$ and $A_{Y'}$ are both transverse elements for $A_{X_j}$ with respect to the standard simplex $\{\Delta_{X_j}^{-n} A_{X_i} \Delta_{X_j}^n\}$. Lemma \ref{lem:std-transv-uniqueness} then implies $Y = Y'$, so $Q = Q'$.
\end{proof}
\begin{corollary}\label{cor:unique-transv-proj}
    Let $\{P_i\}$ be a maximal simplex, and let $Q$ be a transverse element for $P_j$. For each integer $n$ and each $i \neq j$, there is at most one parabolic subgroup $Q'_i$ such that $Q_i'$ is a transverse element for $P_i$ with respect to the base collection $\{P_1 , \cdots, P_{j-1}, Q , P_{j+1}, \cdots\}$ with $\tilde{\pi}_{\{P_i\}}(Q'_i) = n$.
\end{corollary}
\begin{proof}
    As in the proof of the lemma, there is some integer $k$ such that $Q = \underline{g} \Delta_{X_j}^k A_Y \Delta_{X_j}^{-k} \underline{g}^{-1}$. The element $\underline{g} \Delta_{X_j}^k$ is a standardizer for $\{P_1 , \cdots, P_{j-1}, Q ,$ $ P_{j+1}, \cdots\}$ with $\pi_{\{P_i\}}(\underline{g} \Delta_{X_j}^k)=0$, so any potential choices $Q_i'$ and $Q_i''$ must be of the form $Q_i' = \underline{g} \Delta_{X_j}^k \Delta_{X'_i}^n A_{Y'} \Delta_{X'_i}^{-n}\Delta_{X_j}^{-k} \underline{g}^{-1}$ and $Q_i' = \underline{g} \Delta_{X_j}^k \Delta_{X'_i}^n A_{Y''} \Delta_{X'_i}^{-n}\Delta_{X_j}^{-k} \underline{g}^{-1}$ for some $A_{Y'}$ and $A_{Y''}$. The result then follows precisely as in the lemma.
\end{proof}


\begin{lemma}\label{exists a transversal with a proj}
    Let $\{P_i\}$ be a maximal simplex. For each integer $n$, there is at least one parabolic subgroup $Q$ such that $Q$ is an allowable transverse element for a particular $P_j$ and either $\pi_{\{P_i\}}(Q) = n$ or $\pi_{\{P_i\}}(Q) = n-1$.
\end{lemma}

\begin{proof}
    Let $\underline{g} A_{X_i} \underline{g}^{-1}$ be the canonical standardization of $\{P_i\}$. By Proposition \ref{marking existence}, there is some standard transverse element $A_Y = A_{Y_j}$ for $A_{X_j}$ with respect to the collection $\{A_{X_i}\}$. If $\Delta_{X_j}^n$ is a central power of $\Delta_{X_j}$ in $A_{X_j}$, then let $Q$ be the conjugate of $A_{Y_j}$ by $\underline{g}\Delta_{X_j}^n$. Conjugation by $\Delta_{X_j}^n$ normalizes each element in $\{A_{X_i}\}$, so $Q$ is a transverse element for $P_j$ in $\{P_i\}$ with $\pi_{\{P_i\}}(Q) = n$.
    
    If $\Delta_{X_j}^n$ is not a central power, then $\Delta_{X_j}^{n-1}$ must be. Then instead conjugate $A_{Y_j}$ by $\underline{g}\Delta_{X_j}^{n-1}$, and $\pi_{\{P_i\}}(Q) =n -1$.
\end{proof}

\begin{corollary}\label{cor:exists-trasnv-proj}
    Let $\{P_i\}$ be a maximal simplex, and let $Q$ be a transverse element for $P_j$. For each integer $n$ and each $i \neq j$, there is at least one parabolic subgroup $Q'_i$ such that $Q_i'$ is a transverse element for $P_i$ with respect to the base collection $\{P_1 , \cdots, P_{j-1}, Q , P_{j+1}, \cdots\}$ and either $\tilde{\pi}_{\{P_i\}}(Q'_i) = n$ or $\tilde{\pi}_{\{P_i\}}(Q'_i) = n-1$.
\end{corollary}
\begin{proof}
    Let $\underline{g}$ be the canonical positive standardizer for $\{P_i\}$, and let $k_j$ be the integer such that $\underline{g}\Delta_{X_j}^{k_j} A_{X'_i}\Delta_{X_j}^{-k_j}\underline{g}^{-1} = P_i$ for each $i \neq j$ and $\underline{g}\Delta_{X_j}^{k_j} A_{X'_j}\Delta_{X_j}^{-k_j}\underline{g}^{-1}$ $ = Q$. We may choose $A_{Y_i}$ to be a standard transverse element for $A_{X'_i}$, and the result then follows precisely from conjugating by either $\underline{g}\Delta_{X_j}^{k_j} \Delta_{X_i'}^{n}$ or $\underline{g}\Delta_{X_j}^{k_j} \Delta_{X_i'}^{n-1}$ as in the lemma.
\end{proof}

\begin{proposition}\label{finite flips}
    Given a marking $M$, there is at least one and at most $3^{D}(D+1)$ possible markings $M'$ which are obtained from $M$ via a flip move, where $D$ is the dimension of $C_{parab}$. 
\end{proposition}

\begin{proof}
    Let $M = \{(P_i , Q_i)\}$ be a marking, and let $M' = \{(P'_i , Q'_i)\}$ be the result of a flip move across index $j$. 

    Corollaries \ref{cor:unique-transv-proj} and \ref{cor:exists-trasnv-proj} imply that for each index $i \neq j$, there are at least one and at most $3$ suitable choices of $Q'_i$. There are $D$ indices $i$ for which we must choose a $Q'_i$, since we do not need to make a choice at index $j$. Each possible $M'$ is characterized entirely by the collection $\{Q'_i\}$, so an upper bound for the number of possible markings obtained via a flip move across index $j$ is $3^{D}$.

    There are $D+1$ choices of index $j$ we could flip across, so the total number of $M'$ is $3^{D}(D+1)$.
\end{proof}

Recall that there are finitely many possible results of a twist move performed on a given marking. We then obtain the following as a corollary of Proposition \ref{finite flips}.
\begin{corollary}
    The marking graph of a finite-type Artin group is locally finite.
\end{corollary}

\section{Action by isometries}\label{sec:by isometries}

The goal of this section is to show that the action of $A_{\Gamma}/Z(A_{\Gamma})$ on the marking graph is by isometries, i.e., that if $M_1$ and $M_2$ are connected by an edge, so are $x \cdot M_1$ and $x \cdot M_2$ for any $x \in A_{\Gamma}/Z(A_{\Gamma})$. In the twist move case, this is straightforward.

     \begin{lemma}\label{lem:twist-isom}
        Let $x \in A_{\Gamma} / Z(A_{\Gamma})$, and let $M_1$ and $M_2$ be two markings which are connected via a twist move. Then $x \cdot M_1$ and $x \cdot M_2$ are also connected via a twist move.
    \end{lemma}
    \begin{proof}
        Let $M_1 = \{(P_i , Q_i)\}$ and $M_2 = \{(P_1 , Q_1), \cdots, (P_{j-1}, Q_{j-1}),$\\ $(P_j , z_{P_j} Q_{j} z_{P_j}^{-1}), (P_{j+1}, Q_{j+1}), \cdots\}$. To see that $x \cdot M_1$ and $x \cdot M_2$ are connected by a twist edge, it suffices to check that $x z_{P_j} Q_{j} z_{P_j}^{-1} x^{-1}$ is equal to $z_{x P_j x^{-1}} x Q_{j}x^{-1} z_{x P_j x^{-1}}^{-1}$. 
        
        It was shown in \cite{Cumplido-minimal-standardizers} that $z_{x P_i x^{-1}} = x z_{P_i} x^{-1}$. Then 
            \begin{align*}
                z_{x P_j x^{-1}} x Q_{j}x^{-1} z_{x P_j x^{-1}}^{-1} &= x z_{P_j} x^{-1} x Q_{j}x^{-1} x z_{P_j}^{-1} x^{-1}\\
                &= x z_{P_j} Q_{j} z_{P_j}^{-1} x^{-1}
            \end{align*}
        as desired.
    \end{proof}

We next consider edges corresponding to flip moves.
    
    \begin{proposition}\label{prop:flip-isom}
        Let $x \in A_{\Gamma} / Z(A_{\Gamma})$, and let $M_1$ and $M_2$ be two markings which are connected via a flip move. Then $x \cdot M_1$ and $x \cdot M_2$ are also connected via a flip move.
    \end{proposition}

    \begin{proof}
        Let $M_1 = \{(P_i , Q_i)\}$. Suppose the flip move is across index $j$, so 
        \begin{align*}
            M_2 = \{(P_1, Q'_1), \cdots, (P_{j-1}, Q'_{j-1}), (Q_j, P_j), (P_{j+1}, Q'_{j+1}) \cdots\}\text{.}
        \end{align*}
        Consider $x \cdot M_1$ and $x \cdot M_2$. It is clear that the base collections and the index $j$ pairs still have the appropriate form, so it suffices to check that $|\pi_{\{x P_i x^{-1}\}} (x Q_i x^{-1}) - \tilde{\pi}_{\{x P_i x^{-1}\}} (x Q'_i x^{-1})| = |\pi_{\{P_i\}}(Q_i) - \tilde{\pi}_{\{P_i\}}(Q'_i)|$ for each $i \neq j$.

        Let $\underline{g}$ be the canonical positive standardizer and $\{\underline{g} A_{X_i} \underline{g}^{-1}\}$ be the canonical positive standardization for the original base collection $\{P_i\}$, and let $k_i = \pi_{\{P_i\}}(Q_i)$ and $k'_i = \tilde{\pi}_{\{P_i\}}(Q'_i)$ for each $i$. 
        
        Lemma \ref{unique transversal with a proj} implies that there is some $Y_j$ such that
        \begin{align*}
            Q_j = \underline{g} \Delta_{X_j}^{k_j} A_{Y_j} \Delta_{X_j}^{-k_j}\underline{g}^{-1} \\
        \end{align*}

        As in the proof of Proposition \ref{prop:last-before-proj-def}, the element $\underline{g} \Delta_{X_j}^{k_j}$ is a standardizer for the bases of both $M_1$ and $M_2$. 
        

        By Corollary \ref{transversal k existence}, there are some $k_i$ and $Y_i$ such that for any $i \neq j$, 
        \begin{align*}
            Q_i = \underline{g} \Delta_{X_j}^{k_j} \Delta_{X'_i}^{k_i} A_{Y_i} \Delta_{X'_i}^{-k_i}\Delta_{X_j}^{-k_j}\underline{g}^{-1}\text{.} \tag{$\star\star\star$}
        \end{align*}

        Similarly, applying Corollary \ref{transversal k existence} to the collection $Q'_i$ yields $l_i$ and $Y'_i$ for each $i \neq j$ such that
        \begin{align*}
            Q'_i = \underline{g}\Delta_{X_j}^{k_j}\Delta_{X'_i}^{l_i} A_{Y'_i} \Delta_{X'_i}^{-l_i} \Delta_{X_j}^{-k_j} \underline{g}^{-1}\text{.} \tag{$\star\star\star\star$}
        \end{align*}

        By Lemma \ref{well behaved projections}, the forms for $Q_i$ and $Q'_i$ that we found in equations ($\star\star\star$) and ($\star\star\star\star$) imply that $\pi_{\{P_i\}}(Q_i) = k_i$ and $\tilde{\pi}_{\{P_i\}}(Q'_i) = l_i$. By the definition of a flip move, $|k_i - l_i | \leq 1$. 


        Consider the markings $x \cdot M_1$ and $x \cdot M_2$. The base collection of $x \cdot M_1$ is $\{x P_i x^{-1} \} = \{(x\underline{g}\Delta_{X_j}^{k_j}) A_{X'_i} (\Delta_{X_j}^{-k_j}\underline{g}^{-1}x^{-1})\}$, and the base collection of $M_2$ is the same for all $i \neq j$.

        For each $i \neq j$, the transverse elements of $x \cdot M_1$ and $x \cdot M_2$ are \begin{align*}
            x Q_i x^{-1} &= (x \underline{g}\Delta_{X_j}^{k_j}) \Delta_{X_i}^{k_i} A_{Y_i} \Delta_{X_i}^{-k_i}(\Delta_{X_j}^{-k_j}\underline{g}^{-1} x^{-1}) \text{ and}\\
            x Q'_i x^{-1} &= (x \underline{g} \Delta_{X_j}^{k_j}) \Delta_{X_i}^{l_i} A_{Y'_i} \Delta_{X_i}^{-l_i} (\Delta_{X_j}^{-k_j}\underline{g}^{-1} x^{-1})\text{.}
        \end{align*}

        The element $x \underline{g}\Delta_{X_j}^{k_j}$ is a standardizer for both the base $\{x P_i x^{-1}\}$ and the transverse element $x Q_j x^{-1}$, so by the definition of the projections $\pi_{\{P_i\}}$ and $\tilde{\pi}_{\{P_i\}}$, 
        \begin{align*}
            |\pi_{\{x P_i x^{-1}\}}(x Q_i x^{-1}) - \tilde{\pi}_{\{x P_i x^{-1}\}}(x Q'_i x^{-1})| &= |\pi_{\{x P_i x^{-1}\}}(x\underline{g}\Delta_{X_j}^{k_j}) + k_i \\
            &- (\pi_{\{x P_i x^{-1}\}}(x\underline{g}\Delta_{X_j}^{k_j}) + l_i)|\\
            &= |k_i - l_i|\\
            &\leq 1\text{.}
        \end{align*}
    \end{proof}

\section{A connected graph}\label{sec:geometric action}
We have now seen that an irreducible finite-type Artin group $A_{\Gamma}$ modulo its center acts by isometries on the locally finite graph $\mathcal{W}$. It then follows from Theorem \ref{vx stabilizer} that the action is proper. The action is cocompact because every marking is conjugate to a standard one; the entire graph is then generated by the union of the $1$-neighborhoods of the standard markings.

    To complete the proof of Theorem \ref{thmx:main}, it remains to verify that the graph $\mathcal{W}$ is connected. We proceed in two steps. First, we show that all standard markings are connected to one another. Then, we exhibit standard markings which are connected to their conjugates by an Artin generator. 
    
    

    

    We begin by showing that any two markings which have the same base and differ only by choosing different $Q_i$ and $Q'_i$ with $|\pi_{\{P_i\}}(Q_i) - \pi_{\{P_i\}} (Q'_i)| \leq 1$ are bounded distance apart in the marking graph. 
    
    \begin{lemma}\label{changing transversals via flips}
        Suppose the defining graph $\Gamma$ contains at least 3 vertices. Let $M_1 = \{(P_i , Q_i)\}$ and $M_2 = \{(P_i , R_i)\}$ be two markings such that $|\pi_{\{P_i\}}(Q_i) - \pi_{\{P_i\}}(R_i)| \leq 1$ for all $i$. Then $M_1$ and $M_2$ are distance at most 4 in the marking graph.
    \end{lemma}

    \begin{proof}
        When $\Gamma$ contains at least three vertices, a maximal $C_{parab}$-simplex contains at least two vertices by \cref{maximal simplex characterization}, so there are at least two elements in the collection $\{P_i\}$.
    
        Throughout the proof, $M'$, $M''$, and $M'''$ will denote some ``stand-in'' markings in which many of the precise choices of parabolic subgroups are irrelevant. Recall that in Lemma \ref{exists a transversal with a proj}, we saw that for any $n$ and any choice of base collection $\{P_i\}$, either there is a transverse element $Q_i$ for each $P_i$ with $\pi_{\{P_i\}}(Q_i) = n$ or $\pi_{\{P_i\}}(Q_i) = n-1$, depending on whether $n$ is a central power of the Garside element in a standardization of $P_i$, and Corollary \ref{cor:exists-trasnv-proj} extended this to $\tilde{\pi}_{\{P_i\}}$. 
        
        The same proof shows that we could instead chose either $n$ or $n+1$. In particular, if we perform a flip move across some index of $M_1$, then we can always choose transverse elements $Q'_i$ with both $|\pi_{\{P_i\}}(Q_i) - \tilde{\pi}_{\{P_i\}}(Q'_i)| \leq 1$ \emph{and} $|\pi_{\{P_i\}}(R_i) - \tilde{\pi}_{\{P_i\}}(Q'_i)| \leq 1$.
    
        Choose any index $j$. By Proposition \ref{marking existence}, there is some marking $M'$ whose collection of base elements is $\{P_1, \cdots, P_{j-1}, Q_j, P_{j+1}, \cdots\}$ and whose transverse elements are $P_j$ at index $j$ and any appropriate choices $Q'_i$ for $i \neq j$ which agree with both $Q_i$ and $R_i$. By the definition of a flip edge, any such $M'$ is connected to $M_1$ via a flip move.

        It is clear that $M'' = \{(P_1, R_1), \cdots (P_{j-1}, R_{j-1}), (P_j, Q_j), (P_{j+1}, R_{j+1}), \cdots\}$ is a marking, since $M_2$ was a marking and there are no requirements on how transverse elements relate to one another. It is also clear from the assumptions on $M_1$ and $M_2$ that $M''$ is connected to $M'$ via a flip move.


        Choose any $k \neq j$. As before, we know that there is a marking $M'''$ with base elements $\{P_1, \cdots, P_{k-1}, R_k, P_{k+1}, \cdots\}$ and any appropriate choices of transverse elements $Q_i$ such that $\tilde{\pi}_{\{P_i\}}(Q'_i)$ differs from each of $\pi_{\{P_i\}}(Q_i)$ and $\pi_{\{P_i\}}(R_i)$ by at most 1 when $i \neq j$. This $M'''$ is connected to both $M''$ and $M_2$ via flip moves, completing the proof. 

    \end{proof}

    The same result holds when $\Gamma$ is of dihedral type, although the proof differs. 

    \begin{lemma}\label{changing dihedral transversals via flips}
        Let $A_{\Gamma}$ be of dihedral type. Let $M_1 = \{(P_i , Q_i)\}$ and $M_2 = \{(P_i , R_i)\}$ be two markings such that $|\pi_{\{P_i\}}(Q_i) - \tilde{\pi}_{\{P_i\}}(R_i)| \leq 1$ for all $i$. Then $M_1$ and $M_2$ are distance at most 1 in the marking graph.
    \end{lemma}
    \begin{proof}
        Let $M_1$ and $M_2$ be as in the statement. By \cref{maximal simplex characterization}, the simplex $\{P_i\}$ is actually a single vertex $P = g \langle a \rangle g^{-1}$ where $a$ is an Artin generator. It is no loss of generality to assume $P$ is the standard parabolic subgroup $\langle a \rangle$. Applying \cref{transversal k existence} yields that there are some integers $k$ and $l$ such that $Q = a^k \langle b \rangle a^{-k}$ and $R = a^{l} \langle b \rangle a^{-l}$. These satisfy $|\pi_{\{P\}}(Q) - \tilde{\pi}_{\{P\}}(R)| \leq 1$ exactly when $|k-l| \leq 1$, in which case either $R = Q$ and $M_1 = M_2$ or the two markings are connected via a twist move.
    \end{proof}

    Next, we show that all markings involving only standard elements are bounded distance from each other in the marking graph.
    
    \begin{proposition}\label{flip base paths}
        There exists $k$ depending on $|V(\Gamma)|$ such that if $\{P_i\}$ and $\{P'_i\}$ are two collections of standard base elements, then there are markings $M_1$ and $M_2$ such that the following hold. \begin{enumerate}
            \item $\{P_i\}$ forms the base of $M_1$ and $\{P'_i\}$ forms the base of $M_2$.
            \item All transverse elements of $M_1$ and $M_2$ are standard.
            \item $M_1$ and $M_2$ are distance at most $2k$ from each other in the marking graph $\mathcal{W}$.
        \end{enumerate}
    \end{proposition}
    \begin{proof}
        We will show that every marking in which $\{P_i\}$ and $\{Q_i\}$ are all standard is distance at most $k$ from a specific one, $\overline{M}$. Namely, $\overline{M}$ is the marking where $P_i$ is generated by $\{s_1, \cdots, s_i\}$ for the labelings given in Figure \ref{fig:defgraphs}, which we have reprinted below for the reader's convenience, and where all transverse elements are the standard transverse elements constructed in Proposition \ref{marking existence}. Specifically, the transverse element associated to $P_i$ is $Q_i = s_{i+1}$ unless vertex $s_{i-1}$ has valence 3 and does not commute with $s_{i+1}$, in which case $Q_{i}$ is the standard parabolic subgroup generated by $\{s_1, \cdots s_{i-1}, s_{i+1}\}$. Vertex labels are below or to the left of the corresponding vertex, and edge labels are italicized.

        \begin{figure}[h]
            \centering
            \includegraphics[width=10cm]{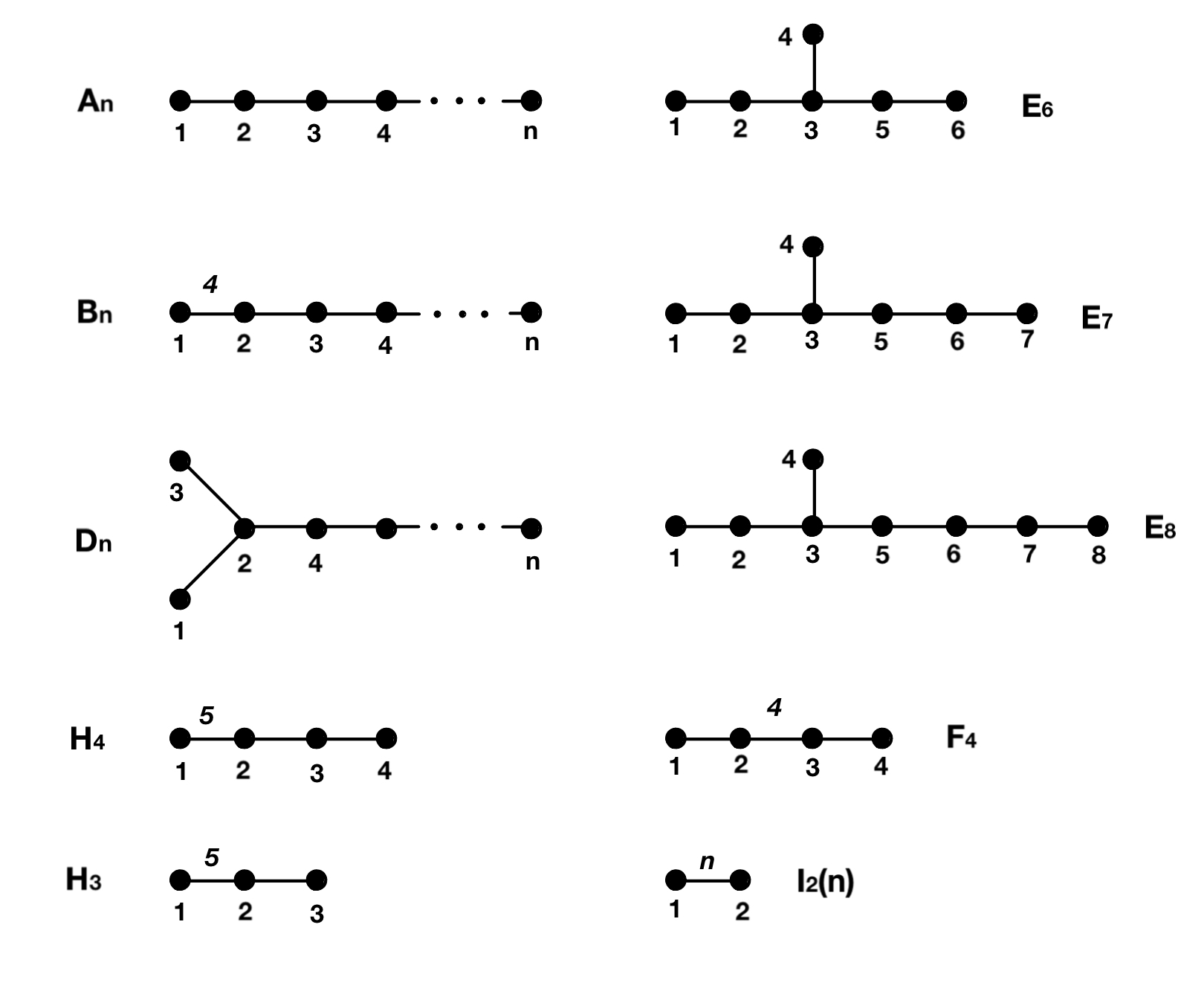}
        \end{figure}
        
        The proof is by induction on the number of vertices in the defining graph $A_{\Gamma}$.

        \textbf{Base case:} Suppose $\Gamma$ has two vertices.  This occurs only if $\Gamma$ is of type $I_2(n)$. There are only two proper standard parabolic subgroups: the two generators. Thus there are only two markings where all elements are standard: one where $s_1$ is the base and $s_2$ is the transverse element, and the reverse. These are clearly connected via a single flip move. 

        \begin{figure}[h!]
            \centering
            \includegraphics[width=8cm]{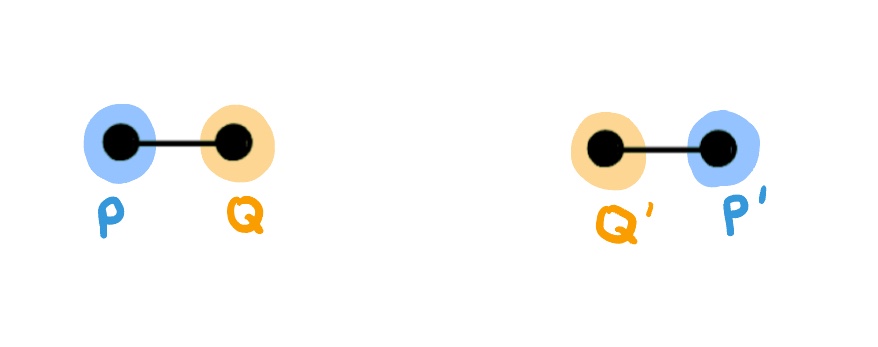}
            \captionsetup{margin=.5cm,justification=centering}
            \caption{The two choices of marking with standard base and transverse elements.}
        \end{figure}

        \textbf{Induction step:} Suppose that all markings on $A_{\Lambda}$ consisting only of standard parabolic subgroups are $k$ flip moves away from $\overline{M}$ when $A_{\Lambda}$ is an irreducible finite-type Artin group and $\Lambda$ has at most $N-1$ vertices. 
        
        Let $\Gamma$ be a graph with $N$ vertices, and let $M$ be a marking on $A_{\Gamma}$ whose base and transverse elements are all standard. By Lemma \ref{lem:std-transv-uniqueness}, the transverse elments must be precisely as constructed in \cref{marking existence}. Let $\{A_{X_i}\}$ denote the base elements and $\{A_{Y_i}\}$ denote the transverse elements.
        
        Consider the maximal components $A_{X_1}$, $A_{X_2}$, and $A_{X_3}$, up to two of which may be empty. We recall three key facts. \begin{itemize}
            \item There is a unique vertex $v$ of $\Gamma$ such that $X_1 \cup X_2 \cup X_3 \cup \{v\} = V(\Gamma)$ by Lemma \ref{3 maxl base components}.
            \item If $X_j \subset X_i$ for any $i$ and $j$, then $Y_j \subset X_i$ by Lemma \ref{lower level transversal inclusion}.
            \item The transverse elements $A_{Y_1}$, $A_{Y_2}$, and $A_{Y_3}$ for each of the maximal elements $A_{X_1}$, $A_{X_2}$, and $A_{X_3}$ are defined as follows. If $A_{X_2 \cup \{v\} \cup X_3}$ commutes with all $A_{X_j} \leq A_{X_1}$, then $Y_1 = X_2 \cup \{v\} \cup X_3$. If it does not, then there is a unique element $A_{X_j}$ of the base such that $A_{X_j}$ is maximal inside $A_{X_1}$ and some vertex of $A_{X_j}$ is connected by an edge to $v$. In this case, we defined $Y_1 = X_2 \cup \{v\} \cup X_3 \cup X_j$.

        \begin{figure}[h!]
            \centering
            \includegraphics[width=8cm]{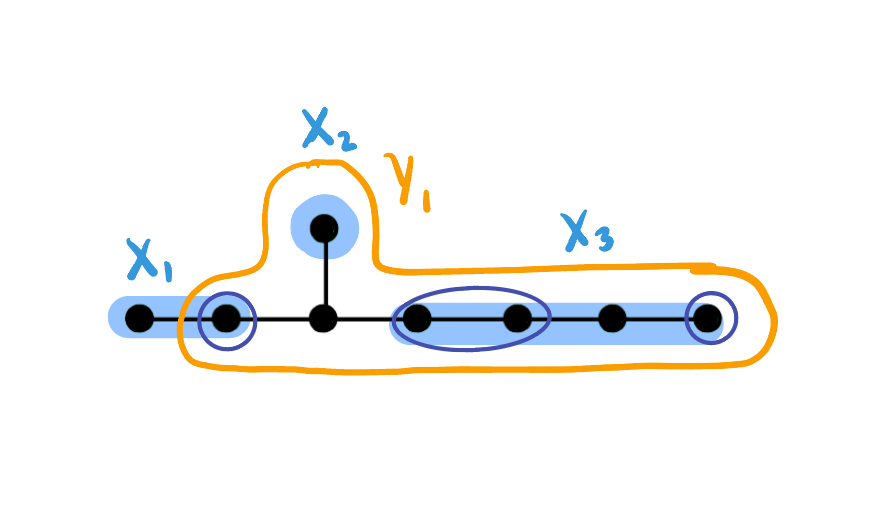}
            \captionsetup{margin=.5cm,justification=centering}
            \caption{An example of the construction of $Y_1$. Maximal base components are shaded in blue, and level 2 base components are circled in dark blue.}
        \end{figure}
        \end{itemize}

        First, suppose that there are at least two non-empty maximal components. The unique vertex $v \in \Gamma$ which is not contained in any maximal component must have valence at least 2, so it cannot be $s_1$. Let $v = s_j$, and let $X_1$ denote the component containing $s_1$. Notice that if all three maximal components are nonempty, one of them must be a single vertex. In this case, label the $A_{X_i}$ so that $X_3$ is the single vertex labeled $3$ or $4$, depending on whether $\Gamma$ is of type $D_n$ versus $E_6$, $E_7$, or $E_8$. 

        The $Y_2$ of \cref{marking existence} contains $X_1$, $s_j$, and $X_3$ if it is nonempty. We perform a flip move across $(X_2 , Y_2)$. The resulting marking has one maximal component which is an irreducible standard parabolic subgroup which contains $s_1$ and is strictly larger than our original $X_1$, and it has up to two additional maximal components. Notice that because $Y_2$ also contains $s_j$ and $X_3$, performing this flip move results in a simplex whose maximal level is missing a vertex of valence $1$ only if that missing vertex is $s_N$. In particular, after performing this flip, we are either still in the setting with at least two non-empty maximal components or the missing vertex is $s_N$.
        
        
        
         Repeating this process at most $N-2$ additional times then yields a marking $\{(A_{X'_i}\}, \{A_{Y'_i})\}$ consisting of all standard parabolic subgroups whose base has exactly one maximal component: the standard parabolic subgroup $P_M = \langle s_1, \cdots s_{N-1} \rangle$.

         \begin{figure}[h]
            \centering
            \includegraphics[width=8cm]{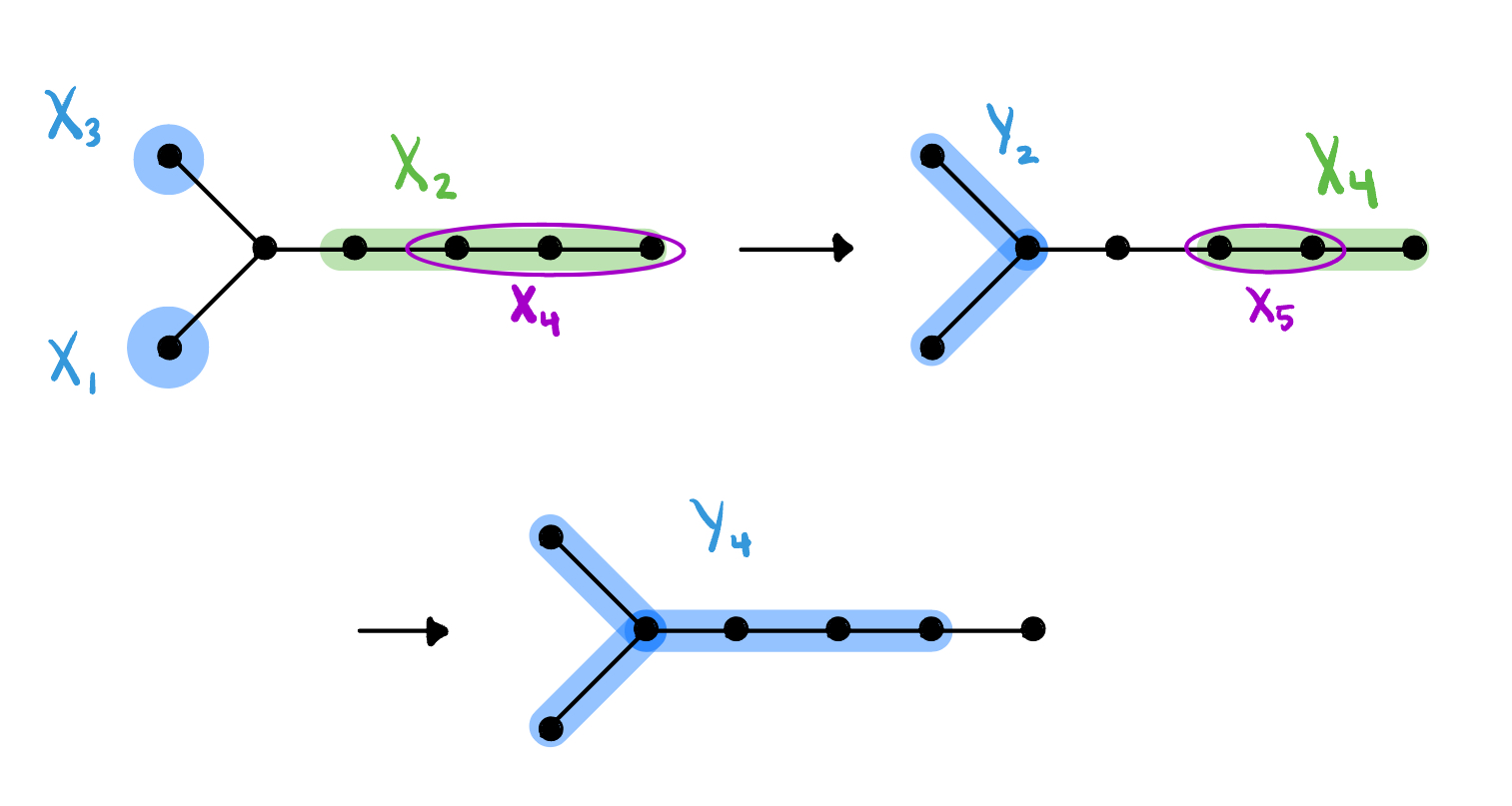}
            \captionsetup{margin=.5cm,justification=centering}
            \caption{An example which requires two flips}
        \end{figure}

        The pairs $(A_{X'_i} , A_{Y'_i})$ which are not maximal form a submarking on $P_M$ (viewed as its own Artin group). It has one fewer generator than $\Gamma$. Notice that performing flip moves inside of this parabolic subgroup will not affect the position of the exterior maximal base element (though it may change the associated transverse element). By the induction hypothesis, performing at most $k$ additional flip moves in this smaller maximal subgroup and choosing standard transverse elements each time results in a marking with the correct set of base curves and all standard transverse elements. By Lemma \ref{lem:std-transv-uniqueness}, this marking is precisely $\overline{M}$. 

        Now, suppose instead that there is a unique maximal component, i.e., the missing vertex $v$ has valence $1$. If $v = s_N$, then we conclude by the induction hypothesis precisely as above.

        If $s_1$ is the missing generator $v$ such that $X \cup \{s_1\} = \Gamma$, let $\Sigma$ be any maximal base simplex inside of $A_{X}$ such that the missing vertex is $s_{N}$ in our labeling for $\Gamma$. Since all lower levels of the marking are contained entirely in $A_{X}$, we can perform at most $2k$ flip moves across lower levels $A_{X_i}$ to reach a marking where $A_{X}$ is the unique maximal component, where the higher levels are $\Sigma$, and where all transverse elements are standard.

        By Lemma \ref{lem:std-transv-uniqueness} and our choice of $\Sigma$, the transverse element $A_Y$ for $A_X$ is precisely $\langle s_1 , \cdots , s_{N-1}\rangle$. 
        

        We can now perform a flip move across $A_{X}$ and $A_{Y}$, again choosing standard transverse elements, to obtain an all standard marking whose unique maximal component is the same as the unique maximal component of $\overline{M}$. By the same reasoning as in the prior case, this marking is at most $k$ additional flip moves away from the desired one.

        If $v = s_3$ in $D_n$ or $s_4$ in $E_n$, we can argue similarly; here, we again choose $\Sigma$ inside of $A_X$ so that the missing vertex in the first level of $\Sigma$ is $s_{N}$, and we apply the induction hypothesis to obtain a marking whose base collection is $\Sigma \cup A_{X}$ and where the transverse element for $A_{X}$ is $A_Y = \langle s_1 , \cdots , s_{N-1} \rangle$ as in \cref{marking existence}.  

        Thus the desired marking can always be obtained after $\text{max}\{3k+1, N + k -1\}$ flip moves.
    \end{proof}

    The $k$ constructed above may not be optimal, but any finite $k$ suffices for our purposes. We can now proceed to the final step: demonstrating that for any Artin generator $a$, there is a standard marking $M$ such that $M$ and $a M a^{-1}$ are adjacent in $\mathcal{W}$.


    \begin{lemma}\label{lem:simplex-with-gen-inside}
        Let $a$ be an Artin generator in $A_{\Gamma}$. There is some marking $M$ containing only standard parabolic subgroups such that $\langle a \rangle$ is in the base of $M$. 
    \end{lemma}
    \begin{proof}
        It suffices to exhibit a standard maximal $C_{parab}$-simplex $\Sigma$ containing $\langle a \rangle$; by \cref{marking existence}, any such simplex forms the base of some standard marking.

        We begin with the parabolic subgroup $\langle a \rangle$ and add subgroups to form $\Sigma$. If $|V(\Gamma)| = 2$, then $\{\langle a \rangle\}$ is a maximal $C_{parab}$-simplex, and we are done. Otherwise, let $x_1$ be a vertex in $\text{Lk}(a)$ in $\Gamma$, and define $X_1 = a \cup x_1$. The subgroup $\langle a \rangle$ is contained in $A_{X_1}$, the subgroup $A_{X_1}$ is irreducible by choice of $x_1$, and $A_{X_1}$ is proper because $|V(\Gamma)| \geq 3$.

        If $\Gamma - X_1$ consists of a single vertex, then the simplex $\Sigma$ with vertex set $\{ \langle a \rangle, A_{X_1}\}$ satisfies the conditions of \cref{maximal simplex characterization} and is thus maximal. Otherwise, we choose a vertex $x_2$ in $\text{Lk}(X_1)$, define $X_2 = x_2 \cup X_1$, and include $A_{X_2}$ in $\Sigma$. Repeating this argument finitely many times necessarily yields a maximal simplex by \cref{maximal simplex characterization}.
    \end{proof}
    \begin{corollary}\label{cor:a-connected}
        For each Artin generator $a$, there is a marking $M$ containing only standard parabolic subgroups such that $M$ and $a M a^{-1}$ are connected in $\mathcal{W}$.
    \end{corollary}
    \begin{proof}
        Let $M$ be the marking which contains $A = \langle a \rangle$ in its base provided by \cref{lem:simplex-with-gen-inside}. Since $z_A = a$, the markings $M$ and $a M a^{-1}$ are connected via a twist move.
    \end{proof}


    \begin{corollary}
        The marking graph $\mathcal{W}$ is connected.
    \end{corollary}
    \begin{proof}
        We have seen that entire marking graph is covered by the action of $A_{\Gamma}/Z(A_{\Gamma})$ on the $1$-neighborhood of the set $K$ of markings where all subgroups are standard. The set $K$ is connected to $a \cdot K$ for each Artin generator $a$ by \cref{cor:a-connected}. Because the action is by isometries, this also implies that $a g \cdot K$ is connected to $g \cdot K$ for every $g$. Induction on the word length of $g$ completes the proof.




    \end{proof}

\bibliographystyle{acm}
\bibliography{refs.bib} 

\end{document}

%% file: top.tex
\newcounter{commentcounter}

\usepackage{amsmath} 
\usepackage{amssymb} 
\usepackage{amsthm} 
\usepackage{stmaryrd} 
\usepackage[english]{babel} 
\usepackage[font=small,justification=centering]{caption} 
\usepackage[nodayofweek]{datetime}
\usepackage[shortlabels]{enumitem} 
\usepackage[T1]{fontenc} 
\usepackage[utf8]{inputenc} 
\usepackage{ifthen} 
\usepackage{mathabx} 
\usepackage{mathtools} 
\usepackage[dvipsnames]{xcolor} 
\usepackage[pdftex,  colorlinks=true]{hyperref} 
    \hypersetup{urlcolor=RoyalBlue, linkcolor=RoyalBlue,  citecolor=black}
\usepackage{setspace} 
\usepackage{tikz-cd} 
\usepackage{xfrac} 
\usepackage[capitalize]{cleveref} 

\usepackage{wrapfig}
\usepackage{caption}
\usepackage{subcaption}

\makeatletter
\@namedef{subjclassname@1991}{Mathematical subject classification 1991}
\@namedef{subjclassname@2000}{Mathematical subject classification 2000}
\@namedef{subjclassname@2010}{Mathematical subject classification 2010}
\@namedef{subjclassname@2020}{Mathematical subject classification 2020}
\makeatother

\newtheorem{theorem}{Theorem}[section]
\newtheorem{lemma}[theorem]{Lemma}
\newtheorem{corollary}[theorem]{Corollary}
\newtheorem{proposition}[theorem]{Proposition}

\newtheorem{thmx}{Theorem}

\newtheorem{propx}[thmx]{Proposition}

\theoremstyle{definition}
\newtheorem{definition}[theorem]{Definition}

\newtheorem{example}[theorem]{Example}

\theoremstyle{plain}

    \newtheoremstyle{TheoremNum}
        {\topsep}{\topsep} 
        {\itshape} 
        {-0.25cm} 
        {\bfseries} 
        {.} 
        { }  
        {\thmname{#1}\thmnote{ \bfseries #3}}
    \theoremstyle{TheoremNum}

\newcommand*{\claimproofname}{My proof}

\usepackage{tikz}
\usetikzlibrary{arrows,quotes}
\tikzstyle{blackNode}=[fill=black, draw=black, shape=circle]